\documentclass[11pt]{article}
\usepackage{amssymb}
\usepackage{caption}
\usepackage{graphicx}
\usepackage{amsmath}
\usepackage{float}
\usepackage{geometry}   
\usepackage{fullpage}
\usepackage{fancyhdr}   
\usepackage{listings} 
\usepackage{xcolor} 
\usepackage{amsthm}  
\usepackage{algorithm}
\usepackage{algpseudocode}
\usepackage[colorlinks,
linkcolor=blue,
anchorcolor=blue, 
citecolor=blue]{hyperref}
\usepackage{comment}
\usepackage[round]{natbib}
\usepackage{nicematrix}
\usepackage{dsfont}
\usepackage{enumitem}
\usepackage{authblk}
\usepackage{tikz}
\usetikzlibrary{intersections}
\usepackage{subcaption}  

\newtheorem{theorem}{Theorem}
\newtheorem{lemma}{Lemma}

\newcommand\fro[1]{\| #1 \|_{\rm{F}}}
\newcommand\op[1]{\| #1 \|}

\newcommand{\mat}[1]{\begin{bmatrix}#1 \\ \end{bmatrix}}
\newcommand{\inp}[2]{\langle #1,#2\rangle}
\newcommand\ltwo[1]{\| #1 \|_{\ell_2}}
\newcommand\psitwo[1]{\| #1 \|_{\psi_2}}

\newcommand\lzero[1]{\| #1 \|_{\ell_0}}
\let\originalparagraph\paragraph
\renewcommand{\paragraph}[2][.]{\originalparagraph{\textit{#2#1}}}

\usepackage{pifont}

\def\calC{{\mathcal C}}
\def\calD{{\mathcal D}}
\def\calE{{\mathcal E}}
\def\calF{{\mathcal F}}

\def\calI{{\mathcal I}}
\def\calJ{{\mathcal J}}

\def\calL{{\mathcal L}}
\def\calM{{\mathcal M}}

\def\calQ{{\mathcal Q}}

\def\calS{{\mathcal S}}
\def\calT{{\mathcal T}}
\def\calU{{\mathcal U}}

\newcommand{\mcM}{\mathcal{M}}

\newcommand{\mcU}{\mathcal{U}}

\newcommand{\mcW}{\mathcal{W}}

\def\EE{{\mathbb E}}

\def\II{{\mathbb I}}

\def\OO{{\mathbb O}}
\def\PP{{\mathbb P}}

\def\RR{{\mathbb R}}
\def\SS{{\mathbb S}}

\def\ZZ{{\mathbb Z}}

\def\g{{\boldsymbol g}}

\def\j{{\boldsymbol j}}

\def\n{{\boldsymbol n}}

\def\s{{\boldsymbol s}}

\def\v{{\boldsymbol v}}
\def\w{{\boldsymbol w}}

\def\y{{\boldsymbol y}}

\def\A{A}
\def\B{B}

\def\D{D}
\def\E{E}

\def\G{G}
\def\H{H}
\def\I{I}

\def\M{M}

\def\U{U}
\def\V{V}

\def\Z{Z}

\def\svd{\textsf{svd}}
\def\vec{\textsf{vec}}

\def\diag{\textsf{diag}}

\def\tr{\textsf{tr}}

\def\bLa{\Lambda}

\def\bDel{\Delta}
\def\bSigma{\Sigma}
\def\bXi{\Xi}
\def\bTheta{\Theta}

\def\beps{\boldsymbol{\varepsilon}}
\def\bdelta{\boldsymbol{\delta}}
\def\eps{\varepsilon}

\def\hat{\widehat}
\def\tilde{\widetilde}
\def\epsilon{\eps}

\def\whU{\widehat{U}}
\def\whSig{\widehat{\Sigma}}

\def\alg{\textsf{Fed-DP-PCA}}

\begin{document}
 
\title{Federated PCA and Estimation for Spiked Covariance Matrices: Optimal Rates and Efficient Algorithm}
\author{Jingyang Li$^1$, T. Tony Cai$^2$, Dong Xia$^3$ and Anru R. Zhang$^4$\\
$^1$ Department of Statistics and Department of Mathematics, University of Michigan Ann Arbor\\
$^2$ Department of Statistics and Data Science, University of Pennsylvania\\
$^3$ Department of Mathematics, Hong Kong University of Science and Technology\\
$^4$ Department of Biostatistics \& Bioinformatics and Department of Computer Science, Duke University}
\date{(\today)}

\maketitle

\begin{abstract}
Federated Learning (FL) has gained significant recent attention in machine learning for its enhanced privacy and data security, making it indispensable in fields such as healthcare, finance, and personalized services. This paper investigates federated PCA and estimation for spiked covariance matrices under distributed differential privacy constraints.

We establish minimax rates of convergence, with a key finding that the central server's optimal rate is the harmonic mean of the local clients' minimax rates. This guarantees consistent estimation at the central server as long as at least one local client provides consistent results. Notably, consistency is maintained even if some local estimators are inconsistent, provided there are enough clients. These findings highlight the robustness and scalability of FL for reliable statistical inference under privacy constraints.

To establish minimax lower bounds, we derive a matrix version of van Trees' inequality, which is of independent interest. Furthermore, we propose an efficient algorithm that preserves differential privacy while achieving near-optimal rates at the central server, up to a logarithmic factor.  We address significant technical challenges in analyzing this algorithm, which involves a three-layer spectral decomposition. Numerical performance of the proposed algorithm is investigated using both simulated and real data.

\end{abstract}

\section{Introduction}



Principal Component Analysis (PCA) and its variants are fundamental tools in statistics and machine learning, particularly valuable for dimension reduction and data visualization when high-dimensional data lie in a low-dimensional space. PCA has been widely applied in data denoising and compression, feature extraction, clustering analysis, factor models, correlation analysis, and regression analysis. Population principal components are typically defined using the leading eigenvectors of population covariance matrices. The estimation and inference of these components from sample data have been extensively studied across various fields, including probability, statistics, and machine learning (see, for example, \cite{vershynin2018high, jolliffe2016principal, silverstein1995empirical, bickel2008covariance, koltchinskii2016asymptotics, benaych2011eigenvalues, cai2013sparse, cai2015optimal, zhang2022heteroskedastic, cai2016minimax}). See \cite{cai2016estimating} for a survey on optimal estimation of high-dimensional covariance structures.

With the digital shift in human activities, such as social networking, online shopping, and healthcare, vast amounts of personal information are collected and analyzed by large information technology firms and governmental organizations. The centralization of data storage raises critical concerns about the misuse of sensitive personal information, whether through intentional abuse or unintentional leaks. Traditional privacy-preserving methods like anonymization have proven insufficient, particularly in the context of classical PCA. As shown by \cite{dwork2006calibrating} and \cite{chaudhuri2013near}, classical PCA is vulnerable to alterations in individual data points and poses a significant risk of personal information leakage.

Differential privacy (DP) provides a robust framework to ensure that published statistics do not reveal whether any individual's data was included in the dataset. Initially introduced by \citet{dwork2006calibrating}, DP has become a widely accepted standard in both industrial and governmental applications \citep{google_privacy,ding2017collecting,apple2017,abowd2016challenge,abowd2020modernization}. DP is typically achieved by adding random noise to statistical outputs, using mechanisms such as the Gaussian or Laplace mechanisms. However, this randomization can compromise the accuracy of statistical methods. Consequently, a growing body of literature explores the trade-offs between privacy and accuracy in fundamental statistical and machine learning problems. The minimax optimal rates for differentially private PCA and covariance matrix estimation under the spiked model are established in \cite{cai2024optimalPCA}.

Federated Learning (FL) is a decentralized machine learning framework where local clients train their models and communicate with a central server without sharing raw data \citep{mcmahan2017communication}. Instead, clients privatize their learned models and share them with the central server or other clients, enabling collaborative machine learning while maintaining data privacy. Federated learning has applications in healthcare, finance, Internet of Things (IoT), and more. However, the heterogeneity of datasets, privacy constraints, and the increasing number of local clients pose significant challenges to understanding the theoretical performance of federated learning. Under the DP constraint, the special case where each client holds only one datum is referred to as the \emph{local} differential privacy setting \citep{duchi2013local}.

This paper investigates the minimax optimal rates in federated PCA under the spiked covariance model, considering diverse privacy constraints and sample sizes at local clients. A surprising and significant finding is that the minimax optimal rates achieved by the central server are the (scaled) harmonic mean of the minimax optimal rates achieved by local clients. This indicates that federated learning is multiply robust, meaning the central server attains a consistent estimator as long as at least one local client provides a consistent estimator. We believe this phenomenon is general and applies to many other federated learning problems under DP constraints. 

The lower bound is established by leveraging a matrix version of the van Trees' inequality, inspired by a similar strategy in \cite{cai2024optimal}. This matrix version of van Trees' inequality is of independent interest. Additionally, we develop a computationally efficient algorithm that preserves DP at local clients and achieves the minimax optimal rate at the central server (up to logarithmic factors). The final estimator is obtained by applying three layers of spectral decomposition, posing significant technical challenges in deriving the sharp upper bound.

When there is only one local client, federated PCA simplifies to DP-PCA, and the upper bound derived in this study aligns with the results presented in \cite{cai2024optimalPCA}. However, we emphasize that the technical contributions of these two works are fundamentally distinct. The primary contribution of \cite{cai2024optimalPCA} lies in the precise characterization of the sensitivity of empirical spectral projectors and eigenvalues under the spiked covariance model, which serves as the foundation for our methodology and theoretical framework for Federated PCA presented in this paper. Specifically, their results are directly leveraged to determine the appropriate level of artificial noise to be added at each local client. In contrast, the technical challenges in Federated PCA stem from the need to perform a sharp analysis of aggregated PCA across multiple clients. Our proposed method involves not one, but \emph{three} layers of spectral decomposition, and the precise perturbation analysis of the final estimator relies on an explicit characterization of both the stochastic error and the artificial noise introduced in the first and second layers, respectively. Each layer's spectral decomposition is represented by a Neumann series expansion, leading to a composition of three Neumann series. Consequently, we had to develop a unified strategy to derive concentration bounds for numerous higher-order perturbation terms,  which required new techniques beyond those used in single-client DP-PCA. See the proof sketch of Theorem~\ref{thm:highprob:upperbound} for more details on our approach. 
 
\subsection{Problem formulation}

The spiked covariance model has been widely applied and extensively investigated for extracting low-dimensional covariance structure from potentially high-dimensional data. It has found applications in diverse fields such as genomics \citep{leek2007capturing}, wireless communication \citep{telatar1999capacity}, asset pricing \citep{chamberlain1982arbitrage}, econometrics \citep{fan2008high}, and population genetics \citep{patterson2006population,novembre2008interpreting}. Under the spiked model, the covariance matrix $\Sigma$ is a low-rank deformation of the (scaled) identity matrix, which admits the following decomposition:
\begin{align}\label{eq:spiked-model}
\Sigma=U\Lambda U^{\top}+\sigma^2 I_p,
\end{align}
where $U=(u_1,\cdots,u_r)\in\OO^{p\times r}$ and $\Lambda={\rm diag}(\lambda_1,\cdots,\lambda_r)$ are the leading eigenvectors and eigenvalues of the low-rank deformation with $\lambda_1\geq \cdots\geq \lambda_r>0$. Here $I_p$ represents the $p\times p$ identity matrix and $\OO^{p\times r}$ is the set of $p\times r$ matrices satisfying $U^{\top}U=I_r$.  
 
Estimating the population covariance matrix and its leading eigenvectors from a random sample $X=(X_1,\cdots,X_n)\in\RR^{p\times n}$, where the column vectors are i.i.d. with ${\rm cov}(X_i)=\Sigma$, is a fundamental problem in multivariate statistical analysis. The spiked structure often significantly facilitates the theoretical derivation of the distribution of the sample eigenvectors and eigenvalues. Minimax optimal PCA and covariance matrix estimation have been extensively studied under the spiked model.  An incomplete list of representative work includes \cite{nadler2008finite, donoho2018optimal, cai2010optimal,  cai2016estimating, koltchinskii2017concentration, johnstone2001distribution, fan2008high,paul2007asymptotics} and references therein. 

Differential privacy \citep{dwork2006calibrating} is a framework designed to provide privacy guarantee when analyzing and sharing data.  Let $X\in\RR^{p\times n}$ be a data set consisting of $n$ observations.  In standard definitions, a matrix $X'\in\RR^{p\times n}$ is called a neighboring data set of $X$ if and only if $X$ and $X'$ differ by only one datum,  i.e.,  one column of $X$ is replaced by some other, possibly arbitrary, observation of the same dimension. In the context of PCA, since the observations in $X$ are independently sampled from the same distribution, a neighboring data set $X'$ is obtained by replacing one datum in $X$ with an independent copy. For a given data set $X$ and any $\eps>0,  \delta\in[0, 1)$,  a randomized algorithm $A$ that maps $X$ into $\RR^{d_1\times d_2}$ is called $(\eps,  \delta)$-differentially private ($(\eps, \delta)$-DP) over the data set $X$ if 
$$
\PP\big(A(X)\in\calQ\big)\leq e^{\eps}\cdot \PP\big(A(X')\in\calQ\big)+\delta,
$$
for all measurable subset $\calQ\subset \RR^{d_1\times d_2}$ and all neighboring data set $X'$.  

Differentially private PCA algorithms have been proposed and investigated in \cite{blum2005practical,chaudhuri2011differentially, dwork2014analyze} by treating each datum $X_i$ as a fixed vector.  More recently,  \cite{liu2022dp} and \cite{cai2024optimalPCA} studied the minimax optimal rates for differentially private PCA and covariance estimation under the spiked covariance model (\ref{eq:spiked-model}).  In particular,  \cite{cai2024optimalPCA} showed that the minimax optimal rates,  up to logarithmic factors,  are
\begin{equation}\label{eq:dp-opt-spiked-cov}
	\begin{aligned}
		\inf_{\whU\in\mcU_{\varepsilon,\delta}}\ \sup_{\Sigma\in\Theta (\lambda,\sigma^2)} \ \EE \|\whU\whU^{\top}-UU^{\top}\|_{\rm F}^2& \asymp \Psi_{0}^2(n, \eps, \delta):=\bigg(\frac{\sigma^4}{\lambda^2}+\frac{\sigma^2}{\lambda}\bigg) \left(\frac{pr}{n} +\frac{p^2 r^2}{n^2\epsilon^2}\right);\\
		\inf_{\whSig\in\mcM_{\varepsilon,\delta}}\ \sup_{\Sigma\in\Theta (\lambda,\sigma^2)}\ \EE \|\whSig-\Sigma\|_{\rm F}^2& \asymp  \lambda^2\cdot \Psi_1(n,\eps,\delta)+\lambda^2\cdot \Psi_0(n, \eps, \delta)\\
		&:= \lambda^2 \left(\frac{r^2}{n} + \frac{r^4}{n^2\varepsilon^2}\right) + \sigma^2(\lambda+\sigma^2) \left(\frac{pr}{n} + \frac{p^2r^2}{n^2 \varepsilon^2} \right),
	\end{aligned}
\end{equation} 
conditioned on $\max\{\Psi_0(n,\eps,\delta), \Psi_{1}(n, \eps, \delta)\}\leq \sqrt{r}$ (otherwise,  a trivial estimator suffices) and under certain constraint on $\delta$. The parameter set $\Theta(\lambda,\sigma^2)$ is defined by 
\begin{align*}
	\Theta(\lambda,\sigma^2): = \bigg\{\Sigma = &U\Lambda U^\top + \sigma^2 I: U\in\OO^{p\times r}, \Lambda = \diag(\lambda_1,\cdots, \lambda_r), c_0\lambda\leq \lambda_r\leq\cdots\leq \lambda_1\leq C_0\lambda\bigg\},
\end{align*}
with universal constants $c_0, C_0>0$.   Here $\calU_{\eps,\delta}$ and $\calM_{\eps,\delta}$ represent the collection of all $(\eps,\delta)$-DP estimators of $U$ and $\Sigma$,  respectively.  The terms in (\ref{eq:dp-opt-spiked-cov}) involving $\eps$ reflect the cost of privacy.  The error bound of $\EE\|\whSig-\Sigma\|_{\rm F}^2$ consists of two terms,  where $\lambda^2\cdot \Psi_0^2(n,\eps,\delta)$ and $\lambda^2\cdot \Psi_1^2(n,\eps,\delta)$ are mainly contributed from estimating the eigenvectors and eigenvalues,  respectively.  

We formulate the problem of differentially private federated PCA as follows. There are $m$ local clients,  where the $j$-th client holds data $\calD_j:=\big\{X^{(j)}_i\in\RR^p, i=1,\cdots, n_j\big\}$ for each $j\in [m]$.  Under the spiked model, we assume that $X_i^{(j)}\stackrel{{\rm i.i.d.}}{\sim} N(0, \Sigma)$ for all $ j\in[m]$ and for all $i\in[n_j]$.  Here $n_j$ represents the sample size in the $j$-th local client.   
All the data share a common spiked covariance matrix, and we assume zero mean and Gaussian distribution for simplicity.  There is a central server that can communicate with the local clients,  whose goal is to estimate the underlying covariance matrix $\Sigma$ and its principal components.  Local clients have privacy constraints and cannot share raw data with the central server or other local clients.  Let $\eps_j>0$ and $\delta_j\in[0,1),  j\in[m]$ be two given sequences representing the privacy budgets on all local clients.  Basically,  the $j$-th local client requires to achieve the $(\eps_j,  \delta_j)$-differential privacy when communicating its local information to the central server.  Denote $\beps:=(\eps_1,\cdots,\eps_m)^{\top}$, $\bdelta:=(\delta_1,\cdots,\delta_m)^{\top}$, and $\n=(n_1,\cdots, n_m)^{\top}$.   An estimator is called {\it federated} $(\beps,  \bdelta)$-DP if it is $(\eps_j, \delta_j)$-DP for data in the $j$-th local client for all $j\in[m]$.  In this paper,  we aim to investigate the minimax optimal rates for federated $(\beps, \bdelta)$-DP PCA and covariance matrix estimation under the spiked model.  We also propose computationally and communication-efficient estimators that are federated  $(\beps, \bdelta)$-DP and achieve the minimax optimality. By slightly abuse of notations,  we denote $\mcU_{\n, \beps,\bdelta}$ and $\mcM_{\n, \beps, \bdelta}$ as the collection of all federated  $(\beps, \bdelta)$-DP estimators of $U$ and $\Sigma$,  respectively, when the sample sizes in local clients are represented by $\n$.

\subsection{Main contribution}

In this paper,  we establish the minimax optimal rates for federated PCA and covariance matrix estimation under the spiked model with distributed DP constraints.  Recall the definitions of $\Phi_0(n_j,\eps_j,\delta_j)$ and $\Phi_1(n_j,\eps_j,\delta_j)$ in (\ref{eq:dp-opt-spiked-cov}).   Under mild conditions, these rates, up to logarithmic factors,  are 
\begin{equation}\label{eq:fed-dp-opt-spiked-pca}
\begin{aligned}
		\inf_{\whU\in\mcU_{\n, \beps,\bdelta}}\ \sup_{\Sigma\in\Theta (\lambda,\sigma^2)} \ \EE \|\whU\whU^{\top}&-UU^{\top}\|_{\rm F}^2 \asymp \frac{1}{\sum_{j=1}^m \Psi_0^{-2}(n_j, \eps_j, \delta_j)}\bigwedge r\\
		&\asymp \left(\left(\frac{\sigma^2}{\lambda}+\frac{\sigma^4}{\lambda^2}\right)\frac{1}{\sum_{j=1}^m \big(\frac{n_j}{rp}\wedge \frac{n_j^2\eps_j^2}{r^2p^2}\big)}\right)\bigwedge r,
\end{aligned}
\end{equation}
and 
\begin{equation}\label{eq:fed-dp-opt-spiked-cov}
\begin{aligned}
		\inf_{\whSig\in\mcM_{\n, \beps,\bdelta}}\ \sup_{\Sigma\in\Theta (\lambda,\sigma^2)} \ \EE \|\whSig-&\Sigma\|_{\rm F}^2 \asymp \left(\frac{\lambda^2}{\sum_{j=1}^m \Psi_0^{-2}(n_j, \eps_j, \delta_j)}+\frac{\lambda^2}{\sum_{j=1}^m \Psi_1^{-2}(n_j, \eps_j, \delta_j)}\right)\bigwedge (r\lambda^2)\\
		&\asymp \left(\frac{(\lambda\sigma^2+\sigma^4)}{\sum_{j=1}^m \big(\frac{n_j}{rp} \wedge \frac{n_j^2\eps_j^2}{r^2p^2}\big)} +\frac{\lambda^2}{\sum_{j=1}^m \big(\frac{n_j}{r^2}\wedge \frac{n_j^2\eps_j^2}{r^4}\big)}\right) \bigwedge (r\lambda^2).
\end{aligned}
\end{equation}
The bounds in (\ref{eq:fed-dp-opt-spiked-pca}) and (\ref{eq:fed-dp-opt-spiked-cov}) show that the minimax optimal rates achievable by the central server are proportional to the \emph{harmonic mean} of the minimax optimal rates achievable by local clients. By the harmonic mean inequality\footnote{ Harmonic mean inequality: $\frac{m}{\sum_{i=1}^m a_i^{-1}}\leq \frac{\sum_{i=1}^m a_i}{m}$ and the fact: $\sum_{i=1}^m a_{i}^{-1}\geq \sum_{i=1}^{\lceil m/2 \rceil} a_i^{-1}\geq (m/2)\cdot\textsf{med}\{a_i^{-1}\}_{i=1}^m$ for positive numbers $0<a_1\leq a_2\leq \cdots\leq a_m$.}, we get 
\begin{align*}
\frac{1}{\sum_{j=1}^m \Psi_0^{-2}(n_j, \eps_j, \delta_j)}\leq \min\bigg\{\min_{j\in[m]} \Psi_0^{2}(n_j, \eps_j, \delta_j), \frac{\textsf{avg}\big\{\Psi_0^2(n_j, \eps_j,\delta_j)\big\}_{j=1}^m}{m}, \frac{2\textsf{med}\big\{\Psi_0^2(n_j, \eps_j,\delta_j)\big\}_{j=1}^m}{m}\bigg\},
\end{align*}
where $\textsf{avg}$ and $\textsf{med}$ stand for the sample mean and median, respectively. Two intriguing implications can be derived from the aforementioned bound. First, federated PCA exhibits multiple robustness in the sense that the estimator produced by the central server remains consistent as long as at least one local estimator is consistent. Second, even if all local estimators are inconsistent, the central server can still deliver a consistent estimator provided the number of local clients $m\to\infty$. 

Federated PCA reduces to the differentially private PCA problem when $m=1$, in which case the bounds (\ref{eq:fed-dp-opt-spiked-pca}) and (\ref{eq:fed-dp-opt-spiked-cov}) align with the minimax optimal rates established in \cite{cai2024optimalPCA}. Our results immediately imply a performance bound for (non-interactive) \emph{local differentially private} (LDP) PCA under the spiked model. By setting $n_j\equiv r\equiv 1$ and assuming $\eps_j\equiv \eps=O(1)$ , the bound (\ref{eq:fed-dp-opt-spiked-pca}) suggests that the rate of LDP PCA under the spiked model is $p^2/(m\eps^2)$. Here, $m$ represents the sample size. The minimax lower bound easily follows from Theorem~\ref{thm:lower-bound}. However, our proposed estimator from Algorithm~\ref{alg:dpfedpca} will require a strong signal-to-noise ratio condition as  stated in Theorem~\ref{thm:highprob:upperbound} because spectral decomposition is implemented on a single datum. We leave this as future work.

\subsection{Related work}

Differentially private PCA was studied by \cite{blum2005practical,chaudhuri2011differentially, dwork2014analyze} in a deterministic setting without assuming data are independently sampled from a common distribution. \cite{liu2022dp} investigated online methods and established the minimax optimal rate for rank-one PCA under the spiked model. \cite{cai2024optimalPCA} leveraged spectral tools and established the minimax optimal rates for general rank-$r$ PCA and covariance matrix estimation. Federated PCA with homogeneous sample sizes and privacy constraints was studied by \cite{grammenos2020federated}, assuming data arrive sequentially and all data points are uniformly bounded. Their estimator is sub-optimal without exploiting the statistical properties of sample data under the spiked covariance model. \cite{wang2020principal} studied non-interactive local differentially private PCA assuming that the observations are sampled independently from a common distribution but are uniformly bounded.

\subsection{Organization of the paper}

The rest of the paper is organized as follows.  In Section~\ref{sec:upperbd}, we introduce a federated algorithm for differentially private PCA and covariance estimation. The algorithm incorporates three layers of spectral decomposition and employs the Gaussian mechanism to ensure privacy guarantees.  We demonstrate that the algorithm produces valid DP estimators of the population covariance matrix and its spectral projectors, achieving minimax optimal error rates up to logarithmic factors. Additionally, we provide a proof sketch of the main theorem, outlining the technical challenges and our proof strategy.
 Section~\ref{sec:lowerbound} establishes the minimax lower bounds for differentially private federated PCA and covariance estimation. The proof leverages a matrix version of Van Tree's inequality, which we believe is of independent interest.  In Section~\ref{sec:numerical},  we comprehensively evaluate the performance of our algorithm through numerical experiments and real data analysis, comparing it with existing methods.  All technical proofs are included in the supplementary material.

\section{Optimal Federated PCA by Gaussian Mechanism}\label{sec:upperbd}

In this section, we present the federated PCA and covariance matrix estimators under distributed differential privacy constraints.  Due to the different levels of  sensitivity of eigenvectors and eigenvalues,  our approach estimates the eigenvectors and eigenvalues separately.  Based on the given privacy budget,  each local client produces its own differentially private estimator of the eigenvectors and send them to the central server.  The central server aggregate these estimators with specially designed weights.  Since the aggregation may break the geometric constraints of eigenvectors, an additional step of eigen-decomposition is applied,  from which the spectral projector serves as the final estimator of eigenvectors.   The algorithm essentially consists of {\it three} layers of spectral decomposition:  two performed by the local clients and one by the central server. These multiple spectral decompositions are crucial for ensuring differential privacy and achieving minimax optimality. They pose significant technical challenges to the theoretical analysis, where we leverage sophisticated spectral representation tools \citep{xia2021normal,cai2024optimalPCA} to carefully examine the behavior of three-layer eigen-decompositions. 

After the differentially private estimator of eigenvectors is determined, the central server broadcasts them back to the local clients. These are then used to produce differentially private estimators of eigenvalues at each local client according to the given privacy budget. The central server receives these estimators, aggregates them by a weighted sum, and outputs the final estimator of the covariance matrix. The details of our approach are summarized in Algorithm~\ref{alg:dpfedpca}.  The operation $\svd_r(\cdot)$ returns the top-$r$ left singular vectors of a matrix.  
For simplicity, we assume that the rank $r$ and the nuisance noise level $\sigma^2$ are both known. The algorithmic parameters $\alpha_j$ and $\beta_j$ represent the sensitivity levels of empirical eigenvectors and eigenvalues (up to rotations).

\begin{algorithm}
	\caption{Differentially Private Federated PCA and Covariance Estimation}
	\label{alg:dpfedpca}
	\begin{algorithmic}
		\State{\textbf{Input: } sample data $\calD_j:=\big\{X_i^{(j)}: i\in[n_j]\big\}$ at the $j$-th local client  and its privacy budget $(\eps_j, \delta_j)$ for any $ j\in[m]$;  weights $w_j$ and $ v_j>0$ satisfying $\sum_{j=1}^m w_j=\sum_{j=1}^m v_j=1$. }
		\State{\ding{116} Part 1: PCA}
		\For{$j = 1,\cdots,m$} \Comment{on each local client}
		\State{Sample covariance matrix and eigenvectors 
			\begin{align*}
				\hat\Sigma_j = \frac{1}{n_j}\sum_{i=1}^{n_j}X_i^{(j)}X_i^{(j)\top}\quad {\rm and}\quad 
				\tilde U_j = \svd_r(\hat\Sigma_j),
			\end{align*}
		}
		\State{Gaussian mechanism for ensuring $(\epsilon_j,\delta_j)$-DP:
			\begin{align*}
				\hat U_j = \svd_r(\tilde U_j\tilde U_j^\top + Z_j), \quad [Z_j]_{kl} = [Z_j]_{lk}\stackrel{i.i.d.}{\sim} N(0,\alpha_j^2), k>l, [Z_j]_{kk}\stackrel{i.i.d.}{\sim} N(0,2\alpha_j^2),
			\end{align*}
		\State{$\quad\quad$ with $\alpha_j^2:= \frac{8}{\epsilon_j^2}\log(\frac{2.5}{\delta_j})\frac{\sigma^2}{\lambda}(\frac{\sigma^2}{\lambda}+1)\frac{p(r+\log n_j)}{n_j^2}$.}
		}
		\State{Send $\hat U_j$ to the central server.}
		\EndFor
		\State{Weighted average: $\hat U = \svd_r(\sum_{j=1}^m w_j\hat U_j\hat U_j^\top)$.}  \Comment{on central server}
		\State{\ding{116} Part 2: Covariance matrix estimation}
		\State{Send $\hat U$ to local client} \Comment{on central server}
		\For{$j=1,\cdots, m$}
		\State{$(\epsilon_j,\delta_j)$-DP estimator of eigenvalues:  \Comment{on each local client}
		\begin{align*}
			\hat\Lambda_j = \hat U^\top(\hat\Sigma_j - \sigma^2 I)\hat U +  E_j, \ [E_j]_{kl} = [E_j]_{lk}\stackrel{i.i.d.}{\sim} N(0,\beta_j^2), k>l, [E_j]_{kk}\stackrel{i.i.d.}{\sim} N(0,2\beta_j^2),
		\end{align*}
	\State{$\quad\quad$ with $\beta_j^2 := \frac{8}{\epsilon_j^2}\log\left(\frac{2.5}{\delta_j}\right)\frac{\lambda^2(r+\log n_j)^2 + \sigma^4p^2}{n_j^2}$.}
	}
		\State{Send $\hat\Lambda_j$ to the central server.}
		\EndFor
		\State{$\hat\Sigma := \sum_{j=1}^mv_j\hat U\hat\Lambda_j\hat U^\top + \sigma^2 I$.} \Comment{on central server}
		\State{{\bf Output:} $\hat U$ and $\hat\Sigma$.}
	\end{algorithmic}
\end{algorithm}

\begin{lemma}\label{lem:alg-dp}
Suppose that $X_i^{(j)}\stackrel{{\rm i.i.d.}}{\sim} N(0, \Sigma)$ with $\Sigma\in\Theta(\lambda, \sigma^2)$ for $j\in[m]$ and $i\in[n_j]$.  For any weight vectors $\v=(v_1,\cdots,v_m)^{\top}$ and $\w=(w_1,\cdots,w_m)^{\top}$, the output $\hat U\hat U^{\top}$ and $\hat\Sigma$ by Algorithm~\ref{alg:dpfedpca} are federated $(\beps, \bdelta)$-differentially private with probability at least $1-20\sum_{j=1}^m e^{-c_0(n_j\wedge p)}-\sum_{j=1}^m n_j^{-100}$ for some absolute constant $c_0>0$. 
\end{lemma}

By the post-processing property of differential privacy \citep[Proposition 2.1]{dwork2014algorithmic}, $\hat U\hat U^{\top}$ is federated $(\beps, \bdelta)$-DP as long as the estimator $\hat U_j\hat U_j^{\top}$ is $(\eps_j, \delta_j)$-DP at the $j$-th local client for all $j\in[m]$.  The proof of Lemma~\ref{lem:alg-dp} mainly focuses on establishing the privacy guarantee at local clients, which follows similarly to the proof of Lemma~2.2 in \cite{cai2024optimalPCA} except that we have an improved probability bound here.  

The following theorem shows that the final estimator $\hat U\hat U^{\top}$ is minimax optimal, up to logarithmic factors and the dependence on $\delta_j$'s,  if the weights $w_k$ are properly chosen.  Recall that $\Psi_{0}(n_j,  \eps_j, \delta_j)$, defined in (\ref{eq:dp-opt-spiked-cov}), quantifies the error rate for $\EE\|\hat U_j\hat U_j^{\top}-UU^{\top}\|_{\rm F}$ achieved at the $j$-th local client. 

\begin{theorem}\label{thm:highprob:upperbound}
Suppose $X_i^{(j)}\stackrel{{\rm i.i.d.}}{\sim} N(0, \Sigma)$ with $\Sigma\in\Theta(\lambda,\sigma^2)$, $n_j\geq C_1(r\log n_j + \log^2 n_j)$, $p\geq C_1\log n_j$ for some large constant $C_1>0$, and define $\tilde\Psi_{0}(n_j,  \eps_j, \delta_j)$ as 
\begin{equation}\label{eq:gamma_j}
\tilde\Psi_{0}(n_j, \eps_j,\delta_j):=\left(\frac{\sigma^2}{\lambda}+\sqrt{\frac{\sigma^2}{\lambda}}\right)\left(\sqrt{\frac{rp}{n_j}}+\frac{p\sqrt{r(r+\log n_j)}}{n_j\eps_j}\sqrt{\log\frac{2.5}{\delta_j}}\right)<c_1\sqrt{r},\quad \forall j\in[m].
\end{equation}
 satisfying $\tilde\Psi_{0}(n_j, \eps_j,\delta_j)<c_1\sqrt{r}$ for some small universal constant $c_1\in(0,1/2)$ for all $j\in[m]$. Let $\hat U$ be the estimator output by Algorithm~\ref{alg:dpfedpca} with weight $w_k:=\tilde\Psi_{0}^{-2}(n_k,\eps_k,\delta_k)/\sum_{j=1}^m \tilde\Psi_{0}^{-2}(n_j,\eps_j,\delta_j)$ for all $k\in[m]$.  Then  there exist  absolute constants $c_2, C_2>0$ such that 
	\begin{align}\label{eq:thm-PCA-upb1}
		\fro{\hat U\hat U^\top - UU^\top}^2 &\leq  \frac{C_2}{\sum_{j=1}^m \tilde\Psi_{0}^{-2}(n_j,\eps_j,\delta_j)} \bigwedge (2r), 
	\end{align}
which	holds with probability at least $1 - 22\sum_{j=1}^m e^{-c_2(n_j\wedge p)}$. Moreover,  if $(\lambda/\sigma^2)\sum_{j=1}^m n_j\leq e^{c_2\min_{j\in[m]} (n_j \wedge p)}$, then
\begin{align}\label{eq:thm-PCA-upb2}
	\EE\fro{\hat U\hat U^\top - UU^\top}^2 &\leq  \frac{C_2}{\sum_{j=1}^m \tilde\Psi_{0}^{-2}(n_j,\eps_j,\delta_j)} \bigwedge (2r).
	\end{align}
\end{theorem}

Note that the order of $\tilde\Psi_{0}(n_j,\eps_j,\delta_j)$ and $\Psi_0(n_j, \eps_j,\delta_j)$ only differs by $O(\sqrt{\log(1/\delta_j)})$ and $O(\log n_j)$ factors.  They represent the minimax optimal spectral norm rate of estimating $UU^{\top}$ for the $j$-th local client. The condition $\tilde\Psi_{0}(n_j,\eps_j,\delta_j)<\sqrt{r}$ requires that the differentially private estimator published by each local client is non-trivial and informative, albeit not necessarily consistent.  
Based on Theorem~\ref{thm:highprob:upperbound}, the optimal weights $w_k$ for aggregation are proportional to $\tilde\Psi_{0}^{-2}(n_k,\eps_k,\delta_k)$, respectively.  While the definitions of $\tilde\Psi_{0}(n_k,\eps_k,\delta_k)$'s involve the unknown signal strength $\lambda$, the weight $w_k$ only depends on known sample sizes and privacy constraints. In fact, we can simply set the following data-independent weight:
$$
w_k:=\frac{\sqrt{p/n_k}+(p/n_k\eps_k)\sqrt{(r+\log n_k)\log(2.5/\delta_k)}}{\sum_{j=1}^m \sqrt{p/n_j}+(p/n_j\eps_j)\sqrt{(r+\log n_j)\log(2.5/\delta_j)} },\quad \forall k\in[m].
$$
In the homogeneous case where $n_k\asymp n$,  $\eps_k\asymp \eps$,  and $\delta_k\asymp \delta$ for all $k\in[m]$,   these weights are $\omega_k\asymp m^{-1}$ of the same order.

The upper bound (\ref{eq:thm-PCA-upb1}) is the (scaled) harmonic mean of the error bounds for $\|\hat U_j \hat U_j^{\top}-UU^{\top}\|_{\rm F}^2$ for all $j\in[m]$.  Let us briefly elaborate on the technical challenges. Under mild conditions, the Davis-Kahan theorem \citep{davis1970rotation} yields
\begin{align}\label{eq:EhatU-U}
\EE \|\hat U\hat U^{\top}-UU^{\top}\|_{\rm F}^2\lesssim \EE \Big\|\sum_{j=1}^m w_j \Delta_j\Big\|_{\rm F}^2
=\sum_{j=1}^m w_j^2\EE\|\Delta_j\|_{\rm F}^2+\sum_{1\leq k_1\neq k_2\leq m} w_{k_1}w_{k_2}\EE \big< \Delta_{k_1},  \Delta_{k_2}\big>,
\end{align}
where $\Delta_j:=\hat U_j\hat U_j^{\top}-UU^{\top}$. The bound (\ref{eq:thm-PCA-upb2}) is primarily contributed by the first term. It remains to carefully control the expected inner product $\EE \big< \Delta_{k_1},  \Delta_{k_2}\big>$, where the naive approach 
by applying the Cauchy-Schwartz inequality delivers a sub-optimal bound. We exploit the spectral representation formula from \cite{xia2021normal} to show that the second term in \eqref{eq:EhatU-U} is dominated by the first one. 

\begin{proof}[Proof sketch of Theorem~\ref{thm:highprob:upperbound}] There exist three layers of spectral decomposition in Algorithm~\ref{alg:dpfedpca}.  Applying the spectral representation formula from \cite{xia2021normal} to the last eigen-decomposition,  we obtain 
\begin{align}\label{eq:proof-sketch-eq1}
	&\quad \frac{1}{2}\fro{\hat U\hat U^\top - UU^\top}^2\notag\\
	&= \sum_{l\geq 2}\sum_{\s\in\SS_l}(-1)^{\lzero{\s}}\sum_{j_1,\cdots,j_l\in[m]}w_{j_1}\cdots w_{j_l}\cdot\tr\Big(U^\top M(s_1)\underline{M(s_1)^\top\Delta_{j_1} M(s_2)}\notag\\
	&\hspace{8cm}\cdots\underline{M(s_l)^\top\Delta_{j_l} M(s_{l+1})}M(s_{l+1})^\top U\Big),
\end{align}
where $M(s)$ is a matrix-valued function, such that $M(0) = U_{\perp}$ and $M(s) = U$ for $s>0$, and $\SS_l := \big\{\s=(s_1,\cdots, s_{l+1})^{\top}\in\ZZ^{l+1}: s_1,\cdots,s_{l+1}\geq 0, s_1+\cdots + s_{l+1}=l\big\}.$ We use the underline below to emphasize the recurrent terms in the pattern $M(s_l)\Delta_j M(s_{l+1})$. Essentially,  three different patterns of terms appear in the summands of products in eq. (\ref{eq:proof-sketch-eq1}): $U^{\top}\Delta_j U$,  $U^{\top}\Delta_j U_{\perp}$,  and $U_{\perp}^{\top}\Delta_j U_{\perp}$.  

Recall $\Delta_j=\hat U_j\hat U_j^{\top}-UU^{\top}$ where $\hat U_j$ consists of the top-$r$ eigenvectors of $UU^{\top}+D_j$ with $D_j:=\tilde U_j\tilde U_j^{\top}-UU^{\top}+Z_j$.  Similarly,  we can write 
\begin{align}\label{eq:proof-sketch-eq2}
	\Delta_j =\sum_{l\geq1}\sum_{\s\in\SS_l}(-1)^{\lzero{\s} + 1}M(s_1)\cdot \underline{M(s_1)^\top D_{j}M(s_2)}\cdots \underline{M(s_l)^\top D_j M(s_{l+1})}\cdot M(s_{l+1})^\top
\end{align}
and
\begin{align}
\tilde U_j\tilde U_j^{\top}-&UU^{\top}=\sum_{l\geq 1}\sum_{\s\in\SS_l}  (-1)^{\lzero{\s} + 1}\cdot\notag\\
&M(s_1)\Lambda^{-s_1}\underline{M(s_1)^\top\Xi_j M(s_2)}\Lambda^{-s_2}\cdots\Lambda^{-s_l}\underline{M(s_l)^\top\Xi_j M(s_{l+1})}\Lambda^{-s_{l+1}}M(s_{l+1})^\top.
\end{align}
The above representation formulas show that the basic building elements are the terms $U^{\top}(\Xi_j+Z_j)U$,  $U^{\top}(\Xi_j+Z_j)U_{\perp}$,  and $U_{\perp}^{\top}(\Xi_j+Z_j)U_{\perp}$.  As a result,  we will show that there is an event $\calE$ with $\PP(\calE)\geq 1-14\sum_{j=1}^m e^{-c_2(p\wedge n_j)}$,  in which the following bounds hold 
\begin{equation}\label{eq:proof-sketch-eq4}
\begin{aligned}
\max\Big\{\big\|U^{\top}\Delta_j U \big\|, \big\|U_{\perp}^{\top}\Delta_j U_{\perp} \big\|\Big\}\lesssim & \left(\frac{\sigma^2}{\lambda}+\frac{\sigma^4}{\lambda^2}\right)\left(\frac{p}{n_j}+\frac{p^2(r+\log n_j)}{n_j^2\eps_j^2}\log\frac{2.5}{\delta_j}\right),\\
\big\|U^{\top}\Delta_j U_{\perp}\big\|\lesssim& \left(\frac{\sigma}{\sqrt{\lambda}}+\frac{\sigma^2}{\lambda}\right)\left(\sqrt{\frac{p}{n_j}}+\frac{p\sqrt{r+\log n_j}}{n_j\eps_j}\log^{1/2}\frac{2.5}{\delta_j}\right).
\end{aligned}
\end{equation}
For each fixed $\s\in\SS_l$, we consider the upper bound for 
\begin{align}\label{eq:proof-sketch-eq3}
	\bigg|\EE\sum_{j_1,\cdots,j_l\in[m]}w_{j_1}\cdots w_{j_l} \tr(U^\top M(s_1)\underline{M(s_1)^\top\Delta_{j_1} M(s_2)}\cdots\underline{M(s_l)^\top\Delta_{j_l} M(s_{l+1})}M(s_{l+1})^\top U)\cdot\mathds{1}(\calE)\bigg|.
\end{align}
The above summand is non-zero if and only if $s_1, s_{l+1}\geq 1$. Since $s_1+\cdots + s_{l+1} =l$, there exists $1\leq i_1<i_2\leq l$, such that 
$s_{i_1}>0, s_{i_1+1}=0$, and $s_{i_2}=0,s_{i_2+1}>0$. We define
\begin{align*}
	\II_1(\s) = \bigg\{\j\in[m]^l: j_{i_1}\neq j_{i_2}, \{j_{i_1},j_{i_2}\}\cap \{j_1,\cdots, \bar{j_{i_1}},\cdots, \bar{j_{i_2}},\cdots, j_{l}\} = \emptyset\bigg\}.
\end{align*}
Here $\bar\cdot$ means $\cdot$ is absent from the set. Then $|\II_1(\s)| = m(m-1)(m-2)^{l-2}$. Denote $\II_2(\s):=[m]^l\setminus\II_1(\s)$. The sum in (\ref{eq:proof-sketch-eq3}) can be decomposed into two parts: over $\II_1(\s)$ and $\II_2(\s)$, respectively.  The proof is concluded by bounding the summands in (\ref{eq:proof-sketch-eq3}) for all $\j\in\II_s(\s)$ using the facts (\ref{eq:proof-sketch-eq4}). 
\end{proof}

We now show that the covariance matrix estimator $\hat\Sigma$ output by Algorithm~\ref{alg:dpfedpca} achieves the minimax optimal rate. For each $j\in[m]$, define
\begin{align}\label{eq:gamma_1}
\tilde\Psi_{1}(n_j,\eps_j,\delta_j):=\sqrt{\frac{r(r+\log n_j)}{n_j}}+\frac{\sqrt{r(r+\log n_j)^{3}}}{n_j\eps_j}\sqrt{\log\frac{2.5}{\delta_j}},
\end{align}
which satisfies $\tilde\Psi_{1}(n_j,\eps_j,\delta_j)\asymp \Psi_1(n_j, \eps_j, \delta_j)$, up to $O\big(\sqrt{\log(1/\delta_j)}\big)$ and $O(\log n_j)$ factors.  Recall that $\lambda \cdot \tilde\Psi_{1}(n_j,\eps_j,\delta_j)$ quantifies the error rate for estimating eigenvalues under the $(\eps_j, \delta_j)$-DP  constraint achieved by the $j$-th local client.

\begin{theorem}\label{thm:covariance:upperbound}
Suppose the conditions in Theorem \ref{thm:highprob:upperbound} hold, and set the weights in Algorithm~\ref{alg:dpfedpca} such that $\sum_{j=1}^m v_j=1$ and 
$$
v_j\propto \left(\frac{\lambda^2+\sigma^4}{n_j} + \frac{8}{\epsilon_j^2}\log\left(\frac{2.5}{\delta_j}\right)\frac{\lambda^2(r+\log n_j)^2 + \sigma^4p^2}{n_j^2}\right)^{-1}.
$$ 
There exist absolute constants $c_2, C_2>0$ such that the bound
\begin{align*}
\fro{\hat\Sigma - \Sigma}^2 \leq & C_2\left(\frac{\lambda^2}{\sum_{j=1}^m \tilde\Psi_{1}^{-2}(n_j,\eps_j,\delta_j)}+\frac{\lambda^2}{\sum_{j=1}^m \tilde\Psi_{0}^{-2}(n_j,\eps_j,\delta_j)}\right)\bigwedge (2r\lambda^2)
\end{align*}
holds with probability at least $1-23\sum_{j=1}^m e^{-c_0(n_j\wedge p)} -\sum_{j=1}^m n_j^{-100}$. Moreover, if $(\lambda/\sigma^2)\sum_{j=1}^m n_j \leq  e^{c_0\min_{j\in[m]}(n_j\wedge p)}$, then we have 
\begin{align}\label{eq:thm-cov-upb1}
\EE\fro{\hat\Sigma - \Sigma}^2 \leq & C_2\left(\frac{\lambda^2}{\sum_{j=1}^m \tilde\Psi_{1}^{-2}(n_j,\eps_j,\delta_j)}+\frac{\lambda^2}{\sum_{j=1}^m \tilde\Psi_{0}^{-2}(n_j,\eps_j,\delta_j)}\right)\bigwedge (2r\lambda^2).
\end{align}
\end{theorem}

Our proposed Algorithm~\ref{alg:dpfedpca} separately estimates the eigenvectors and eigenvalues under privacy constraints. The central server aggregates differentially private estimators of both the eigenvalues and eigenvectors sent from the local clients. Therefore, the bound (\ref{eq:thm-cov-upb1}) involves two terms, primarily contributed by the estimation of the eigenvalues and eigenvectors, respectively. The bound (\ref{eq:EhatU-U}) demonstrates the (doubly) multiple robustness of the estimator $\hat\Sigma$. As long as one client can provide a consistent estimator of the eigenvectors and another (can be the same client) can provide a consistent estimator of the eigenvalues, the aggregated estimator $\hat\Sigma$ delivered by the central server remains consistent. The weights $v_j$  rely on the unknown eigenvalue $\lambda$. For simplicity, the empirical eigenvalue can be used in practice. Alternatively, one can resort to random matrix theory \citep{benaych2011eigenvalues} to obtain a sharper estimate of $\lambda$. 

In the homogeneous case when $n_j\asymp n, \eps_j\asymp \eps$ and $\delta_j\asymp \delta$, we have $\Psi_0(n_j, \eps_j, \delta_j)\asymp \Psi_0(n, \eps, \delta)$ and $\Psi_1(n_j, \eps_j, \delta_j)\asymp \Psi_1(n, \eps, \delta)$ for all $j\in[m]$. Theorems~\ref{thm:highprob:upperbound} and \ref{thm:covariance:upperbound} show that  the estimators $\widehat U$ and $\widehat \Sigma$ output by Algoirthm~\ref{alg:dpfedpca} achive the rates (up to logarithmic factors):
\begin{equation*}
	\begin{aligned}
		&\EE \|\whU\whU^{\top}-UU^{\top}\|_{\rm F}^2 \lesssim \bigg(\frac{\sigma^4}{\lambda^2}+\frac{\sigma^2}{\lambda}\bigg) \left(\frac{pr}{mn} +\frac{p^2 r^2}{mn^2\epsilon^2}\right)\bigwedge r;\\
		&\EE \|\whSig-\Sigma\|_{\rm F}^2 \lesssim 
		 \lambda^2 \left(\frac{r^2}{mn} + \frac{r^4}{mn^2\varepsilon^2}\right) + \sigma^2(\lambda+\sigma^2) \left(\frac{pr}{mn} + \frac{p^2r^2}{mn^2 \varepsilon^2} \right)\bigwedge (r\lambda^2),
	\end{aligned}
\end{equation*} 
which decay whenever the number of local clients $m$ or local sample size $n$ increases.  The aggregate sample size across all local clients is $mn$. The statistical error, quantified by the rate $pr/(mn)$, is inversely proportional to this total sample size. Notably, this rate aligns with the minimax optimal rate achievable by estimators that utilize all observations collectively \citep{cai2016estimating}. This implies that distributing observations evenly among $m$ local clients does not compromise statistical efficiency.  In contrast, the privacy cost is represented by the rate $p^2r^2/(mn^2\varepsilon^2)$, which decreases as the number of local clients increases. As demonstrated in \cite{cai2024optimalPCA}, the rate $p^2r^2/(n^2\varepsilon^2)$ reflects the privacy cost at each individual local client. This suggests that aggregating multiple differentially private estimators can effectively reduce the overall privacy cost.  Another interpretation of the rate  $p^2r^2/(mn^2\varepsilon^2)$  is to express it as $m\cdot p^2r^2/(m^2n^2\varepsilon^2)$,  where $mn$ in the denominator represents the total sample size. When the total sample size is fixed, the privacy cost increases with the number of local clients $m$.  This is because maintaining differential privacy becomes more challenging as the number of observations per local client decreases.

\section{Minimax Lower Bound}\label{sec:lowerbound}

In this section, we establish the minimax lower bounds for PCA and covariance matrix estimation under the federated $(\beps, \bdelta)$-DP constraints. These lower bounds match, up to logarithmic factors and $\delta_j$-terms, the upper bounds achieved by our proposed estimators derived from Algorithm~\ref{alg:dpfedpca}. 

Under the spiked model with a covariance matrix $\Sigma\in\Theta(\lambda,\sigma^2)$, we denote $\mcU_{\n, \beps,\bdelta}$ and $\mcM_{\n, \beps, \bdelta}$ the collection of all federated  $(\beps, \bdelta)$-DP estimators of $U$ and $\Sigma$,  respectively. The vector $\n:=(n_1,\cdots,n_m)^{\top}$ stands for the sample sizes at local clients. Recall that the rates $\Psi_0(n_j, \eps_j, \delta_j)$ and $\Psi_1(n_j, \eps_j, \delta_j)$ defined in (\ref{eq:dp-opt-spiked-cov}) characterize the minimax optimal rates for differentially private estimators achievable at the $j$-th local client. Moreover, $\Psi_0(n_j,\eps_j,\delta_j)\asymp \tilde\Psi_0(n_j,\eps_j,\delta_j)$ and $\Psi_1(n_j,\eps_j,\delta_j)\asymp \tilde\Psi_1(n_j,\eps_j,\delta_j)$, up to logarithmic factors, for all $j\in[m]$. For presentation clarity, the following theorem focuses on the case $\eps_j=O(1)$ for all $j\in[m]$.

\begin{theorem}\label{thm:lower-bound}
Suppose $X_i^{(j)}\stackrel{{\rm i.i.d.}}{\sim} N(0, \Sigma)$, $p\geq 2r$,  and $\max_{j\in[m]} \eps_j\leq C_0$ for some large absolute constant $C_0>0$.  
There exist absolute constants $c_0, c_1>0$ such that if $\big(rp+\sqrt{rpn_j}\big)\delta_j^{0.9}\leq c_1n_j\eps_j^2$ for all $j\in[m]$, then 
\begin{equation}\label{eq:thm-lwb1}
\begin{aligned}
\inf_{\hat U \in \mcU(\n, \beps,\bdelta)}\sup_{\Sigma\in\Theta(\lambda,\sigma^2)}&\EE\fro{\hat U\hat U^\top - UU^{\top}}^2 \geq  \frac{c_0}{\sum_{j=1}^m\Psi_0^{-2}(n_j,\eps_j,\delta_j)}\bigwedge (2r),\\
\inf_{\hat \Sigma \in \mcM(\n, \beps,\bdelta)}\sup_{\Sigma\in\Theta(\lambda,\sigma^2)}&\EE\fro{\hat\Sigma - \Sigma}^2 \geq  c_0\left(\frac{\lambda^2}{\sum_{j=1}^m\Psi_0^{-2}(n_j,\eps_j,\delta_j)}+\frac{\lambda^2}{\sum_{j=1}^m\Psi_1^{-2}(n_j,\eps_j,\delta_j)}\right)\bigwedge (2r\lambda^2).
\end{aligned}
\end{equation}
\end{theorem}

In the special case of  $m=1$, the bound (\ref{eq:thm-lwb1}) matches the lower bound for differentially private PCA established in \cite{cai2024optimalPCA}. Theorem~\ref{thm:lower-bound} shows that the minimax lower bound in federated PCA is the harmonic mean of the minimax lower bounds at each local client. The technical tool in \cite{cai2024optimalPCA} is a differentially private version of Fano's lemma, which imposes a restricted condition on the range of allowed $\delta_j$'s. In contrast, Theorem~\ref{thm:lower-bound} allows a much wider range of $\delta_j$'s. We remark that the exponent $0.9$ can be replaced by $k/(k+1)$ for any positive integer $k\geq 1$.  The minimax lower bounds in Theorem~\ref{thm:lower-bound} hold as long as $\big(rp+\sqrt{rpn_j}\big)\delta_j^{1-\zeta}\leq c_1n_j\eps_j^2$ for any $\zeta\in(0,1)$. 
 
Our main technical tool for proving Theorem~\ref{thm:lower-bound} is a matrix version of Van Tree's inequality, which quantifies a lower bound for the average error rate of estimating principal components under privacy constraints. We then establish the inequality (\ref{eq:thm-lwb1}) by specifying a prior distribution over the set $\OO^{p\times r}$ and bounding the Fisher information. The detailed proof is provided in Appendix~\ref{sec:proof-lower-bound} in the supplementary materials.

\section{Numerical Experiments}\label{sec:numerical}

Our proposed algorithm, \alg, is easy to implement. In this section, we evaluate its numerical performance through simulations and demonstrate its practical utility by applying it to a lung cancer dataset. To provide a comprehensive evaluation, we also compare its performance against two alternative approaches: the equal-weight aggregation method and the \textsf{Fed-DP-Oja} algorithm \citep{grammenos2020federated, liu2022dp}.

\subsection{Simulations}
We present simulation results comparing our proposed algorithm, \alg, with existing algorithms and their variations. Specifically, we evaluate the \textsf{Fed-DP-Oja} algorithm introduced in \cite{grammenos2020federated}, which addresses federated PCA under homogeneous sample sizes and privacy constraints. Additionally, we compare our approach with an alternative aggregation method that assigns equal weights to each client. We also examine a strategy where each local client transmits $\tilde U_j\tilde U_j^\top + Z_j$ to the central server. While this method ensures privacy protection, it is not an optimal estimator of principal components, as it generally fails to qualify as a valid spectral projector and incurs additional communication costs. Nevertheless, we include the results from this approach as a \textsf{reference}. In all experiments, we set the covariance matrix to $\Sigma = \lambda UU^\top + I_p$, where $U \in \mathbb{R}^{p \times r}$ is an orthogonal matrix generated by extracting the left singular vectors of a randomly generated matrix with i.i.d. entries via QR decomposition.
Performance is assessed using the projection distance between the estimated subspace and the true subspace, defined by $\|\hat\U\hat\U^\top - \U\U^\top\|_{\mathrm{F}}$.

In the first simulation setting, we examine the utility-privacy trade-off under homogeneous conditions. We set the dimensionality to $p = 50$, rank to $r = 1$, and signal strength to $\lambda = 10$. The data are distributed across $m = 10$ clients, each with a privacy budget of $\epsilon_j \equiv \epsilon$ and $\delta_j \equiv 0.1$, and a sample size of $n_j = 10,000$. Given the homogeneous setting, the optimal choice of weights is equal weighting. Therefore, we compare our proposed method with the \textsf{Fed-DP-Oja} algorithm and the \textsf{reference} approach. The privacy budget $\epsilon$ varies between 0.1 and 1.0. For each choice of $\epsilon$, the simulation is repeated 50 times. 
The results, presented in Figure~\ref{fig:privacy-utility-eps}, demonstrate that the \textsf{Fed-DP-Oja} algorithm significantly underperforms compared to both our proposed method and the \textsf{reference}  approach.  In contrast, our method achieves performance nearly identical to the \textsf{reference}. These findings confirm that transmitting the top $r$ left singular vectors of $\tilde U_j\tilde U_j^\top + Z_j$. to the server is sufficient for effective federated PCA. Additionally, larger values of $\epsilon$ correspond to weaker privacy guarantees but result in more accurate estimations. This behavior aligns with our theoretical predictions.

In the second experiment, we evaluate the estimation quality as the total number of total clients $m$ varies. We use the same parameters: $p=50, r=1, \lambda=10$. Each client is assigned a privacy budget of $\epsilon_j \equiv 0.5$ and $\delta_j \equiv 0.1$. 
We consider a homogeneous setting where each client has a sample size of  $n_j = n = 1000$ and vary the number of clients  $m\in\{10, 20,\cdots, 100\}$.
For each value of $m$, the simulation is repeated 50 times. 
The results, depicted in Figure~\ref{fig:client}, show that our proposed method achieves performance comparable to the reference approach while significantly outperforming the \textsf{Fed-DP-Oja} algorithm. Furthermore, as the number of clients increases, the estimation accuracy improves. These findings indicate that our method effectively leverages information from multiple clients, enhancing the quality of the estimated principal components as the client population grows.

In the third experiment, we investigate the effect of varying the number of clients $m$ on estimation quality while maintaining a fixed total number of samples $N$. Specifically, for each client $j$, the sample size is set to $n_j\equiv N/m$ . We configure the parameters as 
$p=50,r=1,\lambda=10$, with each client assigned a privacy budget of $\epsilon_j \equiv 0.5$ and $\delta_j \equiv 0.1$. The total sample size is fixed at $N=100,000$, and we vary the number of clients $m$ across the values $\{10,20,25,50\}$. For each configuration, we conduct 50 independent simulation runs. The results are illustrated in Figure~\ref{fig:fix_total}.
The findings indicate that, under a fixed sample complexity, a smaller number of clients leads to more accurate estimations. This occurs because fewer clients allow for larger sample sizes per client, thereby enhancing the quality of the local principal component estimates and facilitates easier privacy preservation. These results align with our theoretical predictions.

Lastly, we assess the performance of our method under heterogeneous sample sizes and privacy budgets. We set $p=50,r=1$, and $\lambda=10$, with data distributed across $m = 10$ clients. For each client, the privacy parameters $\epsilon_j, \delta_j$ are independently and uniformly drawn from (0.1,0.3) and (0.1,0.2), respectively. To introduce heterogeneity in sample sizes, we allocate a sample size of $2*\textsf{N}_{\textsf{sample}}$ to the first five clients and  $20*\textsf{N}_{\textsf{sample}}$ to the remaining five clients, where $\textsf{N}_{\textsf{sample}}\in\{100, 200, \cdots, 1000\}$. The results are presented in Figure~\ref{fig:hetero}.
Our proposed method outperforms the equal weight aggregation approach and even the \textsf{reference} method. This superior performance is attributed to our method's ability to optimally weight clients based on their individual sample sizes and privacy budgets, thereby effectively balancing the trade-offs inherent in a heterogeneous setting. In contrast, the reference method does not account for such heterogeneity in its weighting scheme, resulting in less efficient estimation.

\begin{figure}[htbp]
	\centering
	\begin{subfigure}{0.45\textwidth}
		\centering
		\includegraphics[width=\textwidth]{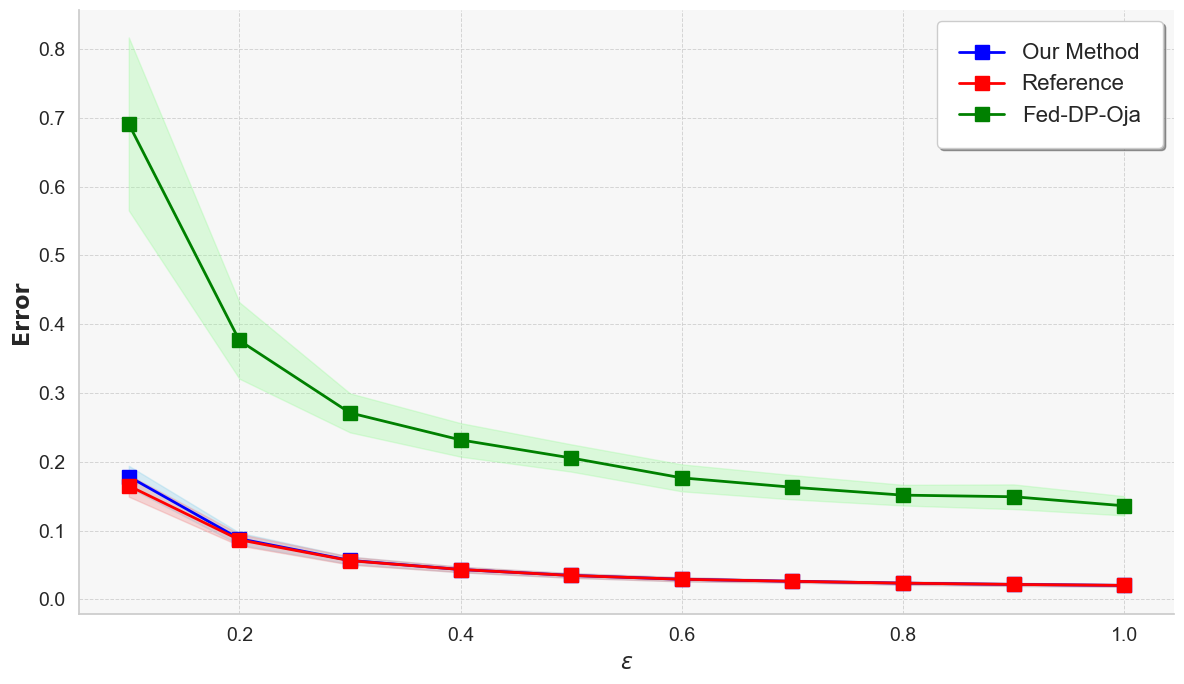}
		\caption{Privacy-utility trade-off}
		\label{fig:privacy-utility-eps}
	\end{subfigure}
	\hfill
	\begin{subfigure}{0.45\textwidth}
		\centering
		\includegraphics[width=\textwidth]{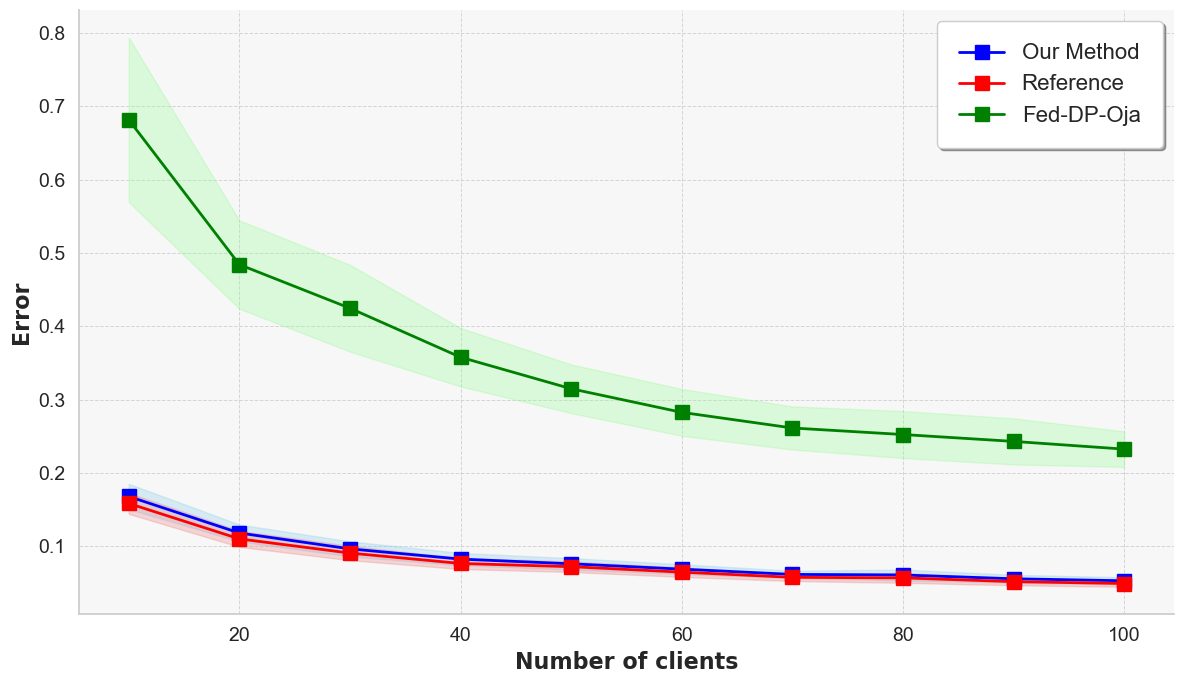}
		\caption{Estimation error versus number of clients}
		\label{fig:client}
	\end{subfigure}
	
	\begin{subfigure}{0.45\textwidth}
		\centering
		\includegraphics[width=\textwidth]{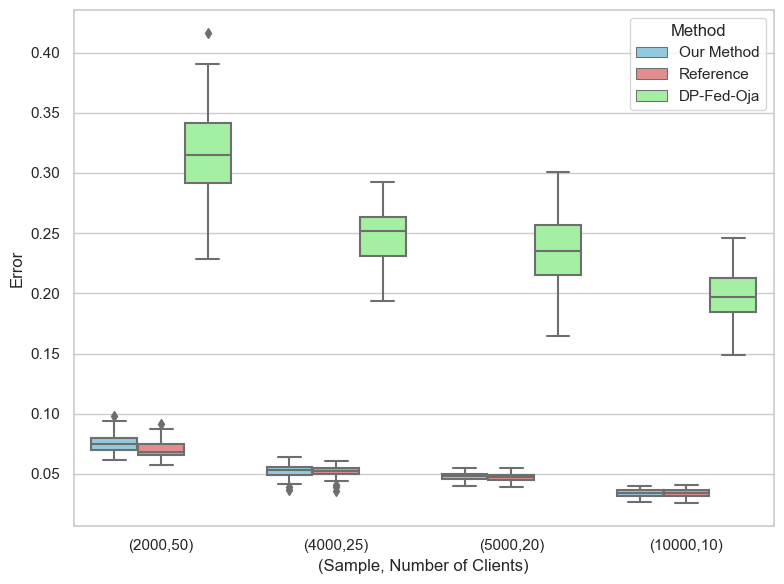}
		\caption{Estimation error under fixed total sample size}
		\label{fig:fix_total}
	\end{subfigure}
	\hfill
	\begin{subfigure}{0.45\textwidth}
		\centering
		\includegraphics[width=\textwidth]{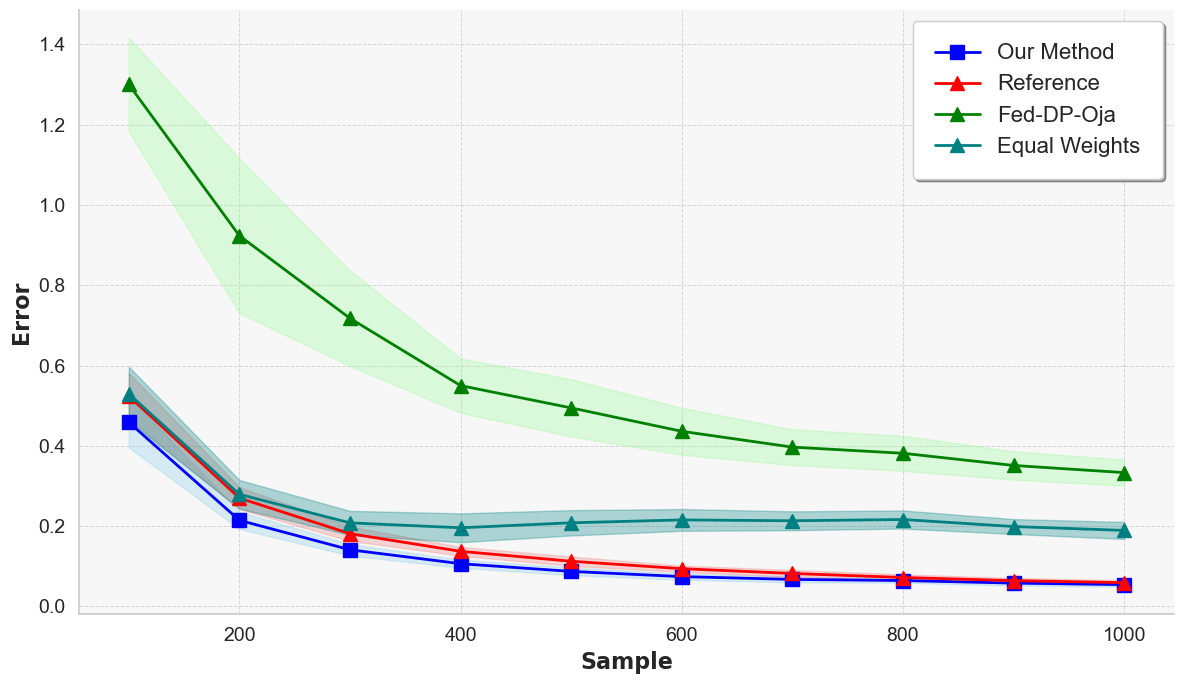}
		\caption{Heterogeneous sample sizes and privacy budgets}
		\label{fig:hetero}
	\end{subfigure}
	\caption{Numerical simulations comparing our method with existing methods and their variations. The performance is assessed using the projection distance $\fro{\hat\U\hat\U^\top - \U\U^\top}$. }
\end{figure}
	
\subsection{The Lung Cancer Data}
In this section, we illustrate the practical utility  of the proposed algorithm, \alg, by applying it to a lung cancer dataset. We also compare its performance with the equal-weight aggregation approach and the \textsf{Fed-DP-Oja} algorithm.

The Lung Cancer dataset, initially collected and cleaned by \cite{gordon2002translation}, comprises expression data for 12,533 genes across 181 subjects, categorized into diseased and normal groups. Following the refinement by \cite{jin2016influential}, genes without differential expression between the groups were excluded, resulting in a curated data matrix with dimensions  $p = 251$. 

For our experiment, we consider a federated setting with $m=2$ clients. We randomly shuffle the sample indices and assign the first 130 samples to Client 1, and the remaining 51 samples to Client 2. We set the target rank to $r=5$. For each client, the signal strength $\lambda$ is estimated by averaging the first three eigenvalues of the sample covariance matrix, and the noise variance $\sigma^2$  is estimated as the mean of the 51st to 251st sample eigenvalues. Both clients are allocated identical privacy budgets of 
$\epsilon=0.4$ and $\delta=0.1$.  

Subsequently, each client computes a differentially private subspace estimation and transmits it, along with the corresponding unnormalized weights, to a central server for aggregation. We compare the performance of our method with that of the \textsf{Fed-DP-Oja} algorithm and the equal-weight aggregation approach. After aggregation, we perform dimensionality reduction using the estimated subspace at the central server and report the explained variance as the evaluation metric. The results are illustrated in Figures \ref{fig:compare}, \ref{fig:equalweight}, and \ref{fig:ours}.
The experimental outcomes indicate that our proposed method outperforms the \textsf{Fed-DP-Oja} algorithm, which requires the addition of excessively large noise, thereby degrading its performance. Moreover, when compared to the equal-weight aggregation approach, our method achieves a higher explained variance, demonstrating its superior ability to capture the underlying data structure effectively. These results underscore the efficacy of our method in balancing privacy constraints with estimation accuracy in a federated learning environment.

\begin{figure}[htbp]
	\centering
	\begin{subfigure}[b]{0.49\textwidth}  
		\centering
		\includegraphics[width=\textwidth]{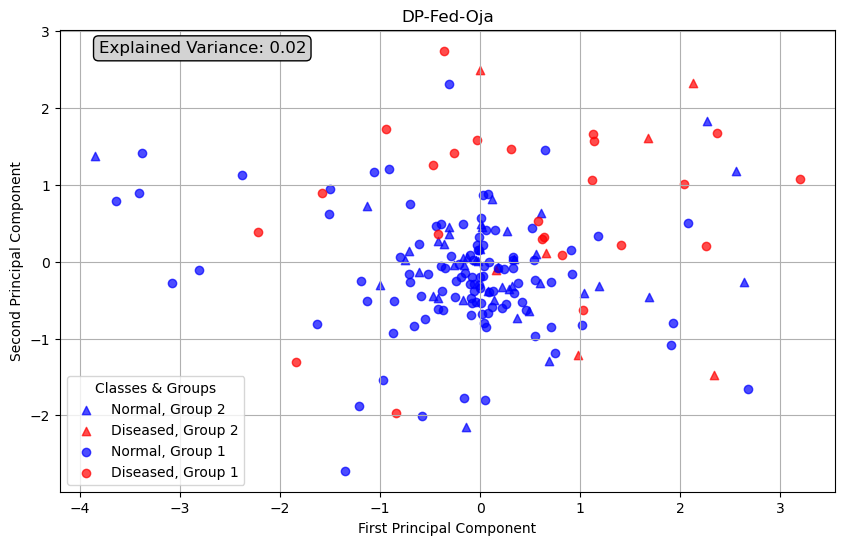}  
		\caption{\textsf{Fed-DP-Oja}}
		\label{fig:compare}
	\end{subfigure}
		\hfill  
	\begin{subfigure}[b]{0.49\textwidth}  
		\centering
		\includegraphics[width=\textwidth]{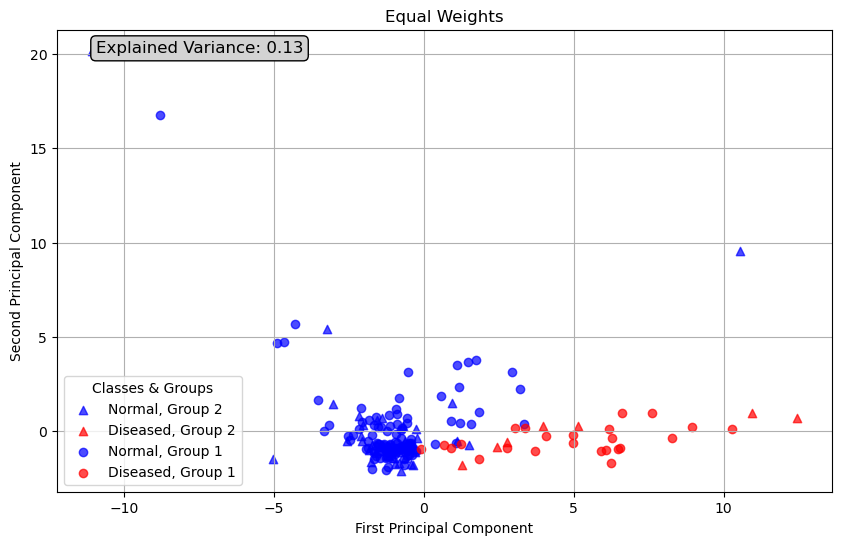}  
		\caption{Using equal weights}
		\label{fig:equalweight}
	\end{subfigure}
	
	\vskip\baselineskip  
	\begin{subfigure}[t]{0.5\textwidth}
		\centering
		\includegraphics[width=\linewidth]{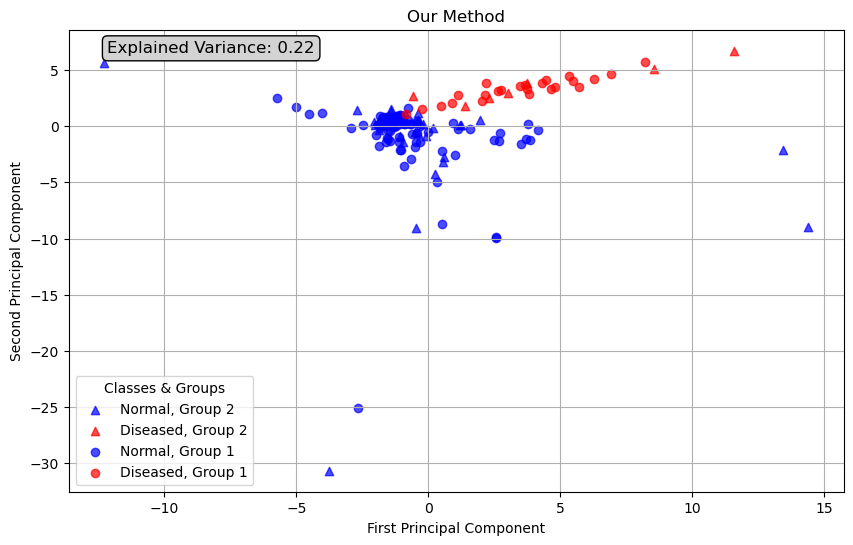}
		\caption{Our method}
		\label{fig:ours}
	\end{subfigure}
	\caption{We compare our proposed method with the equal-weight aggregation approach and the \textsf{Fed-DP-Oja} algorithm using the Lung Cancer dataset. For simplicity, the entire dataset is unevenly divided between two clients, each allocated a privacy budget of $\epsilon = 0.4,\delta= 0.1$.}
\end{figure}

\section{Discussions}
This paper establishes minimax optimal rates and demonstrates the multiple robustness and scalability of federated PCA. The central server's estimator remains consistent as long as at least one local estimator is consistent. Moreover, even if all local estimators are inconsistent, the central estimator can still be consistent given a sufficient number of local clients. These findings highlight federated learning's potential for reliable and robust statistical inference in a privacy-preserving manner, paving the way for further research and application in fields requiring stringent privacy and data security measures. 

For simplicity, we assume in this paper that the mean vector of the data distribution is either zero or known. However, the approach can be readily adapted to handle cases where the mean is unknown. In such instances, we calculate the client-specific sample covariance matrix as   $\widehat \Sigma_j=(n_j-1)^{-1}\sum_{i=1}^{n_j}\big(X_i^{(j)}-\bar X^{(j)}\big)\big(X_i^{(j)}-\bar X^{(j)}\big)^{\top}$, where $\bar{X}^{(j)}$ denotes the sample mean vector for the $j$-th local client. Under the Gaussian assumption, the distribution of $(n_j - 1)\widehat{\Sigma}_j$ remains Wishart, which preserves the validity of all technical proofs presented in this work, except that the sensitivity of empirical spectral projectors and eigenvalues need to be carefully re-examined. For analytical convenience, we assume a Gaussian data distribution throughout our study. Extending these results to sub-Gaussian or more general distributions is an intriguing avenue for future research. Nevertheless, as highlighted earlier, the main technical challenges lie in developing a unified framework to bound higher-order perturbation terms that arise from the three layers of spectral decomposition.

An interesting future research direction is the study of federated SVD under the low-rank matrix denoising model.  While SVD and PCA are closely related in traditional settings, they differ significantly in the context of federated learning under DP constraints due to differences in measurement units. Specifically, the covariance matrix is symmetric, whereas the low-rank signal in the matrix denoising model can have dimensions that differ drastically \citep{cai2018rate}. This introduces additional challenges and unique features when investigating minimax optimal rates for estimating the left and right singular subspaces under distributed differential privacy constraints. Nonetheless, we believe the multiple robustness phenomenon observed in federated PCA also applies to federated SVD, with the minimax optimal rate at the central server being the harmonic mean of the minimax optimal rates achievable at each local client.
	
Additionally, it is worthwhile to explore the minimax optimal rates in federated sparse PCA \citep{cai2013sparse} and tensor PCA \citep{zhang2018tensor} under privacy constraints. These problems often rely on iterative algorithms, making the development of sharp upper bounds technically challenging. Moreover, these settings are known to exhibit a statistical-to-computational gap even without privacy constraints. Understanding the interplay between privacy constraints and computational feasibility in these problems remains an open and important research problem.

\section{Acknowledgment}
Tony Cai's research was supported in part by NSF grant DMS-2413106 and NIH grants R01-GM123056 and R01-GM129781. Dong Xia's research was partially supported by Hong Kong RGC Grant GRF 16302323 and 16303224. Anru R. Zhang's research was partially supported by NSF Grant CAREER-2203741 and NIH Grants R01HL169347 and R01HL168940.

\bibliographystyle{plainnat}
\bibliography{reference.bib}

\newpage
\appendix

\section{Proofs}




\subsection{Proof of Theorem \ref{thm:highprob:upperbound}}

We first derive the upper bound for the expectation, and then derive the high probability upper bound. In the proof, we set $\lambda_{\max} = \lambda_1$ and $\lambda_{\min} =\lambda_r$. 

\hspace{1cm}

\noindent\textbf{Upper bound for expectation. }
We now derive the upper bound for $\EE\fro{\hat\U\hat\U^\top - \U\U^\top}^2$. 
We will first expand $\fro{\hat\U\hat\U^\top - \U\U^\top}^2$. 
Denote $\bDel = \sum_{j=1}^m w_j\hat\U_j\hat\U_j^\top - \U\U^\top =: \sum_{j=1}^m w_j\bDel_j$. 
We define the event $\calF_0 = \{\op{\sum_{j=1}^m w_j\hat\U_j\hat\U_j^\top - \U\U^\top}\leq 1/4\}$ and we will show shortly from \eqref{UDU}, $\calF_0$ holds with high probability. Notice that the columns of $\hat\U$ are the top $r$ left singular vectors of $\sum_{j=1}^m w_j\hat\U_j\hat\U_j^\top$, we can use the representation formula developed in \cite{xia2021normal} to show the following expansion holds under $\calF_0$:
\begin{align*}
	\hat\U\hat\U^\top - \U\U^\top = \sum_{l\geq 1}\calS_{\U\U^\top,l}(\bDel).
\end{align*} 
Here $\calS_{\U\U^\top,l}(\bDel)$ takes the following form:
\begin{align*}
	\calS_{\U\U^\top,l}(\bDel) = \sum_{\s\in\SS_l}(-1)^{\lzero{\s} + 1}\M(s_1)\underline{\M(s_1)^\top\bDel \M(s_2)}\cdots\underline{\M(s_l)^\top\bDel \M(s_{l+1})}\M(s_{l+1})^\top,
\end{align*}
where $\M(s)$ is a matrix-valued function, such that $\M(0) = \U_{\perp}$ and $\M(s) = \U$ for $s>0$, and $$\SS_l = \{(s_1,\cdots, s_{l+1}): s_1,\cdots,s_{l+1}\geq 0, s_1+\cdots + s_{l+1}=l\}.$$
Under $\calF_0$, we can expand $\fro{\hat\U\hat\U^\top - \U\U^\top}^2$ as 
\begin{align*}
	\frac{1}{2}\fro{\hat\U\hat\U^\top - \U\U^\top}^2 &= r - \inp{\hat\U\hat\U^\top}{\U\U^\top} = - \inp{\hat\U\hat\U^\top - \U\U^\top}{\U\U^\top}\\
	&= -\sum_{l\geq 2}\inp{\calS_{\U\U^\top,l}(\bDel)}{\U\U^\top}.
\end{align*}
Plug in the expression for $\calS_{\U\U^\top,l}(\bDel)$, and we have the following expansion under $\calF_0$:
\begin{align}\label{eq}
	&\quad \frac{1}{2}\fro{\hat\U\hat\U^\top - \U\U^\top}^2\notag\\
	&=  \sum_{l\geq 2}\sum_{\s\in\SS_l}(-1)^{\lzero{\s}}\inp{\M(s_1)\underline{\M(s_1)^\top\bDel \M(s_2)}\cdots\underline{\M(s_l)^\top\bDel \M(s_{l+1})}\M(s_{l+1})^\top}{\U\U^\top}\notag\\
	&= \sum_{l\geq 2}\sum_{\s\in\SS_l}(-1)^{\lzero{\s}}\sum_{j_1,\cdots,j_l\in[m]}w_{j_1}\cdots w_{j_l}\cdot\tr(\U^\top\M(s_1)\underline{\M(s_1)^\top\bDel_{j_1} \M(s_2)}\notag\\
	&\hspace{7cm}\cdots\underline{\M(s_l)^\top\bDel_{j_l} \M(s_{l+1})}\M(s_{l+1})^\top\U).
\end{align}
Here the first equality holds due to $\inp{\calS_{\U\U^\top,1}(\bDel)}{\U\U^\top} = 0$. 
Recall $\bDel_j = \hat\U_j\hat\U_j^\top-\U\U^\top$. 
Notice $\hat\U_j$ is the top $r$ left singular vectors of $\tilde \U_j\tilde\U_j^\top +\Z_j$, and $\U$ is the top $r$ left singular vectors of $\U\U^\top$. We denote $\D_j = \tilde \U_j\tilde\U_j^\top - \U\U^\top +\Z_j$, 
then
\begin{align}\label{Deltaj}
	\bDel_j = \hat\U_j\hat\U_j^\top-\U\U^\top = \sum_{l\geq 1}\calS_{\U\U^\top, l}(\D_j),
\end{align}
and 
\begin{align*}
	\calS_{\U\U^\top, l}(\D_j) = \sum_{\s\in\SS_l}(-1)^{\lzero{\s} + 1}\M(s_1)\cdot \underline{\M(s_1)^\top\D_{j}\M(s_2)}\cdots \underline{\M(s_l)^\top\D_j\M(s_{l+1})}\cdot \M(s_{l+1})^\top.
\end{align*}
Therefore
\begin{align}\label{bDel}
	\bDel_j =\sum_{l\geq1}\sum_{\s\in\SS_l}(-1)^{\lzero{\s} + 1}\M(s_1)\cdot \underline{\M(s_1)^\top\D_{j}\M(s_2)}\cdots \underline{\M(s_l)^\top\D_j\M(s_{l+1})}\cdot \M(s_{l+1})^\top.
\end{align}

Since $\D_j = \tilde\U_j\tilde\U_j^\top - \U\U^\top + \Z_j$, we consider the expression for $\tilde\U_j\tilde\U_j^\top - \U\U^\top$. 
Consider the event $\calE_0^{(j)} = \{\op{\hat\bSigma_j - \bSigma}\leq \lambda_{\min}/4\}$. 
Then from Lemma \ref{lemma:covariance}, we have $\PP(\calE_0^{(j)})\geq 1-e^{-p\wedge n_j}$. We denote $\bXi_j = \hat\bSigma_j - \bSigma$, then under $\calE_0^{(j)}$, we have the following expansion under the event $\calE_0^{(j)}$:
\begin{align*}
	\tilde\U_j\tilde\U_j^\top - \U\U^\top= \sum_{l\geq 1}\calS_{\U\bLa\U^\top, l}(\bXi_j),
\end{align*}
where 
\begin{align}\label{def:Sl}
	&\quad\calS_{\U\bLa\U^\top, l}(\bXi_j) \notag\\
	&= \sum_{s\in\SS_l} (-1)^{\lzero{\s} + 1}\M(s_1)\bLa^{-s_1}\underline{\M(s_1)^\top\bXi_j\M(s_2)}\bLa^{-s_2}\cdots\bLa^{-s_l}\underline{\M(s_l)^\top\bXi_j\M(s_{l+1})}\bLa^{-s_{l+1}}\M(s_{l+1})^\top,
\end{align}
here we denote $\bLa^{-0} = \I_{p-r}$ with slight abuse of notation. We denote $g_i^{(j)} = \U^\top X_i^{(j)}$ and $h_i^{(j)} = \U_{\perp}^\top X_i^{(j)}$. Then
$$g_i^{(j)} \sim N(0,\bLa + \sigma^2 \I_r), \quad h_i^{(j)}\sim N(0,\sigma^2\I_{p-r}).$$ 
We define the matrix $\G^{(j)}\in\RR^{r \times n_j}$, and $\H^{(j)}\in\RR^{(p-r)\times n_j}$ as 
\begin{align*}
	\G^{(j)} = [g_1^{(j)}, \cdots g_{n_j}^{(j)}], \quad \H^{(j)} = [h_1^{(j)}, \cdots h_{n_j}^{(j)}]. 
\end{align*}
Then, $\G^{(j)}$ and $\H^{(j)}$ are independent. We also have
\begin{align}\label{eq:MXiM}
	\M(s_1)^\top\bXi_j\M(s_2) = \begin{cases}
		\frac{1}{n_j}\H^{(j)}\H^{(j)\top} -\sigma^2\I_{d-r}, & \text{if } s_1=s_2 =0, \\
		\frac{1}{n_j}\H^{(j)}\G^{(j)\top}& \text{if } s_1=0, s_2>0, \\
		\frac{1}{n_j}\G^{(j)}\H^{(j)\top}&\text{if } s_1>0,s_2=0,\\
		\frac{1}{n_j}\G^{(j)}\G^{(j)\top} - (\bLa + \sigma^2 \I_r)&\text{if } s_1,s_2>0.
	\end{cases}
\end{align}

Notice $\Z_j$ is a Gaussian orthogonal ensemble (GOE), and thus is invariant to orthogonal conjugation. Therefore, $\U^\top\Z_j\U, \U_{\perp}^\top\Z_j\U, \U_{\perp}^\top\Z_j\U_{\perp}$ are independent. We will in the following denote 
\begin{align}\label{UZU}
	\mat{\Z_{j,1}&\Z_{j,2}\\ \Z_{j,2}^\top&\Z_{j,3}} = \mat{\U^\top\Z_j\U& \U^\top\Z_j\U_{\perp}\\ \U_{\perp}^\top\Z_j\U &  \U_{\perp}^\top\Z_j\U_{\perp}}. 
\end{align}
We recap the observations so far in the following diagram. 
\begin{figure}[H]
	\centering
	\begin{tikzpicture}
		\coordinate (A) at (-6,0) {};
		\coordinate (B) at ( 6,0) {};
		\coordinate (C) at (0,6) {};
		\draw[name path=AC] (A) -- (C);
		\draw[name path=BC] (B) -- (C);
		\coordinate (D) at (-4,2.4) {};
		\coordinate (E) at (-3.2,2.4) {};
		\draw[->] (D) -- (E);
		\foreach \y/\A in {
			0/$\G^{(j)}\quad\H^{(j)}$,
			1/$\U^\top\bXi_j\U \quad \U^\top\bXi_j\U_{\perp}\quad \U_{\perp}^\top\bXi_j\U_{\perp}$,
			2/$\U^\top\D_j\U \quad \U^\top\D_j\U_{\perp}\quad \U_{\perp}^\top\D_j\U_{\perp}$,
			3/$\U^\top\bDel_j\U \enspace \U^\top\bDel_j\U_{\perp}  \enspace\U_{\perp}^\top\bDel_j\U_{\perp}$,
			4/$\fro{\hat\U\hat\U^\top - \U\U^\top}^2$
		} {
			\path[name path=horiz] (A|-0,\y) -- (B|-0,\y);
			\draw[name intersections={of=AC and horiz,by=P},
			name intersections={of=BC and horiz,by=Q}] (P) -- (Q)
			node[midway,above,align=center,text width=
			\dimexpr(6em-\y em)*5\relax] {\A};
		}
		\node at (-5.5,2.4) {$\Z_{j,1}\enspace\Z_{j,2}\enspace \Z_{j,3}$};
	\end{tikzpicture}	
	\caption{Layer by layer decomposition of $\fro{\hat\U\hat\U^\top - \U\U^\top}^2$, with the building blocks $\{\G^{(j)},\H^{(j)}, \Z_{j,1}\enspace\Z_{j,2}\enspace \Z_{j,3}\}_{j=1}^M$}
	\label{fig:pyramid}
\end{figure}
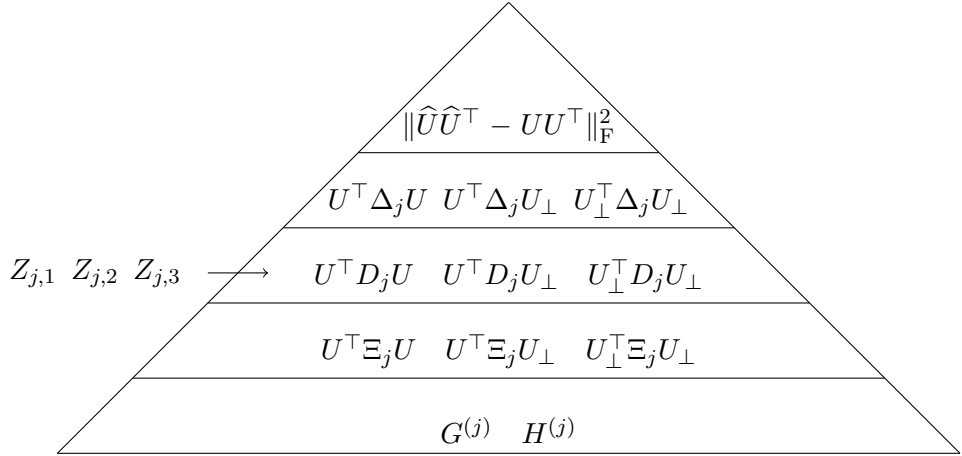

Now we analyze each terms in Figure \ref{fig:pyramid} from bottom to top. 
We first consider $\U^\top \D_j \U$:
\begin{align*}
	\U^\top \D_j \U = \U^\top (\tilde \U_j\tilde\U_j^\top - \U\U^\top )\U + \U^\top \Z_j\U. 
\end{align*}
We have
\begin{align*}
	&\quad \U^\top (\tilde \U_j\tilde\U_j^\top - \U\U^\top )\U \\
	&= \sum_{l\geq 1}\sum_{s\in\SS_l} (-1)^{\lzero{\s} + 1}\U^\top\M(s_1)\bLa^{-s_1}\underline{\M(s_1)^\top\bXi_j\M(s_2)}\bLa^{-s_2}\\
	&\hspace{8cm}\cdots\bLa^{-s_l}\underline{\M(s_l)^\top\bXi_j\M(s_{l+1})}\bLa^{-s_{l+1}}\M(s_{l+1})^\top\U.
\end{align*}
A simple fact is in each summand above, it is symmetric in both $\G^{(j)}, \H^{(j)}$. In details, we have 
\begin{align*}
	\U^\top \D_j \U = f_1(\G^{(j)}, \H^{(j)}) +\Z_{j,1},
\end{align*}
where $f_1$ is a matrix-valued function such that $f_1(\G^{(j)}, \H^{(j)}) = f_1(-\G^{(j)}, \H^{(j)}) =f_1(\G^{(j)}, -\H^{(j)})$.
And we can similarly show 
\begin{align*}
	\U_{\perp}^\top \D_j \U_{\perp} = f_3(\G^{(j)}, \H^{(j)}) +\Z_{j,3},
\end{align*}
for some $f_3$ such that $f_3(\G^{(j)}, \H^{(j)}) = f_3(-\G^{(j)}, \H^{(j)}) =f_3(\G^{(j)}, -\H^{(j)})$.

For $\U^\top \D_j \U_{\perp}$, we have 
\begin{align*}
	\U^\top \D_j \U_{\perp} = f_2(\G^{(j)}, \H^{(j)}) +\Z_{j,2},
\end{align*}
where $-f_2(\G^{(j)}, \H^{(j)}) = f_2(-\G^{(j)}, \H^{(j)}) =f_2(\G^{(j)}, -\H^{(j)})$.

Notice $\U^\top \D_j \U$, $\U^\top \D_j \U_{\perp}$, and $\U_{\perp}^\top \D_j \U_{\perp}$ are the building blocks for $\bDel_j$, we have 
\begin{align*}
	\U^\top\bDel_j\U = g_1(\G^{(j)}, \H^{(j)}, \Z_{j,1},\Z_{j,2},\Z_{j,3}),
\end{align*}
where $g_1$ is an even function in both $\G^{(j)}, \H^{(j)}, \Z_{j,2}$. 
Similarly, 
\begin{align}\label{eqg2}
	\U^\top\bDel_j\U_{\perp}^\top = g_2(\G^{(j)}, \H^{(j)}, \Z_{j,1},\Z_{j,2},\Z_{j,3}),
\end{align}
where $g_2$ is an odd function in both $\G^{(j)}, \H^{(j)}, \Z_{j,2}$, and 
\begin{align*}
	\U_{\perp}^\top\bDel_j\U_{\perp}^\top = g_3(\G^{(j)}, \H^{(j)}, \Z_{j,1},\Z_{j,2},\Z_{j,3}),
\end{align*}
where $g_3$ is an even function in both $\G^{(j)}, \H^{(j)}, \Z_{j,2}$.

\hspace{1cm}

\noindent\textit{Upper bounds for $\op{\U^\top \bXi_j \U}, \op{\U^\top \bXi_j \U_{\perp}}, \op{\U_{\perp}^\top \bXi_j \U_{\perp}}$.}
We denote the normalized versions of $g_i^{(j)}, h_i^{(j)}$ as 
\begin{align*}
	\bar g_i^{(j)} = (\bLa + \sigma^2 \I_r)^{-1/2}, \quad \bar h_i^{(j)} = \sigma^{-1}h_i^{(j)},
\end{align*}
and 
\begin{align*}
	\bar\G^{(j)} = [\bar g_1^{(j)}, \cdots \bar g_{n_j}^{(j)}], \quad \bar \H^{(j)} = [\bar h_1^{(j)}, \cdots \bar h_{n_j}^{(j)}]. 
\end{align*}
Then we have 
\begin{align*}
	\U^\top\bXi_j\U_{\perp} &= \frac{\sigma}{n_j}(\bLa^{1/2} + \sigma \I_r)\bar\G^{(j)}\bar\H^{(j)\top}.
\end{align*}
Standard $\epsilon$-net argument shows with probability exceeding $1-2e^{-p}-2e^{-n_j}$,
\begin{align*}
	\op{\U^\top\bXi_j\U_{\perp}} \lesssim (\lambda_{\max}^{1/2}  +\sigma)\sigma \sqrt{\frac{p-r}{n_j}}.
\end{align*}
For $\U^\top\bXi_j\U$, we have 
\begin{align*}
	\U^\top\bXi_j\U &= \frac{1}{n_j}(\bLa^{1/2} + \sigma \I_r)\bar\G^{(j)}\bar\G^{(j)\top}(\bLa^{1/2} + \sigma \I_r) - (\bLa + \sigma^2 \I_r).
\end{align*}
And with probability exceeding $1-2e^{-\eta_j}$,
\begin{align*}
	\op{\U^\top\bXi_j\U} \lesssim  (\lambda_{\max} +\sigma^2)\frac{\sqrt{r+\eta_j}}{\sqrt{n_j}}. 
\end{align*}
For $\U_{\perp}^\top\bXi_j\U_{\perp}$, we have
\begin{align*}
	\U_{\perp}^\top\bXi_j\U_{\perp} = \frac{\sigma^2}{n_j}\bar\H^{(j)}\bar\H^{(j)\top} -\sigma^2\I_{d-r}. 
\end{align*}
And with probability exceeding $1-2e^{-p}$, 
\begin{align*}
	\op{\U_{\perp}^\top\bXi_j\U_{\perp}} \lesssim \sigma^2\frac{\sqrt{p-r}}{\sqrt{n_j}}.
\end{align*}

We define the event 
\begin{align}\label{event1}
	\calE_1^{(j)}:=\bigg\{&\op{\U^\top\bXi_j\U_{\perp}} \lesssim (\lambda_{\max}^{1/2}  +\sigma)\sigma \sqrt{\frac{p-r}{n_j}}\bigg\}\cap \bigg\{ \op{\U^\top\bXi_j\U} \lesssim  (\lambda_{\max} +\sigma^2)\frac{\sqrt{r+\eta_j}}{\sqrt{n_j}}\bigg\}\notag\\
	&\cap\bigg\{\op{\U_{\perp}^\top\bXi_j\U_{\perp}} \lesssim \sigma^2\sqrt{\frac{p-r}{n_j}}\bigg\}. 
\end{align}
Then $\PP(\calE_1^{(j)})\geq 1 - 4e^{-p}-2e^{-n_j} - 2e^{-\eta_j}$.

\hspace{1cm}

\noindent\textit{Upper bounds for $\op{\Z_{j,1}},\op{\Z_{j,2}}$ and $\op{\Z_{j,3}}$. }
From \eqref{UZU}, we see that with probability exceeding $1-2e^{-\eta_j} - 4e^{-p}$, 
\begin{align*}
	\op{\Z_{j,1}} \lesssim \alpha_j\sqrt{r + \eta_j},
	\quad \op{\Z_{j,2}} \lesssim \alpha_j\sqrt{p},
	\quad \op{\Z_{j,3}} \lesssim \alpha_j\sqrt{p}.
\end{align*}
We define the event, 
\begin{align}\label{event2}
	\calE_2^{(j)}:=\bigg\{	\op{\Z_{j,1}} \lesssim \alpha_j\sqrt{r + \eta_j}\bigg\} \cap\bigg\{	\op{\Z_{j,2}} \lesssim \alpha_j\sqrt{p}\bigg\} \cap \bigg\{\op{\Z_{j,3}}\lesssim \alpha_j\sqrt{p}\bigg\},
\end{align}
$\calE^{(j)} = \calE_0^{(j)}\cap\calE_1^{(j)}\cap\calE_2^{(j)}$, and $\calE := \bigcap_{j=1}^m \calE^{(j)}$. Then $\PP(\calE^{(j)})\geq 1- 4e^{-\eta_j} - 10e^{-p\wedge n_j}$. 

\hspace{1cm}

\noindent\textit{Upper bounds for $\fro{\hat\U\hat\U^\top - \U\U^\top}^2\cdot \mathds{1}(\calE)$. }
Under the SNR condition, we have $\lambda_{\min}^{-1}\bigg((\lambda_{\max}^{1/2}  +\sigma)\sigma \sqrt{\frac{p}{n_j}}\bigg)\lesssim 1$, and under the event $\calE$, we have 
\begin{align*}
	\op{\U^\top \D_j \U} &\leq C^2\lambda_{\min}^{-2}\bigg((\lambda_{\max}^{1/2}  +\sigma)\sigma \sqrt{\frac{p}{n_j}}\bigg)^2\sum_{l\geq 2} 2^{-l+2} + C\alpha_j\sqrt{r+\eta_j}\\
	&\lesssim \lambda_{\min}^{-2}(\lambda_{\max}  +\sigma^2)\sigma^2 \frac{p}{n_j} + \alpha_j\sqrt{r+\eta_j},
\end{align*}
and
\begin{align*}
	\op{\U^\top \D_j \U_{\perp}} &\leq  C\lambda_{\min}^{-1}\bigg((\lambda_{\max}^{1/2}  +\sigma)\sigma \sqrt{\frac{p}{n_j}}\bigg)\sum_{l\geq 1} 2^{-l+1}+ C\alpha_j\sqrt{p}\\
	&\lesssim \lambda_{\min}^{-1}\bigg((\lambda_{\max}^{1/2}  +\sigma)\sigma \sqrt{\frac{p}{n_j}}\bigg)+\alpha_j\sqrt{p},
\end{align*}
and similarly, 
\begin{align*}
	\op{\U_{\perp}^\top \D_j \U_{\perp}} &\leq C^2\lambda_{\min}^{-2}\bigg((\lambda_{\max}^{1/2}  +\sigma)\sigma \sqrt{\frac{p}{n_j}}\bigg)^2\sum_{l\geq 2} 2^{-l+2} + C\alpha_j\sqrt{p}\\
	&\lesssim \lambda_{\min}^{-2}(\lambda_{\max}  +\sigma^2)\sigma^2 \frac{p}{n_j} + \alpha_j\sqrt{p}. 
\end{align*}
As long as $\eta_j+r\leq p$ and since $\alpha_j\sqrt{p}\lesssim 1$, we have 
\begin{align}\label{UDU}
	\max\bigg\{\op{\U^\top \D_j \U}, \op{\U^\top \D_j \U_{\perp}}, \op{\U_{\perp}^\top \D_j \U_{\perp}}\bigg\} \leq \lambda_{\min}^{-1}\bigg((\lambda_{\max}^{1/2}  +\sigma)\sigma \sqrt{\frac{p}{n_j}}\bigg)+\alpha_j\sqrt{p}. 
\end{align}
We denote the right hand side bound $u_j = \lambda_{\min}^{-1}(\lambda_{\max}^{1/2}  +\sigma)\sigma \sqrt{\frac{p}{n_j}}+\alpha_j\sqrt{p}$.
Notice under the given SNR, $\max_j u_j\leq\frac{1}{4}$. 
Together with \eqref{Deltaj}, we conclude $\op{\hat U_j\hat U_j^\top - UU^\top}\leq \frac{1}{4}$. 
This also implies $\op{\sum_{j=1}^m w_j\hat\U_j\hat\U_j^\top - \U\U^\top}\leq 1/4$. That is, $\calE$ implies $\calF_0$. 

For the terms related to $\bDel_j$, we have 
\begin{align*}
	\op{\U^\top \bDel_j \U} &\leq \sum_{l\geq2} C^l \bigg(\lambda_{\min}^{-1}\bigg((\lambda_{\max}^{1/2}  +\sigma)\sigma \sqrt{\frac{p}{n_j}}\bigg)+\alpha_j\sqrt{p}\bigg)^l\\
	&\lesssim \lambda_{\min}^{-2}(\lambda_{\max}+\sigma^2)\sigma^2 \frac{p}{n_j}+\alpha_j^2p,\\
	\op{\U^\top \bDel_j \U_{\perp}} &\lesssim \lambda_{\min}^{-1}(\lambda_{\max}^{1/2}  +\sigma)\sigma \sqrt{\frac{p}{n_j}}+\alpha_j\sqrt{p}, \\
	\op{\U_{\perp}^\top \bDel_j \U_{\perp}}&\lesssim \lambda_{\min}^{-2}(\lambda_{\max}+\sigma^2)\sigma^2 \frac{p}{n_j}+\alpha_j^2p. 
\end{align*}
Then $u_j$ is also the upper bound for $\op{\U^\top \bDel_j \U}$, $\op{\U^\top \bDel_j \U_{\perp}}$, and $\op{\U_{\perp}^\top \bDel_j \U_{\perp}}$ under $\calE$.  

Now we go back to \eqref{eq}. 
For each $\s\in\SS_l$, in order for the summand in \eqref{eq} to be non-zero, $s_1,s_{l+1}$ should be strictly greater than 0. Since $s_1+\cdots + s_{l+1} =l$, there exists $1\leq i_1<i_2\leq l$, such that 
$s_{i_1}>0, s_{i_1+1}=0$, and $s_{i_2}=0,s_{i_2+1}>0$. We define
\begin{align*}
	\II_1(\s) = \bigg\{\j\in[m]^l: j_{i_1}\neq j_{i_2}, \{j_{i_1},j_{i_2}\}\cap \{j_1,\cdots, \bar{j_{i_1}},\cdots, \bar{j_{i_2}},\cdots, j_{l}\} = \emptyset\bigg\}.
\end{align*}
Here $\bar\cdot$ means $\cdot$ is absent in the set. Then $|\II_1(\s)| = m(m-1)(m-2)^{l-2}$. 
We define the complement of $\II_1(\s)$ as $\II_2(\s) = [m]^l\backslash \II_1(\s)$. Then 
\begin{align}\label{inclusion}
	\II_2(\s) \subset \{\j\in[m]^l: j_{i_1}=j_{i_2}\} \bigcup \big(\cup_{k \neq i_1,i_2}\{\j\in[m]^l: j_{i_1}=j_{k}\}\big) \bigcup \big(\cup_{k \neq i_1,i_2}\{\j\in[m]^l: j_{i_2}=j_{k}\}\big).
\end{align}

Next we consider the upper bound for 
\begin{align*}
	\bigg|\EE\sum_{j_1,\cdots,j_l\in[m]}w_{j_1}\cdots w_{j_l} \tr(\U^\top\M(s_1)\underline{\M(s_1)^\top\bDel_{j_1} \M(s_2)}\cdots\underline{\M(s_l)^\top\bDel_{j_l} \M(s_{l+1})}\M(s_{l+1})^\top\U)\cdot\mathds{1}(\calE)\bigg|.
\end{align*}
We can split the above sum into two parts, namely 
\begin{align*}
	\sum_{j_1,\cdots,j_l\in[m]} = \sum_{\II_1(\s)} + \sum_{\II_2(\s)}. 
\end{align*}
Notice from \eqref{eqg2}, we have $\EE U^\top\bDel_{j_{i_1}}U_{\perp}\cdot \mathds{1}(\calE^{(j_{i_1})}) = 0$.
Then for all $\j\in\II_1(\s)$, we have 
\begin{align*}
	\EE \tr(\U^\top\M(s_1)\underline{\M(s_1)^\top\bDel_{j_1} \M(s_2)}\cdots\underline{\M(s_l)^\top\bDel_{j_l} \M(s_{l+1})}\M(s_{l+1})^\top\U)\cdot \mathds{1}(\calE) = 0.
\end{align*}
For each summand with index $\j\in\II_2(\s)$, using Cauchy-Schwarz inequality, we have 
\begin{align*}
	|\tr(\U^\top\M(s_1)\underline{\M(s_1)^\top\bDel_{j_1} \M(s_2)}\cdots\underline{\M(s_l)^\top\bDel_{j_l} \M(s_{l+1})}\M(s_{l+1})^\top\U)|\leq r\cdot u_{j_1}\cdots u_{j_l}. 
\end{align*}
Using these facts, we have
\begin{align*}
	&\quad\bigg|\EE\sum_{j_1,\cdots,j_l\in[m]}w_{j_1}\cdots w_{j_l} \tr(\U^\top\M(s_1)\underline{\M(s_1)^\top\bDel_{j_1} \M(s_2)}\cdots\underline{\M(s_l)^\top\bDel_{j_l} \M(s_{l+1})}\M(s_{l+1})^\top\U)\cdot\mathds{1}(\calE)\bigg|\\
	&\leq r\cdot \sum_{\j\in\II_2(\s)}u_{j_1}\cdots u_{j_l}. 
\end{align*}
Using the inclusion relation in \eqref{inclusion}, this is further upper bounded by 
\begin{align}\label{I2}
	r\cdot \bigg(\sum_{j_{i_1}=j_{i_2}} + \sum_{k\neq i_1,i_2}\sum_{j_{i_1} = j_k} + \sum_{k\neq i_1,i_2}\sum_{j_{i_2} = j_k} \bigg)u_{j_1}\cdots u_{j_l} \leq 2lr(\sum_{k=1}^mw_k^2u_k^2)(\sum_{k=1}^mw_ku_k)^{l-2}. 
\end{align}
Therefore we have 
\begin{align}\label{expectation:upper}
	\EE\fro{\hat\U\hat\U^\top - \U\U^\top}^2\cdot\mathds{1}(\calE) &\leq \sum_{l\geq 2}4^l\cdot 2lr(\sum_{k=1}^mw_k^2u_k^2)(\sum_{k=1}^mw_ku_k)^{l-2} \leq  4r(\sum_{k=1}^mw_k^2u_k^2),
\end{align}
where the last inequality is due to $u_k\leq \frac{1}{2}$. Finally, we set $w_k$ to be such that $\sum_{k=1}^mw_k^2u_k^2$ is minimized, that is $w_k\propto u_k^{-2}$.

On the other hand, we can set $\eta_j = c_0(n_j\wedge p)$, then we have 
\begin{align*}
	\EE\fro{\hat\U\hat\U^\top - \U\U^\top}^2\cdot\mathds{1}(\calE^c)\leq  2r \cdot \PP(\calE^c) \leq 28r \sum_{j=1}^m  e^{-c_0(n_j\wedge p)}. 
\end{align*}
In summary, we have 
\begin{align*}
	\EE\fro{\hat\U\hat\U^\top - \U\U^\top}^2 \leq  \frac{4r}{\sum_{j=1}^m u_j^{-2}} + 28r \sum_{j=1}^m e^{-c_0(n_j\wedge p)}. 
\end{align*}

\hspace{1cm}

\noindent\textbf{High probability upper bound. }
Recall the event $\calE_1^{(j)},\calE_2^{(j)}$ defined respectively in \eqref{event1} and \eqref{event2}. Moreover, we define
\begin{align*}
	\calE^{(j)}_3:= \bigg\{\op{\bar\H^{(j)}}\lesssim \sqrt{n_j\vee p}\bigg\} \cap \bigg\{\op{\bar\G^{(j)}}\lesssim \sqrt{p}\bigg\}.
\end{align*}
Then $\PP(\calE^{(j)}_3) \geq 1-2e^{-p}-2e^{-n_j\vee p}$. We denote $\calF^{(j)} = \calE_1^{(j)}\cap\calE_2^{(j)}\cap \calE_3^{(j)}$ and $\calF = \cap_{j=1}^m\calF^{(j)}$.

We define the function $\phi(s;t_0)$ for given $t_0>0$ as
\begin{align}\label{phi}
	\phi(s;t_0) = \begin{cases}
		&1, \quad s\leq t_0\\
		&2-\frac{s}{t_0},\quad t_0<s\leq 2t_0\\
		&0, \quad s\geq 2t_0. 
	\end{cases}
\end{align}
Also define 
\begin{align*}
	\psi_j(\bar\G^{(j)}, \bar\H^{(j)}):&=\phi(\op{\bar\H^{(j)}}; \sqrt{n_j\vee p})\cdot \phi(\op{\bar\H^{(j)}\bar\H^{(j)\top} - n_j\I}; \sqrt{pn_j})\cdot \phi(\op{\bar\G^{(j)}\bar\H^{(j)\top}}; \sqrt{n_jp})\\
	&\quad\cdot \mathds{1}(\op{\bar\G^{(j)}}\leq \sqrt{n_j})\cdot\mathds{1}(\op{\bar\G^{(j)}\bar\G^{(j)\top}-n_j\I}\leq \sqrt{(r+\eta_j)n_j}). 
\end{align*}
Then we have 
\begin{align*}
	\big|\fro{\U^\top\bXi_j\U}\cdot \psi_j(\bar\G^{(j)}, \bar\H^{(j)})\leq \frac{\sqrt{r+\eta_j}}{\sqrt{n_j}}(\lambda_{\max}+\sigma^2).
\end{align*}

In the following, we shall condition on $\bar\G^{(j)}$. In order to compute the Lipschitz constant, we denote 
\begin{align*}
	\bXi_j' = \mat{\U & \U_{\perp}}\mat{\frac{1}{n_j}\G^{(j)}\G^{(j)\top} - (\bLa + \sigma^2\I) & \frac{1}{n_j}\G^{(j)}\H^{(j)'\top}\\\frac{1}{n_j}\H^{(j)'}\G^{(j)\top}&\frac{1}{n_j}\H^{(j)'}\H^{(j)'\top} - \sigma^2\I }\mat{\U^\top\\\U_{\perp}^\top}
\end{align*}
Then we have 
\begin{align*}
	&\quad\big|\fro{\U_{\perp}^\top\bXi_j\U}\cdot \psi_j(\bar\G^{(j)}, \bar\H^{(j)}) - \fro{\U_{\perp}^\top\bXi_j'\U}\cdot \psi_j(\bar\G^{(j)}, \bar\H^{(j)'}) \big| \\
	&\leq \frac{1}{\sqrt{n_j}}\cdot\sigma(\lambda_{\max}^{1/2}+\sigma)\cdot \big\|\bar\H^{(j)}- \bar\H^{(j)'}\big\|_{\rm F}.
\end{align*}
And 
\begin{align*}
	&\quad\big|\fro{\U_{\perp}^\top\bXi_j\U_{\perp}}\cdot \psi_j(\bar\G^{(j)}, \bar\H^{(j)}) - \fro{\U_{\perp}^\top\bXi_j'\U_{\perp}}\cdot \psi_j(\bar\G^{(j)}, \bar\H^{(j)'})\big| \\ 
	&\leq \frac{2\sigma^2}{n_j}\cdot\sqrt{n_j\vee p}\cdot\fro{\bar\H^{(j)} -\bar\H^{(j)'}}. 
\end{align*}
Under the given SNR condition, we have 
\begin{align*}
	\frac{1}{\sqrt{n_j}}\cdot\sigma(\lambda_{\max}^{1/2}+\sigma) \geq \frac{2\sigma^2}{n_j}\cdot\sqrt{n_j\vee p}. 
\end{align*}
Next we analyze $\U^\top\D_j\U, \U_{\perp}^\top\D_j\U$ and $\U_{\perp}^\top\D_j\U_{\perp}$. Recall $\D_j = \tilde\U_j\tilde\U_j^\top - \U\U^\top + \Z_j$, and 
\begin{align*}
	&\quad\tilde\U_j\tilde\U_j^\top - \U\U^\top \\
	&= \sum_{l\geq 1}\sum_{s\in\SS_l} (-1)^{\lzero{\s} + 1}\M(s_1)\bLa^{-s_1}\underline{\M(s_1)^\top\bXi_j\M(s_2)}\bLa^{-s_2}\cdots\bLa^{-s_l}\underline{\M(s_l)^\top\bXi_j\M(s_{l+1})}\bLa^{-s_{l+1}}\M(s_{l+1})^\top,
\end{align*}
Now for each $l, \s\in\SS_l$, we have 
\begin{align*}
	&\bigg\| \M(s_1)\bLa^{-s_1}\underline{\M(s_1)^\top\bXi_j\M(s_2)}\bLa^{-s_2}\cdots\bLa^{-s_l}\underline{\M(s_l)^\top\bXi_j\M(s_{l+1})}\bLa^{-s_{l+1}}\M(s_{l+1})^\top\cdot\psi_j(\bar\G^{(j)}, \bar\H^{(j)})\\
	&-\M(s_1)\bLa^{-s_1}\underline{\M(s_1)^\top\bXi_j'\M(s_2)}\bLa^{-s_2}\cdots\bLa^{-s_l}\underline{\M(s_l)^\top\bXi_j'\M(s_{l+1})}\bLa^{-s_{l+1}}\M(s_{l+1})^\top\cdot\psi_j(\bar\G^{(j)}, \bar\H^{(j)'})\bigg\|_{\rm F}\\
	&\leq l\cdot\frac{1}{8^{l-1}}\cdot\lambda_{\min}^{-1}\cdot \sqrt{\frac{1}{n_j}}\sigma(\lambda_{\max}^{1/2}+\sigma)\cdot\fro{\bar\H^{(j)} - \bar\H^{(j)'}}
\end{align*}
Therefore we conclude 
\begin{align*}
	\fro{\U^\top\D_j\U\cdot\psi_j(\bar\G^{(j)}, \bar\H^{(j)}) &- \U^\top\D_j'\U\cdot\psi_j(\bar\G^{(j)}, \bar\H^{(j)'})}\\
	&\lesssim \lambda_{\min}^{-1}\cdot \sqrt{\frac{1}{n_j}}\sigma(\lambda_{\max}^{1/2}+\sigma)\cdot\fro{\bar\H^{(j)} - \bar\H^{(j)'}}  + \alpha_j\fro{\bar\Z_{j,1} -\bar\Z_{j,1}'}, \\
	\fro{\U_{\perp}^\top\D_j\U\cdot\psi_j(\bar\G^{(j)}, \bar\H^{(j)}) &- \U_{\perp}^\top\D_j'\U\cdot\psi_j(\bar\G^{(j)}, \bar\H^{(j)'})}\\
	&\lesssim \lambda_{\min}^{-1}\cdot \sqrt{\frac{1}{n_j}}\sigma(\lambda_{\max}^{1/2}+\sigma)\cdot\fro{\bar\H^{(j)} - \bar\H^{(j)'}}  + \alpha_j\fro{\bar\Z_{j,2} -\bar\Z_{j,2}'}, \\
	\fro{\U_{\perp}^\top\D_j\U_{\perp}\cdot\psi_j(\bar\G^{(j)}, \bar\H^{(j)}) &- \U_{\perp}^\top\D_j'\U_{\perp}\cdot\psi_j(\bar\G^{(j)}, \bar\H^{(j)'})}\\
	&\lesssim \lambda_{\min}^{-1}\cdot \sqrt{\frac{1}{n_j}}\sigma(\lambda_{\max}^{1/2}+\sigma)\cdot\fro{\bar\H^{(j)} - \bar\H^{(j)'}}  + \alpha_j\fro{\bar\Z_{j,3} -\bar\Z_{j,3}'}.
\end{align*}

Recall
\begin{align*}
	\bDel_j =\sum_{l\geq1}\sum_{\s\in\SS_l}(-1)^{\lzero{\s} + 1}\M(s_1)\cdot \underline{\M(s_1)^\top\D_{j}\M(s_2)}\cdots \underline{\M(s_l)^\top\D_j\M(s_{l+1})}\cdot \M(s_{l+1})^\top.
\end{align*}
We also define 
\begin{align*}
	\tilde\psi_j(\bar\G^{(j)}, \bar\H^{(j)},\bar\Z_{j,1},\bar\Z_{j,2},\bar\Z_{j,3}) := \psi_j(\bar\G^{(j)}, \bar\H^{(j)})\cdot \phi(\op{\bar\Z_{j,1}}; \sqrt{r+\eta_j})\cdot\phi(\op{\bar\Z_{j,2}}; \sqrt{p})\cdot\phi(\op{\bar\Z_{j,3}}; \sqrt{p}). 
\end{align*}

And we have conditioning on $\bar\G^{(j)}$, for arbitrary given $\fro{\M}\leq 1$, the function $\tr(\U_{\perp}^\top\bDel_j\U\M)\cdot\tilde\psi_j(\bar\G^{(j)}, \bar\H^{(j)},\bar\Z_{j,1},\bar\Z_{j,2},\bar\Z_{j,3})$ is 
\begin{align*}
	\lambda_{\min}^{-1}\cdot \sqrt{\frac{1}{n_j}}\sigma(\lambda_{\max}^{1/2}+\sigma) + \alpha_j
\end{align*}
Lipschitz. 
Using Gaussian concentration theorem, this indicates 
\begin{align}\label{psi2bound:Lj}
	\psitwo{\vec(\U_{\perp}^\top\bDel_j\U)\cdot\tilde\psi_j(\bar\G^{(j)}, \bar\H^{(j)},\bar\Z_{j,1},\bar\Z_{j,2},\bar\Z_{j,3})} \leq	\underbrace{C\lambda_{\min}^{-1}\cdot \sqrt{\frac{1}{n_j}}\sigma(\lambda_{\max}^{1/2}+\sigma) + C\alpha_j}_{=:L_j}. 
\end{align}
For notation simplicity, we collect $\g_j:= [\vec(\bar \G^{(j)})^\top, \vec(\bar\H^{(j)})^\top, \vec(\bar \Z_{j,1})^\top, \vec(\bar \Z_{j,2})^\top, \vec(\bar \Z_{j,3})^\top]^\top$. Then $\g_j\sim N(0,\I)$. 
We shall define two matrix-valued functions $f_1,f_2$ as 
\begin{align*}
	f_1(\g_j) &= \U^\top\bDel_j\U_{\perp}\cdot\tilde\psi_j(\g_j), \\
	f_2(\g_{j_2},\cdots,\g_{j_{l-1}}) &= \underline{\U_{\perp}^\top\bDel_{j_2} \M(s_3)}\cdots\underline{\M(s_{l-1})^\top\bDel_{j_{l-1}} \U_{\perp}}\cdot\prod_{u=2}^{l-1}\tilde\psi_{j_u}(\g_{j_u}). 
\end{align*}
Then it boils down to estimating 
\begin{align*}
	\sum_{\j\in\II_1} w_{j_1}\cdots w_{j_l}\cdot \tr\big(f_1(\g_{j_1})\cdot f_2(\g_{j_2},\cdots,\g_{j_{l-1}})\cdot f_1(\g_{j_l})^\top\big)
\end{align*}
We define a projection map: $\pi_2^{l-1}: (j_1,\cdots,j_l)\mapsto (j_2,\cdots,j_{l-1})$. 
And we denote $\pi_2^{l-1}(\II_1)$ the image of $\pi_2^{l-1}$ applied to $\II_1$. 
And then we can rewrite $\sum_{\j\in\II_1}$ as 
\begin{align*}
	\sum_{\j\in\II_1} = \sum_{(j_2,\cdots, j_{l-1})\in\pi_2^{l-1}(\II_1)}\sum_{\substack{j_1\neq j_l\\ \{j_1,j_l\}\cap\{j_2,\cdots, j_{l-1}\} = \emptyset}}.
\end{align*}
Next, we shall fix $(j_2,\cdots, j_{l-1})$, and use the decoupling to derive the upper bound. For given $(j_2,\cdots, j_{l-1})$, condition on $\g_{j_2},\cdots, \g_{j_{l-1}}$, using the decoupling technique (e.g. \cite{de2012decoupling}), we have
\begin{align*}
	&\quad\PP\bigg(\big|\sum_{\substack{j_1\neq j_l\\ \{j_1,j_l\}\cap\{j_2,\cdots, j_{l-1}\} = \emptyset}}w_{j_1}\cdots w_{j_l} \tr\big(f_1(\g_{j_1})\cdot f_2(\g_{j_2},\cdots,\g_{j_{l-1}})\cdot f_1(\g_{j_l})^\top\big)\big| \geq t\bigg| \g_{j_2},\cdots,\g_{j_{l-1}}\bigg)\\
	&\leq C\PP\bigg(C\big|\sum_{\substack{j_1\neq j_l\\ \{j_1,j_l\}\cap\{j_2,\cdots, j_{l-1}\} = \emptyset}}w_{j_1}\cdots w_{j_l} \tr\big(f_1(\g_{j_1})\cdot f_2(\g_{j_2},\cdots,\g_{j_{l-1}})\cdot f_1(\g'_{j_l})^\top\big)\big| \geq t\bigg| \g_{j_2},\cdots,\g_{j_{l-1}}\bigg),
\end{align*}
for some absolute constant $C>0$, where $\g'_j$ is an i.i.d. copy of $\g_j$. Notice 
\begin{align*}
	&\quad\sum_{\substack{j_1\neq j_l\\ \{j_1,j_l\}\cap\{j_2,\cdots, j_{l-1}\} = \emptyset}}w_{j_1}\cdots w_{j_l}\cdot\tr\big(f_1(\g_{j_1})\cdot f_2(\g_{j_2},\cdots,\g_{j_{l-1}})\cdot f_1(\g'_{j_l})^\top\big)\\
	&= \tr\bigg(\big(\sum_{j_1\in[m]\backslash\{j_2,\cdots, j_{l-1}\}}w_{j_1}f_1(\g_{j_1})\big)\cdot w_{j_2}\cdots w_{j_{l-1}}f_2(\g_{j_2},\cdots,\g_{j_{l-1}})\cdot \big(\sum_{j_l\in[m]\backslash\{j_2,\cdots, j_{l-1}\}}w_{j_l}f_1(\g'_{j_l})\big)^\top\bigg)\\
	&\quad -\sum_{j_1\in[m]\backslash\{j_2,\cdots, j_{l-1}\}}w_{j_1}^2w_{j_2}\cdots w_{j_{l-1}}\tr\bigg(f_1(\g_{j_1})\cdot f_2(\g_{j_2},\cdots,\g_{j_{l-1}})\cdot f_1(\g'_{j_1})^\top\bigg).
\end{align*}
For the first term above, due to \eqref{psi2bound:Lj}, we have
\begin{align*}
	\psitwo{\sum_{j_1\in[m]\backslash\{j_2,\cdots, j_{l-1}\}}w_{j_1}\vec\big(f_1(\g_{j_1})\big)}^2 &= \psitwo{\sum_{j_l\in[m]\backslash\{j_2,\cdots, j_{l-1}\}}w_{j_1}\vec\big(f_1(\g_{j_l}')\big)}^2 \\
	&\leq \sum_{j\in[m]\backslash\{j_2,\cdots, j_{l-1}\}} w_j^2L_j^2 \leq \sum_{j\in[m]}w_j^2L_j^2
\end{align*}
Therefore we have 
\begin{align*}
	&\bigg|\tr\bigg(\big(\sum_{j_1\in[m]\backslash\{j_2,\cdots, j_{l-1}\}}w_{j_1}f_1(\g_{j_1})\big)\cdot w_{j_2}\cdots w_{j_{l-1}}f_2(\g_{j_2},\cdots,\g_{j_{l-1}})\cdot \big(\sum_{j_l\in[m]\backslash\{j_2,\cdots, j_{l-1}\}}w_{j_l}f_1(\g'_{j_l})\big)^\top\bigg)\bigg| \Bigg| \big\{\g_j\big\}_{j=1}^m\\
	&\leq \fro{\big(\sum_{j_1\in[m]\backslash\{j_2,\cdots, j_{l-1}\}}w_{j_1}f_1(\g_{j_1})\big)\cdot w_{j_2}\cdots w_{j_{l-1}}f_2(\g_{j_2},\cdots,\g_{j_{l-1}})} \cdot (\sum_j w_j^2L_j^2)^{1/2}\cdot s_{l,1}
\end{align*}
holds with probability exceeding $1-e^{-s_{l,1}^2}$. 
Using Lemma \ref{lemma:l2-norm-concentration}, we have 
\begin{align*}
	&\quad \fro{\big(\sum_{j_1\in[m]\backslash\{j_2,\cdots, j_{l-1}\}}w_{j_1}f_1(\g_{j_1})\big)\cdot w_{j_2}\cdots w_{j_{l-1}}f_2(\g_{j_2},\cdots,\g_{j_{l-1}})}\\
	&\leq w_{j_2}\cdots w_{j_{l-1}}\op{f_2(\g_{j_2},\cdots,\g_{j_{l-1}})} \cdot \fro{\sum_{j_1\in[m]\backslash\{j_2,\cdots, j_{l-1}\}}w_{j_1}f_1(\g_{j_1})}\\
	&\leq w_{j_2}u_{j_2}\cdots w_{j_{l-1}}u_{j_{l-1}} \cdot (\sum_j w_j^2L_j^2)^{1/2}\cdot (s_{l,2} + \sqrt{pr}),
\end{align*}
with probability exceeding $1-e^{-s_{l,2}^2}$. Taking union bound, we conclude 
\begin{align*}
	&\quad\bigg|\tr\bigg(\big(\sum_{j_1\in[m]\backslash\{j_2,\cdots, j_{l-1}\}}w_{j_1}f_1(\g_{j_1})\big)\cdot w_{j_2}\cdots w_{j_{l-1}}f_2(\g_{j_2},\cdots,\g_{j_{l-1}})\cdot \big(\sum_{j_l\in[m]\backslash\{j_2,\cdots, j_{l-1}\}}w_{j_l}f_1(\g'_{j_l})\big)^\top\bigg)\bigg| \\
	&\leq w_{j_2}u_{j_2}\cdots w_{j_{l-1}}u_{j_{l-1}}\sum_jw_j^2L_j^2\cdot s_{l,1}(s_{l,2}+\sqrt{pr})
\end{align*}
with probability exceeding $1-M^{l-2}(e^{-s_{l,1}^2}+e^{-s_{l,2}^2})$. 
This leads to 
\begin{align*}
	&\quad \bigg|\sum_{\substack{j_1\neq j_l\\ \{j_1,j_l\}\cap\{j_2,\cdots, j_{l-1}\} = \emptyset}}w_{j_1}\cdots w_{j_l}\cdot\tr\big(f_1(\g_{j_1})\cdot f_2(\g_{j_2},\cdots,\g_{j_{l-1}})\cdot f_1(\g'_{j_l})^\top\big)\bigg|\\
	&\leq w_{j_2}u_{j_2}\cdots w_{j_{l-1}}u_{j_{l-1}} \cdot (\sum_j w_j^2L_j^2)^{1/2}\cdot (s_{l,2} + \sqrt{pr}) + r\cdot\sum_{j_1\in[m]\backslash\{j_2,\cdots, j_{l-1}\}}(w_{j_1}u_{j_1})^2w_{j_2}u_{j_2}\cdots w_{j_{l-1}}u_{j_{l-1}}. 
\end{align*}

In conclusion, we have 
\begin{align*}
	&\quad\bigg|\sum_{\j\in\II_1} w_{j_1}\cdots w_{j_l}\tr\big(f_1(\g_{j_1})\cdot f_2(\g_{j_2},\cdots,\g_{j_{l-1}})\cdot f_1(\g_{j_l})^\top\big)\bigg|\\
	&\leq(\sum_{j=1}^m w_ju_j)^{l-2}\cdot  \sum_j w_j^2L_j^2 \cdot s_1(s_2+\sqrt{pr})+ r\cdot \sum_{\j\in\II_2}w_{j_1}u_{j_1}\cdots w_{j_l}u_{j_l}\\
	&\leq \frac{1}{8^{l-2}}\sum_j w_j^2L_j^2 \cdot s_{l,1}(s_{l,2}+\sqrt{pr}) + r\cdot \sum_{\j\in\II_2}w_{j_1}u_{j_1}\cdots w_{j_l}u_{j_l}. 
\end{align*}
Now we sum up over all $\s\in\SS_l$ and $l\geq 2$, and we set $s_{l,1} = s_{l,2}= \max\{\tau_0, \sqrt{2l\cdot\log (4m)}\}$ for some $\tau_0\leq \sqrt{pr}$ to be chosen later, and we get with failure probability 
\begin{align*}
	&\quad\sum_{l\geq 2} 2\cdot 4^lm^{l-2}\cdot \exp(-\max\{\tau_0^2, l\log m\})\\
	&\leq \sum_{l=2}^{\lceil \tau_0^2/2\log(4m)\rceil}(4m)^l \cdot e^{-\tau_0^2} + \sum_{l\geq \lceil \tau_0^2/2\log(4m)\rceil +1}(4m)^{-l}\\
	&\leq 4e^{-\tau_0^2/2},
\end{align*}
the following holds:
\begin{align*}
	&\quad\bigg|\sum_{\j\in\II_1} w_{j_1}\cdots w_{j_l}\tr\big(f_1(\g_{j_1})\cdot f_2(\g_{j_2},\cdots,\g_{j_{l-1}})\cdot f_1(\g_{j_l})^\top\big)\bigg|\\
	&\leq \frac{1}{8^{l-2}}\sum_j w_j^2L_j^2 \cdot \max\{\tau_0, \sqrt{2l\cdot\log (4m)}\}(\max\{\tau_0, \sqrt{2l\cdot\log (4m)}\}+\sqrt{pr}) \\
	&\quad+ r\cdot \sum_{\j\in\II_2}w_{j_1}u_{j_1}\cdots w_{j_l}u_{j_l}\\
	&\leq \frac{1}{8^{l-2}}\sum_j w_j^2L_j^2 \cdot \max\{\tau_0, \sqrt{2l\cdot\log (4m)}\}(\max\{\tau_0, \sqrt{2l\cdot\log (4m)}\}+\sqrt{pr}) \\
	&\quad+ 2lr\cdot (\sum_{k=1}^mw_k^2u_k^2)(\sum_{k=1}^mw_ku_k)^{l-2},
\end{align*}
where the last line is due to \eqref{I2}. 
Using this, we have 
\begin{align*}
	&\quad\bigg|\sum_{l\geq 2}\sum_{\s\in\SS_l}(-1)^{\lzero{\s}+1} \sum_{\j\in[m]^l}w_{j_1}\cdots w_{j_l}\tr\big(f_1(\g_{j_1})\cdot f_2(\g_{j_2},\cdots,\g_{j_{l-1}})\cdot f_1(\g_{j_l})^\top\big)\bigg|\\
	&\leq \bigg|\sum_{l\geq 2}\sum_{\s\in\SS_l}(-1)^{\lzero{\s}+1} \sum_{\j\in\II_1(\s)}w_{j_1}\cdots w_{j_l}\tr\big(f_1(\g_{j_1})\cdot f_2(\g_{j_2},\cdots,\g_{j_{l-1}})\cdot f_1(\g_{j_l})^\top\big)\bigg|\\
	&\quad + \bigg|\sum_{l\geq 2}\sum_{\s\in\SS_l}(-1)^{\lzero{\s}+1} \sum_{\j\in\II_2(\s)}w_{j_1}\cdots w_{j_l}\tr\big(f_1(\g_{j_1})\cdot f_2(\g_{j_2},\cdots,\g_{j_{l-1}})\cdot f_1(\g_{j_l})^\top\big)\bigg|.
\end{align*}
The first term above can be bounded by 
\begin{align*}
	&\quad\sum_{j}w_j^2L_j^2\cdot \bigg(\sum_{l=2}^{\lceil \tau_0^2/2\log(4m)\rceil}\tau_0\sqrt{pr}\cdot 2^{-l} + \sum_{l\geq \lceil \tau_0^2/2\log(4m)\rceil +1}2^{-l}(2l\log(4m) + \sqrt{2lpr \log (4m)})\bigg)\\
	&\quad+\sum_{l\geq 2}4^l\cdot 2lr(\sum_{k=1}^mw_k^2u_k^2)(\sum_{k=1}^mw_ku_k)^{l-2}\\
	&\leq C\sum_{j}w_j^2L_j^2\bigg(\tau_0\sqrt{pr} + \log(4m) + \sqrt{pr\log(4m)}\bigg) + Cr(\sum_{k=1}^mw_k^2u_k^2)\\
	&\leq C\tau_0\sqrt{pr}\sum_{j}w_j^2L_j^2 + Cr(\sum_{k=1}^mw_k^2u_k^2),
\end{align*}
if we set $\tau_0 = c_0\sqrt{p}\geq \log (4m)$. Here the first inequality is due to \eqref{expectation:upper}.

In summary, 
\begin{align*}
	&\quad \bigg|\sum_{l\geq 2}\sum_{\s\in\SS_l}(-1)^{\lzero{\s}+1} \sum_{\j\in[m]^l}w_{j_1}\cdots w_{j_l}\tr\big(f_1(\g_{j_1})\cdot f_2(\g_{j_2},\cdots,\g_{j_{l-1}})\cdot f_1(\g_{j_l})^\top\big)\bigg|\\
	&\leq C\tau_0\sqrt{pr}\sum_{j=1}^mw_j^2L_j^2 + Cr(\sum_{j=1}^mw_j^2u_j^2).
\end{align*}
with probability $1 - 22\sum_{j=1}^m e^{-c_0(n_j\wedge p)}$. By setting $w_j = \frac{u_j^{-2}}{\sum_j u_j^{-2}}$, we obtain 
\begin{align*}
	\fro{\hat\U\hat\U^\top - \U\U^\top}^2 &\leq Cr (\sum_j u_j^{-2})^{-1} = \frac{Cpr}{\sum_j\big(\lambda_{\min}^{-1}(\lambda_{\max}^{1/2}  +\sigma)\sigma \sqrt{\frac{1}{n_j}}+\alpha_j\big)^{-2}}\\
	&\leq \frac{Cpr}{\sum_j \Big(n_j \wedge \big(n_j^2\epsilon_j^2\log^{-1}(\frac{2.5}{\delta_j})p^{-1}(r+\log n_j)^{-1}\big)\Big)}\cdot \frac{\sigma^2}{\lambda_{\min}}(\lambda_{\max}/\lambda_{\min} + \frac{\sigma^2}{\lambda_{\min}}). 
\end{align*}
Finally since $\lambda_{\max}\asymp\lambda_{\min}\asymp\lambda$, we finish the proof.

\subsection{Proof of Theorem \ref{thm:covariance:upperbound}}

We have 
\begin{align}\label{diffSigma}
	\hat\bSigma - \bSigma &= \sum_j v_j \big(\hat\U\hat\U^\top (\hat\bSigma_j- \sigma^2\I)\hat\U\hat\U^\top  + \hat\U\E_j\hat\U^\top\big) - \U\U^\top(\bSigma - \sigma^2\I)\U\U^\top\notag\\
	&= \sum_jv_j \hat\U\hat\U^\top(\hat\bSigma_j- \sigma^2\I)\hat\U\hat\U^\top - \U\U^\top(\bSigma - \sigma^2\I)\U\U^\top + \sum_jv_j\hat\U\E_j\hat\U^\top.
\end{align}
For the first term in \eqref{diffSigma}, we can further decompose it as 
\begin{align}\label{diffSigma:1}
	&\quad \sum_jv_j \hat\U\hat\U^\top(\hat\bSigma_j- \sigma^2\I)\hat\U\hat\U^\top - \U\U^\top(\bSigma - \sigma^2\I)\U\U^\top\notag\\
	&= (\hat\U\hat\U^\top - \U\U^\top)\sum_jv_j(\hat\bSigma_j- \sigma^2\I)\U\U^\top + \hat\U\hat\U^\top \sum_jv_j(\hat\bSigma_j- \sigma^2\I)(\hat\U\hat\U^\top - \U\U^\top)\notag\\
	&\quad+ \U\U^\top (\sum_j v_j\hat\bSigma_j - \bSigma)\U\U^\top.
\end{align}
And notice from Lemma \ref{lemma:covariance}, for each $j$, with probability exceeding $1-e^{-t_{1,j}}$, 
\begin{align*}
	\op{\hat\bSigma_j - \bSigma} \lesssim \left(\sqrt{\frac{\tilde r+t_{1,j}}{n_j}}\vee \frac{\tilde r+t_{1,j}}{n_j}\right)(\lambda + \sigma^2),
\end{align*}
where $\tilde r = \frac{r\lambda + p\sigma^2}{\lambda + \sigma^2}$. Under the given SNR, and by setting $t_{1,j} = p\wedge n_j$, we obtain $\op{\hat\bSigma_j - \bSigma}\lesssim \lambda$, and therefore with probability exceeding $1-e^{-p\wedge n_j}$,
\begin{align*}
	\op{\hat\bSigma_j - \sigma^2\I}\leq \op{\hat\bSigma_j -\bSigma} + \op{\bSigma - \sigma^2\I} \lesssim \lambda.
\end{align*}
Moreover, we have
\begin{align*}
	\U^\top (\sum_j v_j\hat\bSigma_j - \bSigma)\U = \sum_{j=1}^m\sum_{i=1}^{n_j}\frac{v_j}{n_j}\U^\top X_{i}^{(j)}X_{i}^{(j)\top}\U - (\bLa + \sigma^2\I).
\end{align*}
Now applying Lemma \ref{lemma:covariance}, we obtain with probability exceeding $1-e^{-t_2}$, 
\begin{align*}
	\op{\U^\top (\sum_j v_j\hat\bSigma_j - \bSigma)\U} \lesssim (\sigma^2+\lambda)\sqrt{\sum_j \frac{v_j^2}{n_j}}\sqrt{r+t_2}.
\end{align*}
In summary, we obtain the upper bound as follows:
\begin{align*}
	&\quad\bigg\|\sum_jv_j \hat\U\hat\U^\top(\hat\bSigma_j- \sigma^2\I)\hat\U\hat\U^\top - \U\U^\top(\bSigma - \sigma^2\I)\U\U^\top\bigg\|_{\rm F}\\
	&\lesssim \lambda\cdot\fro{\hat\U\hat\U^\top - \U\U^\top} +(\sigma^2+\lambda)\sqrt{\sum_j \frac{v_j^2}{n_j}}\sqrt{(r+t_2)r}\\
	&\lesssim \Bigg(\sqrt{\frac{pr}{\sum_{j=1}^m \Big(n_j \wedge \big( n_j^2\epsilon_j^2p^{-1}(r+\log n_j)^{-1}\log^{-1}(2.5\delta_j^{-1})\big)\Big)}}\cdot\sqrt{(\sigma^2+\lambda)\sigma^2}\Bigg) \bigwedge (\sqrt{2r}\lambda)\\
	&\quad+(\sigma^2+\lambda)\sqrt{\sum_j \frac{v_j^2}{n_j}}\sqrt{(r+t_2)r}, 
\end{align*}
where the last inequality comes from Theorem \ref{thm:highprob:upperbound}.
For the second term in \eqref{diffSigma}, $[\sum_jv_j\E_j]_{kl} = [\sum_jv_j\E_j]_{kl}\sim N(0,\sum_{j}v_j^2\beta_j^2)$, and  $[\sum_jv_j\E_j]_{kl} \sim N(0,2\sum_jv_j^2\beta_j^2)$. 
And with probability exceeding $1-e^{-t_2}$, 
\begin{align*}
	\fro{\sum_jv_j\E_j}\lesssim \sqrt{r}\cdot\sqrt{r+t_2}\sqrt{\sum_{j}v_j^2\beta_j^2}. 
\end{align*}
So we conclude with probability exceeding $1-23\sum_j e^{-(n_j\wedge p)} - e^{-t_2}$,
\begin{align*}
	\fro{\hat\bSigma - \bSigma}^2 &\lesssim \Bigg(\frac{pr}{\sum_{j=1}^m \Big(n_j \wedge \big( n_j^2\epsilon_j^2d^{-1}(r+\log n_j)^{-1}\log^{-1}(2.5\delta_j^{-1})\big)\Big)}\cdot(\sigma^2+\lambda)\sigma^2\Bigg)\bigwedge (2r\lambda^2)\\
	&\quad + (r+t_2)r\cdot\sum_jv_j^2\bigg(\frac{1}{n_j}(\lambda^2+\sigma^4) + \beta_j^2\bigg).
\end{align*}

Next we consider the expectation for $\EE\fro{\hat\bSigma - \bSigma}^2$. From \eqref{diffSigma} and \eqref{diffSigma:1}, we see 
\begin{align}\label{expectation:diffSigma}
	\EE\fro{\hat\bSigma - \bSigma}^2 &\lesssim  \EE\fro{\sum_jv_j(\hat\bSigma_j- \sigma^2\I)(\hat\U\hat\U^\top - \U\U^\top)}^2 + \EE\fro{\sum_j v_j\E_j}^2\notag\\
	&\quad + \EE\fro{\U^\top (\sum_j v_j\hat\bSigma_j - \bSigma)\U}^2. 
\end{align}
We consider the event 
\begin{align*}
	\calF_1 = \bigcap_{j=1}^m\bigg\{\op{\hat\bSigma_j-\sigma^2\I}\leq C(\lambda + \sigma^2)\bigg(\sqrt{\frac{\tilde r+ t_{1,j}}{n_j}}\vee  \frac{\tilde r+ t_{1,j}}{n_j}\bigg)+\lambda\bigg\}. 
\end{align*}
By setting $t_{1,j} = n_j\wedge p$, then under $\calF_1$, we have $\op{\hat\bSigma_j-\sigma^2\I}\leq2\lambda$ under the given SNR condition, and $\PP(\calF_1^c)\leq \sum_je^{-(n_j\wedge p)}$. 
Then 
\begin{align*}
	&\quad\EE\fro{\sum_jv_j(\hat\bSigma_j- \sigma^2\I)(\hat\U\hat\U^\top - \U\U^\top)}^2\\
	&= \EE\fro{\sum_jv_j(\hat\bSigma_j- \sigma^2\I)(\hat\U\hat\U^\top - \U\U^\top)}^2\cdot\mathds{1}(\calF_1) + \EE\fro{\sum_jv_j(\hat\bSigma_j- \sigma^2\I)(\hat\U\hat\U^\top - \U\U^\top)}^2\cdot\mathds{1}(\calF_1^c) \\
	&\leq 4\lambda^2\EE\fro{\hat\U\hat\U^\top - \U\U^\top}^2 +\bigg(\EE\fro{\sum_jv_j(\hat\bSigma_j- \sigma^2\I)(\hat\U\hat\U^\top - \U\U^\top)}^4\bigg)^{1/2} \cdot \big(\PP(\calF_1^c)\big)^{1/2}. 
\end{align*}
Notice 
\begin{align*}
	\EE\fro{\sum_jv_j(\hat\bSigma_j- \sigma^2\I)(\hat\U\hat\U^\top - \U\U^\top)}^4 &\leq 4r^2\cdot\EE\op{\sum_jv_j \hat\bSigma_j- \sigma^2\I}^4\\
	&\leq 4r^2\cdot m^3\sum_j v_j^4\EE\op{\hat\bSigma_j- \sigma^2\I}^4.
\end{align*}
And 
\begin{align*}
	\EE\op{\hat\bSigma_j- \sigma^2\I}^4 &\leq 8\EE\op{\hat\bSigma_j - \bSigma}^4 + 8\lambda^4.
\end{align*}
From Lemma \ref{lemma:covariance}, we see 
\begin{align*}
	\PP\bigg(\op{\hat\bSigma_j - \bSigma} \leq C(\lambda + \sigma^2)\big(\sqrt{\frac{\tilde r+t}{n_j}}\vee \frac{\tilde r+t}{n_j}\big)\bigg)\leq e^{-t}. 
\end{align*}
Then from Lemma \ref{lemma:mixtail:moments}, we see 
\begin{align*}
	\EE\op{\hat\bSigma_j - \bSigma}^4\leq  C(\lambda + \sigma^2)^4\frac{\tilde r^2}{n_j^2} \leq \lambda^4,
\end{align*}
where the last inequality comes from the SNR condition. 
And thus 
\begin{align*}
	\EE\fro{\sum_jv_j(\hat\bSigma_j- \sigma^2\I)(\hat\U\hat\U^\top - \U\U^\top)}^4 \leq 64r^2m^3\lambda^4
\end{align*}
Therefore
\begin{align*}
	\bigg(\EE\fro{\sum_jv_j(\hat\bSigma_j- \sigma^2\I)(\hat\U\hat\U^\top - \U\U^\top)}^4\bigg)^{1/2} \cdot \big(\PP(\calF_1^c)\big)^{1/2} \leq 8rm^{3/2}\lambda^2\cdot(\sum_j e^{-p\wedge n_j})^{1/2}.
\end{align*}
This term is dominated by the first term as long as $\lambda/\sigma^2\lesssim \frac{1}{m^{3/2}(\sum_jn_j)(\sum_j e^{-(n_j\wedge p)})}$. 
For the second term in \eqref{expectation:diffSigma}, we have 
\begin{align*}
	\EE\fro{\sum_{j}v_j\E_j}^2 = r^2\sum_jv_j^2\beta_j^2. 
\end{align*}
We now consider the last term in \eqref{expectation:diffSigma}. 
From Lemma \ref{lemma:covariance}, we have 
\begin{align*}
	\PP\bigg(\op{\U^\top (\sum_j v_j\hat\bSigma_j - \bSigma)\U}\geq C(\sigma^2+\lambda)\sqrt{\sum_j \frac{v_j^2}{n_j}}\sqrt{r+t}\bigg)\leq e^{-t}. 
\end{align*}
Which gives 
\begin{align*}
	\EE\fro{\U^\top (\sum_j v_j\hat\bSigma_j - \bSigma)\U}^2\leq r\EE\op{\U^\top (\sum_j v_j\hat\bSigma_j - \bSigma)\U}^2\leq Cr^2(\lambda+\sigma)^2\sum_j \frac{v_j^2}{n_j}. 
\end{align*}
In conclusion, we have
\begin{align*}
	\EE\fro{\hat\bSigma - \bSigma}^2&\leq \left(\frac{Cpr}{\sum_{j=1}^m \Big(n_j \wedge \big( n_j^2\epsilon_j^2p^{-1}(r+\log n_j)^{-1}\log^{-1}(2.5\delta_j^{-1})\big)\Big)}(\sigma^2\lambda+\sigma^4) \right)\bigwedge (2r \lambda^2)\\
	&\quad+ Cr^2\sum_j v_j^2\left(\frac{\lambda^2+\sigma^4}{n_j} + \frac{8}{\epsilon_j^2}\log\left(\frac{2.5}{\delta_j}\right)\frac{\lambda^2(r+\log n_j)^2 + \sigma^4p^2}{n_j^2}\right)
\end{align*}
Now by setting $v_j\propto \left(\frac{\lambda^2+\sigma^4}{n_j} + \frac{8}{\epsilon_j^2}\log\left(\frac{2.5}{\delta_j}\right)\frac{\lambda^2(r+\log n_j)^2 + \sigma^4p^2}{n_j^2}\right)^{-1}$, we obtain 
\begin{align*}
	\EE\fro{\hat\bSigma - \bSigma}^2&\leq \left(\frac{Cpr}{\sum_{j=1}^m\Big( n_j \wedge \big( n_j^2\epsilon_j^2p^{-1}(r+\log n_j)^{-1}\log^{-1}(2.5\delta_j^{-1})\big)\Big)}(\sigma^2\lambda+\sigma^4) \right)\bigwedge (2r \lambda^2)\\
	&\quad+ \frac{Cr^2}{\sum_{j=1}^m \Big(\big(n_j(\lambda^2+\sigma^4)^{-1}\big) \wedge \big(n_j^2\epsilon_j^2\log^{-1}\left(\frac{2.5}{\delta_j}\right)\big(\lambda^2(r+\log n_j)^2 + \sigma^4p^2\big)^{-1}\big)\Big)}.
\end{align*}
Finally we show 
\begin{align*}
	&\quad\frac{Cr^2}{\sum_{j=1}^m \Big(\big(n_j(\lambda^2+\sigma^4)^{-1}\big) \wedge \big(n_j^2\epsilon_j^2\log^{-1}\left(\frac{2.5}{\delta_j}\right)\big(\lambda^2(r+\log n_j)^2 + \sigma^4p^2\big)^{-1}\big)\Big)}\\
	&\leq \frac{Cpr}{\sum_{j=1}^m \Big(n_j \wedge \big( n_j^2\epsilon_j^2p^{-1}(r+\log n_j)^{-1}\log^{-1}(2.5\delta_j^{-1})\big)\Big)}(\sigma^2\lambda+\sigma^4) \\
	&\quad+ \frac{Cr^2}{\sum_{j=1}^m \Big(\big(n_j\lambda^{-2}\big) \wedge \big(n_j^2\epsilon_j^2\log^{-1}\left(\frac{2.5}{\delta_j}\right)\lambda^{-2}(r+\log n_j)^{-2}\big)\Big)}.
\end{align*}
We consider the different cases for $\lambda/\sigma^2$. When $\lambda/\sigma^2\leq 1$, the left hand side is bounded by 
\begin{align*}
	\frac{Cr^2\sigma^4}{\sum_{j=1}^m \Big(n_j\wedge \big(n_j^2\epsilon_j^2\log^{-1}\left(\frac{2.5}{\delta_j}\right)\big)\Big)},
\end{align*}
which is bounded by the first term on the right hand side. Next if $\lambda/\sigma^2\geq \frac{p}{r}$, we have the left hand side is bounded by 
\begin{align*}
	\frac{Cr^2\lambda^2}{\sum_{j=1}^m \Big(n_j \wedge \big(n_j^2\epsilon_j^2\log^{-1}\left(\frac{2.5}{\delta_j}\right)(r+\log n_j)^{-2}\big)\Big)},
\end{align*}
which is bounded by the second term on the right hand side. Finally we consider if $1\leq \lambda/\sigma^2\leq \frac{p}{r}$. Then the left hand side is bounded by 
\begin{align*}
	\frac{Cr^2}{\sum_{j=1}^m \Big(\big(n_j\lambda^{-2}\big) \wedge \big(n_j^2\epsilon_j^2\log^{-1}\left(\frac{2.5}{\delta_j}\right)\big(\lambda^2(r+\log n_j)^2 + \sigma^4p^2\big)^{-1}\big)\Big)}
\end{align*}
Notice the first term on the right hand side is lower bounded by 
\begin{align*}
	\frac{Cpr\sigma^2\lambda}{\sum_{j=1}^m \Big(n_j \wedge \big( n_j^2\epsilon_j^2p^{-1}(r+\log n_j)^{-1}\log^{-1}(2.5\delta_j^{-1})\big)\Big)}.
\end{align*}
Therefore it is equivalent to showing 
\begin{align*}
	&\quad\frac{Cr^2}{\sum_{j=1}^m \bigg(\frac{\lambda^2}{n_j} + \frac{1}{n_j^2\epsilon_j^2}\log\left(\frac{2.5}{\delta_j}\right)\big(\lambda^2(r+\log n_j)^2 + \sigma^4p^2\big)\bigg)^{-1}}\\
	&\leq \frac{Cpr\sigma^2\lambda}{\sum_{j=1}^M \bigg(\frac{1}{n_j}+  \frac{1}{n_j^2\epsilon_j^2}p(r+\log n_j)\log(2.5\delta_j^{-1})\bigg)^{-1}},
\end{align*}
which is true if and only if
\begin{align*}
	&r\sum_{j=1}^m \bigg(\frac{1}{n_j}+  \frac{1}{n_j^2\epsilon_j^2}p(r+\log n_j)\log(2.5\delta_j^{-1})\bigg)^{-1}\leq p\sigma^2\lambda\sum_{j=1}^m \bigg(\frac{\lambda^2}{n_j} + \frac{1}{n_j^2\epsilon_j^2}\log\left(\frac{2.5}{\delta_j}\right)\big(\lambda^2(r+\log n_j)^2 + \sigma^4p^2\big)\bigg)^{-1}.
\end{align*}
This can be implied by 
\begin{align*}
	r\bigg(\frac{\lambda^2}{n_j} + \frac{1}{n_j^2\epsilon_j^2}\log\left(\frac{2.5}{\delta_j}\right)\big(\lambda^2(r+\log n_j)^2 + \sigma^4p^2\big)\bigg)\leq d\sigma^2\lambda\bigg(\frac{1}{n_j}+  \frac{1}{n_j^2\epsilon_j^2}p(r+\log n_j)\log(2.5\delta_j^{-1})\bigg),
\end{align*}
which is true if $1\leq \frac{\lambda}{\sigma^2}\leq \frac{p}{r}$.

\subsection{Proof of Lemma \ref{lem:alg-dp}}

We first state the following lemma, which will be helpful. 
\begin{lemma}\label{lemma:subspace-sensitivity}
	For any $j\in[m]$, suppose $\lambda/\sigma^2\geq C_1(p/n_j + \sqrt{p/n_j})$, and $p\geq \log n_j$. Then with probability exceeding $1-n_j^{-100} - 12e^{-c_0(p\wedge n_j)}$, 
	\begin{align*}
		\max_{i\in[n_j]}\fro{\tilde U_j\tilde U_j^\top - \tilde U_j^{(i)}\tilde U_j^{(i)\top}} \leq C\frac{1}{n_j}\sqrt{\frac{\lambda+\sigma^2}{\lambda}\frac{\sigma^2}{\lambda}} \sqrt{p(r+\log n_j)}.
	\end{align*}
\end{lemma}
\begin{proof}
	Most of the proof is the same as the proof in Lemma 3 in \cite{cai2024optimalPCA}, we only aim at improving the probability. Notice we have from Lemma \ref{lemma:covariance}, $\op{\hat \Sigma_j - \Sigma}, \op{\hat \Sigma_j^{(i)} - \Sigma}\leq c_1\lambda$ with probability exceeding $1-2e^{-c_0 (n_j\wedge p)}$, where $\hat \Sigma_j^{(i)} =\frac{1}{n_j}\sum_{i'\neq i}X_{i'}^{(j)}X_{i'}^{(j)\top} + \frac{1}{n_j}\tilde X_{i}^{(j)}\tilde X_{i}^{(j)\top}$, where $\tilde X_{i}^{(j)}$ is an i.i.d. copy of $X_{i}^{(j)}$. Therefore we have
	\begin{align*}
		\tilde U_j\tilde U_j^\top - UU^\top &= \sum_{k\geq 1}\calS_{\Sigma,k}(\Xi), \\
		\tilde U_j^{(i)}\tilde U_j^{(i)\top} - UU^\top &= \sum_{k\geq 1}\calS_{\Sigma,k}(\Xi^{(i)}),
	\end{align*}
	where $\Xi = \hat\Sigma_j-\Sigma$, and $\Xi^{(i)} = \hat\Sigma_j^{(i)}-\Sigma$.
	This implies
	\begin{align*}
		\tilde U_j \tilde U_j^\top - 	\tilde U_j^{(i)}\tilde U_j^{(i)\top}  = \calS_{\Sigma,1}(\Xi)-\calS_{\Sigma,1}(\Xi^{(i)}) + \sum_{k\geq 2}\big(\calS_{\Sigma,k}(\Xi) - \calS_{\Sigma,k}(\Xi^{(i)})\big). 
	\end{align*}
	Notice 
	\begin{align*}
		\calS_{\Sigma,1}(\Xi)-\calS_{\Sigma,1}(\Xi^{(i)}) &= U\Lambda^{-1}U^\top(\Xi-\Xi^{(i)})U_{\perp}U_{\perp}^\top + U_{\perp}U_{\perp}^\top(\Xi-\Xi^{(i)})U\Lambda^{-1}U^\top\\
		&=\frac{1}{n_j}U\Lambda^{-1}U^\top( X_{i}^{(j)} X_{i}^{(j)\top}-\tilde X_{i}^{(j)}\tilde X_{i}^{(j)\top})U_{\perp}U_{\perp}^\top \\
		&\quad+ \frac{1}{n_j}U_{\perp}U_{\perp}^\top( X_{i}^{(j)} X_{i}^{(j)\top}-\tilde X_{i}^{(j)}\tilde X_{i}^{(j)\top})U\Lambda^{-1}U^\top.
	\end{align*}
	We consider the event 
	\begin{align*}
		\calE_1 = &\bigg\{\ltwo{U^\top X_{i}^{(j)}}, \ltwo{U^\top \tilde X_{i}^{(j)}}\lesssim \sqrt{\lambda+\sigma^2}\sqrt{r+\log n_j}: \forall i\in[n_j]\bigg\}\\
		&\quad\cap\bigg\{\ltwo{U_{\perp}^\top X_{i}^{(j)}}, \ltwo{U_{\perp}^\top \tilde X_{i}^{(j)}}\lesssim \sigma\sqrt{p}: \forall i\in[n_j]\bigg\}.
	\end{align*}
	Then $\PP(\calE_1)\geq 1-n_j^{-100}$. 
	Therefore under $\calE_1$, 
	\begin{align*}
		\fro{\calS_{\Sigma,1}(\Xi)-\calS_{\Sigma,1}(\Xi^{(i)}) }&\leq \frac{2}{n_j}\fro{U\Lambda^{-1}U^\top X_{i}^{(j)} X_{i}^{(j)\top}U_{\perp}U_{\perp}^\top} + \frac{2}{n_j}\fro{U\Lambda^{-1}U^\top \tilde X_{i}^{(j)}\tilde X_{i}^{(j)\top}U_{\perp}U_{\perp}^\top}\\
		&\lesssim \frac{1}{n_j}\lambda^{-1}\sqrt{\lambda+\sigma^2} \sigma\sqrt{p}\sqrt{r+\log n_j}.
	\end{align*}
	Now we consider for $k\geq 2$. We denote the index set 
	\begin{align*}
		\SS_k = \{\s=(s_1,\cdots, s_{k+1}): s_1,\cdots,s_{k+1}\geq 0, s_1+\cdots + s_{k+1}=k\},
	\end{align*}
	whose cardinality is bounded by $|\SS_k| = \binom{2k}{k}\leq 4^k$. 
	We define 
	\begin{align*}
		\calT_{\Sigma, k,\s,l} = M(s_1)\Lambda^{-s_1}\underline{M(s_1)^\top\Xi^{(i)}M(s_2)}&\cdots \underline{M(s_l)^\top(\Xi - \Xi^{(i)})M(s_{l+1})}\\
		&\cdots \underline{M(s_k)^\top\Xi M(s_{k+1})}\Lambda^{-s_{k+1}}M(s_{k+1})^\top
	\end{align*}
	for $k\geq 2$, $\s \in \SS_k$, and $l\in[k]$ and $M(0):=U_{\perp}, M(s) = U$ for $s>0$. With slight abuse of notation, $\Lambda^{-0}= I_{p-r}$.  
	Then we have
	\begin{align}\label{S2-S2}
		\sum_{k\geq 2}\big(\calS_{\Sigma,k}(\Xi) - \calS_{\Sigma,k}(\Xi^{(i)})\big) = \sum_{k\geq 2}\sum_{\s\in\SS_k}\sum_{l\in[k]}\calT_{\Sigma, k,\s,l}. 
	\end{align}
	We consider the event 
	\begin{align*}
		\calE_2 &=\bigg\{\op{U^\top\Xi U}, \op{U^\top\Xi^{(i)} U}\lesssim (\lambda+\sigma^2)\sqrt{\frac{r+\eta}{n_j}}:\forall i\in[n_j]\bigg\}\\
		&\quad\cap\bigg\{\op{U_{\perp}^\top\Xi U}, \op{U_{\perp}^\top\Xi^{(i)} U}\lesssim (\lambda^{1/2}+\sigma)\sigma\sqrt{\frac{p}{n_j}}:\forall i\in[n_j]\bigg\}\\
		&\quad\cap\bigg\{\op{U_{\perp}^\top\Xi U_{\perp}}, \op{U_{\perp}^\top\Xi^{(i)} U_{\perp}}\lesssim \sigma^2\sqrt{\frac{p}{n_j}}:\forall i\in[n_j]\bigg\}.
	\end{align*}
	Then $\PP(\calE_2)\geq 1- 4(e^{-c_0p}+e^{-c_0n_j}) - n_je^{-\eta}$ for some $\eta>0$ to be specified. 
	Then as long as $\eta+r\leq p\wedge n_j$, and under the given SNR condition, we have 
	\begin{align*}
		\lambda^{-1}\max\bigg\{(\lambda+\sigma^2)\sqrt{\frac{r+\eta}{n_j}},(\lambda^{1/2}+\sigma)\sigma\sqrt{\frac{p}{n_j}}, \sigma^2\sqrt{\frac{p}{n_j}}\bigg\}\leq \frac{1}{10}. 
	\end{align*}
	Now we bound $\fro{\calT_{\Sigma, k,\s,l} }$ under $\calE_1\cap\calE_2$. We discuss different choices of $s_l,s_{l+1}$. 
	
	\hspace{1cm}
	
	\noindent\textit{Case 1: $s_l, s_{l+1}>0$. } In this case, we have
	\begin{align*}
		M(s_l)^\top(\Xi - \Xi^{(i)})M(s_{l+1})= \frac{1}{n_j}U^\top(X_{i}^{(j)}X_{i}^{(j)\top} - \tilde X_{i}^{(j)}\tilde X_{i}^{(j)\top})U.
	\end{align*}
	Therefore under $\calE_1$, we have
	\begin{align*}
		\fro{M(s_l)^\top(\Xi - \Xi^{(i)})M(s_{l+1})} \lesssim \frac{(\lambda+\sigma^2)(r+\log n_j)}{n_j}.
	\end{align*}
	Now we consider the rest terms in $\calT_{\Sigma, k,\s,l}$. Since $s_l,s_{l+1}>0$, there exists $l'\neq l \in [k]$, $s_{l'} = 0, s_{l'+1}>0$ or $s_{l'} > 0, s_{l'+1}=0$. Therefore we have
	\begin{align*}
		\fro{\calT_{\Sigma, k,\s,l}} &\leq \frac{C_1}{10^{k-2}}\lambda^{-2}\frac{(\lambda+\sigma^2)(r+\log n_j)}{n_j}(\lambda^{1/2}+\sigma)\sigma\sqrt{\frac{p}{n_j}}\\
		&\leq \frac{1}{10^{k-1}}\frac{1}{n_j}\lambda^{-1}\sqrt{\lambda+\sigma^2} \sigma\sqrt{p}\sqrt{r+\log n_j},
	\end{align*}
	where the last line holds given the SNR condition. 
	
	\hspace{1cm}
	
	\noindent\textit{Case 2: $s_l=0, s_{l+1}>0$ or $s_{l}>0, s_{l+1}=0$. } In this case, we have
	\begin{align*}
		\fro{M(s_l)^\top(\Xi - \Xi^{(i)})M(s_{l+1})}&= \frac{1}{n_j}\fro{U_{\perp}^\top(X_{i}^{(j)}X_{i}^{(j)\top} - \tilde X_{i}^{(j)}\tilde X_{i}^{(j)\top})U}\\
		&\lesssim \frac{1}{n_j}\sqrt{\lambda+\sigma^2} \sigma\sqrt{p}\sqrt{r+\log n_j}.
	\end{align*}
	And under $\calE_2$, we have
	\begin{align*}
		\fro{\calT_{\Sigma, k,\s,l}}\leq \frac{1}{10^{k-1}}\frac{1}{n_j}\lambda^{-1}\sqrt{\lambda+\sigma^2} \sigma\sqrt{p}\sqrt{r+\log n_j}. 
	\end{align*}
	
	\hspace{1cm}
	
	\noindent\textit{Case 3: $s_l=s_{l+1}=0$. } In order to derive a tight upper bound, we need to use the leave-one-out technique. Notice
	\begin{align*}
		\calT_{\Sigma, k,\s,l} &= M(s_1)\Lambda^{-s_1}\underline{M(s_1)^\top\Xi^{(i)}M(s_2)}\cdots \underline{U_{\perp}^\top(\Xi - \Xi^{(i)})U_{\perp}}\\
		&\hspace{5cm}\cdots \underline{M(s_k)^\top\Xi M(s_{k+1})}\Lambda^{-s_{k+1}}M(s_{k+1})^\top\\
		&=\frac{1}{n_j}M(s_1)\Lambda^{-s_1}\underline{M(s_1)^\top\Xi^{(i)}M(s_2)}\cdots \underline{U_{\perp}^\top(X_{i}^{(j)}X_{i}^{(j)\top} - \tilde X_{i}^{(j)}\tilde X_{i}^{(j)\top})U_{\perp}}\\
		&\hspace{5cm}\cdots \underline{M(s_k)^\top\Xi M(s_{k+1})}\Lambda^{-s_{k+1}}M(s_{k+1})^\top.
	\end{align*}
	We only consider the bound for 
	$$\fro{M(s_1)\Lambda^{-s_1}\underline{M(s_1)^\top\Xi^{(i)}M(s_2)}\cdots \underline{U_{\perp}^\top X_{i}^{(j)}X_{i}^{(j)\top}U_{\perp}}\cdots \underline{M(s_k)^\top\Xi M(s_{k+1})}\Lambda^{-s_{k+1}}M(s_{k+1})^\top},$$
	and the other term can be bounded similarly. Since $s_l = s_{l+1} = 0$. There exists some $l_0\in[k+1]$, $s_{l_0} >0$. We assume wlog $l_0>l+1$ and that $l_0$ is the smallest integer that $s_{l_0}>0$. In fact, if $l_0<l$, then the term can be easier to bound due to the independence between $X_{i}^{(j)}$ and $\Xi^{(i)}$. 
	Now we consider the term
	\begin{align*}
		\fro{	\underbrace{\underline{X_{i}^{(j)\top}U_{\perp}}}_{1\text{st}}\cdot\underline{U_{\perp}^\top \Xi U_{\perp}}\cdots\underbrace{\underline{U_{\perp}^\top \Xi U}}_{(l_0-l)\text{-th}} }.
	\end{align*}
	We now decompose $\Xi = \Xi_1 + \Xi_2$, with $\Xi_1 = \frac{1}{n_j}(X_{i}^{(j)}X_{i}^{(j)\top}-  \Sigma)$, and $\Xi_2 = \frac{1}{n_j}(\sum_{i'\neq i}X_j^{(i')}X_j^{(i')\top}-  \Sigma)$. 
	Then
	\begin{align*}
		\underbrace{\underline{X_{i}^{(j)\top}U_{\perp}}}_{1\text{st}}\cdot\underline{U_{\perp}^\top \Xi U_{\perp}}\cdots\underbrace{\underline{U_{\perp}^\top \Xi U}}_{(l_0-l)\text{-th}} &= \underline{X_{i}^{(j)\top}U_{\perp}}\cdot\underline{U_{\perp}^\top \Xi U_{\perp}}\cdots\underline{U_{\perp}^\top \Xi_1 U}\\
		&+\underline{X_{i}^{(j)\top}U_{\perp}}\cdot\underline{U_{\perp}^\top \Xi U_{\perp}}\cdots\underline{U_{\perp}^\top \Xi_1 U_{\perp}}\cdots\underline{U_{\perp}^\top \Xi_2 U}\\
		& + \cdots\\
		& + \underline{X_{i}^{(j)\top}U_{\perp}}\cdot\underline{U_{\perp}^\top \Xi_2 U_{\perp}}\cdots\underline{U_{\perp}^\top \Xi_2 U_{\perp}}\cdots\underline{U_{\perp}^\top \Xi_2 U}\\
		&=: g_1^\top + \cdots + g_{l_0-l}^\top.
	\end{align*}
	Notice $\Xi_2$ is independent of $X_{i}^{(j)}$. Therefore condition on $\Xi_2$, 
	\begin{align*}
		g_{l_0-l}\sim N(0,\sigma^2\underline{U^\top \Xi_2 U_{\perp}}\cdot\underline{U_{\perp}^\top \Xi_2 U_{\perp}}\cdots\underline{U_{\perp}^\top \Xi_2 U_{\perp}}\cdot\underline{U_{\perp}^\top \Xi_2 U_{\perp}}\cdots\underline{U_{\perp}^\top \Xi_2 U_{\perp}}\cdots\underline{U_{\perp}^\top \Xi_2 U}).
	\end{align*}
	We define 
	\begin{align*}
		\calE_3 &=
		\bigg\{\op{U_{\perp}^\top\Xi_2 U} \lesssim (\lambda^{1/2}+\sigma)\sigma\sqrt{\frac{p}{n_j}}:\forall i\in[n_j]\bigg\}
		\cap\bigg\{\op{U_{\perp}^\top\Xi_2 U_{\perp}}\lesssim \sigma^2\sqrt{\frac{p}{n_j}}:\forall i\in[n_j]\bigg\}.
	\end{align*}
	Then $\PP(\calE_3)\geq 1- 4(e^{-c_0p}+e^{-c_0n_j})$. 
	Then under $\calE_1\cap\calE_2\cap\calE_3$, 
	\begin{align*}
		\ltwo{g_{l_0-l}} \lesssim \frac{1}{10^{l_0-l-2}}\sqrt{r}\sigma(\lambda^{1/2}+\sigma)\sigma\sqrt{\frac{p}{n_j}}.
	\end{align*}
	And similarly for all $l' = 1, \cdots, l_0-l-1$, we have 
	\begin{align*}
		\ltwo{g_{l'}} \lesssim \frac{1}{10^{l_0-l-2}}\frac{p}{n_j}\sigma^3\sqrt{r}\lesssim \frac{1}{10^{l_0-l-2}}\sqrt{r}\sigma(\lambda^{1/2}+\sigma)\sigma\sqrt{\frac{p}{n_j}}. 
	\end{align*}
	In summary, 
	\begin{align*}
		\fro{	\underbrace{\underline{X_{i}^{(j)\top}U_{\perp}}}_{1\text{st}}\cdot\underline{U_{\perp}^\top \Xi U_{\perp}}\cdots\underbrace{\underline{U_{\perp}^\top \Xi U}}_{(l_0-l)\text{-th}} } \lesssim \frac{l_0}{10^{l_0-l-2}}\sqrt{r}\sigma(\lambda^{1/2}+\sigma)\sigma\sqrt{\frac{p}{n_j}}.
	\end{align*}
	Now we use the event $\calE_1, \calE_2$ to bound the rest of the terms, which give
	\begin{align*}
		&\quad\fro{M(s_1)\Lambda^{-s_1}\underline{M(s_1)^\top\Xi^{(i)}M(s_2)}\cdots \underline{U_{\perp}^\top X_{i}^{(j)}X_{i}^{(j)\top}U_{\perp}}\cdots \underline{M(s_k)^\top\Xi M(s_{k+1})}\Lambda^{-s_{k+1}}M(s_{k+1})^\top}\\
		&\lesssim \lambda^{-2}\sqrt{p}\sigma\frac{1}{10^{k-2}}\sqrt{r}\sigma(\lambda^{1/2}+\sigma)\sigma\sqrt{\frac{p}{n_j}}.
	\end{align*}
	This implies 
	\begin{align*}
		\fro{\calT_{\Sigma, k,\s,l}} \lesssim \frac{1}{n_j}\lambda^{-2}\sqrt{p}\sigma\frac{1}{10^{k-2}}\sqrt{r}\sigma(\lambda^{1/2}+\sigma)\sigma\sqrt{\frac{p}{n_j}}.
	\end{align*}
	In other words, under the given SNR, we have
	\begin{align*}
		\fro{\calT_{\Sigma, k,\s,l}} \leq \frac{1}{10^{k-1}}\frac{1}{n_j}\lambda^{-1}\sqrt{\lambda+\sigma^2} \sigma\sqrt{p}\sqrt{r+\log n_j}. 
	\end{align*}
	Finally from \eqref{S2-S2}, we conclude 
	\begin{align*}
		\fro{\sum_{k\geq 2}\big(\calS_{\Sigma,k}(\Xi) - \calS_{\Sigma,k}(\Xi^{(i)})\big)} &\leq \sum_{k\geq 2}\sum_{\s\in\SS_k}\sum_{l\in[k]}\fro{\calT_{\Sigma, k,\s,l}}\\
		&\leq \sum_{k\geq 2}\sum_{\s\in\SS_k}\sum_{l\in[k]}\frac{1}{10^{k-1}}\frac{1}{n_j}\lambda^{-1}\sqrt{\lambda+\sigma^2} \sigma\sqrt{p}\sqrt{r+\log n_j}\\
		&\lesssim \sum_{k\geq 2}\sum_{\s\in\SS_k}\frac{1}{8^{k}}\frac{1}{n_j}\lambda^{-1}\sqrt{\lambda+\sigma^2} \sigma\sqrt{p}\sqrt{r+\log n_j}\\
		&\lesssim \sum_{k\geq 2}\frac{1}{2^{k}}\frac{1}{n_j}\lambda^{-1}\sqrt{\lambda+\sigma^2} \sigma\sqrt{p}\sqrt{r+\log n_j}\\
		&\leq\frac{1}{n_j}\lambda^{-1}\sqrt{\lambda+\sigma^2} \sigma\sqrt{p}\sqrt{r+\log n_j}.
	\end{align*}
	In summary, by setting $\eta = n_j\wedge p$ and taking union bound over all $i\in[n_j]$, we conclude with probability exceeding $1-n_j^{-100} - 12e^{-c_0(p\wedge n_j)}$, 
	\begin{align*}
		\max_{i\in[n_j]}\fro{\tilde U_j\tilde U_j^\top - 	\tilde U_j^{(i)}\tilde U_j^{(i)\top}}\leq C\frac{1}{n_j}\sqrt{\frac{\lambda+\sigma^2}{\lambda}\frac{\sigma^2}{\lambda}} \sqrt{p(r+\log n_j)}.
	\end{align*}
\end{proof}
The proof for the sensitivity of singular values is a direct result of Lemma 4 in \cite{cai2024optimalPCA}. 
The claim of Lemma \ref{lem:alg-dp} then follows the sensitivity of Gaussian mechanism (see e.g. Lemma 1 in \cite{cai2024optimalPCA}).

\subsection{Proof of Theorem~\ref{thm:lower-bound}} \label{sec:proof-lower-bound}
We first show the lower bound for subspace estimation and then the lower bound for covariance matrix estimation.

\noindent\textbf{Lower bound for subspace estimation. }
Let $\Theta$ be a random matrix of size $p\times r$ with its entries i.i.d. $N(0,1)$. The density function of $\Theta$ is $p(\Theta) = (2\pi)^{-pr/2}\cdot\exp(-\fro{\Theta}^2/2)$.  
Let $W := \Theta^\top\Theta$ has the Wishart distribution $\mcW_r(I_r, p)$. 
Define a map $\psi:\RR^{p\times r}\rightarrow \RR^{p\times p}$ as $\psi(\Theta) = \Theta(\Theta^\top\Theta)^{-1}\Theta^\top$. Denote $\psi_{k_1k_2}(\Theta) := e_{k_1}^\top\psi(\Theta) e_{k_2}$ be the $(k_1, k_2)$-th component function of $\psi(\Theta)$ for all $ k_1, k_2\in[p]$. Basically, $\psi$ maps a given $p\times r$ matrix to a $p\times p$ rank-$r$ projection matrix.  Moreover, denote $\bar\Theta:=\Theta(\Theta^{\top}\Theta)^{-1/2}\in\OO^{p\times r}$ the left singular vectors of $\Theta$. It is clear by definition that $\psi(\Theta)=\bar\Theta \bar\Theta^{\top}$. 

Suppose that $X_i^{(j)}\overset{\text{i.i.d.}}{\sim} N(0,\lambda\bar\Theta\bar\Theta^\top + \sigma^2 I_p)$ for all $ j\in[m]$ and $\forall i\in[n_j]$.  Denote $\bar\Theta_{\perp}\in\OO^{p\times (p-r)}$ such that $(\bar\Theta, \bar\Theta_{\perp})$ is a $p\times p$ orthogonal matrix. 
We denote 
\begin{align}\label{YZ}
	\mat{Y_i^{(j)}\\ Z_i^{(j)}} := \mat{\bar\Theta^\top\\ \bar\Theta_{\perp}^\top}X_i^{(j)} \stackrel{{\rm i.i.d.}}{\sim} \mat{N(0,(\lambda+\sigma^2) I_r) \\ N(0,\sigma^{2}I_{p-r})},\quad \forall j\in[m],  i\in[n_j]
\end{align}
We define the score corresponding to $X_i^{(j)}$ as 
\begin{align}\label{eq:Sji}
	S_{j,i} := \nabla\log p(X_i^{(j)};\bTheta)= \big((\lambda + \sigma^2)^{-1} - \sigma^{-2}\big)\bar\bTheta_{\perp} Z_i^{(j)}Y_i^{(j)\top} W^{-1/2}\in\RR^{p\times r}. 
\end{align}
Denote $\calD_j:=\{X^{(j)}_i: i\in[n_j]\}$ the data set stored at $j$-th local client. We define
$$
S_j :=\nabla\log p(\calD_j;\bTheta)= \sum_{i=1}^{n_j}\nabla\log p(X_i^{(j)};\bTheta) = \sum_{i=1}^{n_j} S_{j,i}.
$$ 
This induces a linear operator $\RR^{p\times r}\mapsto \RR^{p\times r}$ for all $ j\in[m], i\in[n_j]$ defined by 
\begin{align}\label{eq:gcg}
	\calC_{j,i}(V) := \EE\inp{S_{j,i}}{V} S_{j,i} =  \frac{\lambda^2}{(\lambda+\sigma^2)\sigma^2}\bar\bTheta_{\perp}\bar\bTheta_{\perp}^\top VW^{-1},
\end{align}
where the expectation is taken w.r.t. $X_i^{(j)}$. We denote the sum as 
\begin{align}\label{def:calCj}
	\calC_j(V) := \sum_{i=1}^{n_j}\calC_{j,i}(V) =  \frac{n_j\lambda^2}{(\lambda+\sigma^2)\sigma^2}\bar\bTheta_{\perp}\bar\bTheta_{\perp}^\top VW^{-1}.
\end{align}

The following lemma states a matrix version of the Van Trees' inequality. We first clarify some useful notations. In the following, we view the gradient $\nabla\psi(\Theta)$ as an operator maps from $\RR^{p\times r}$ to $\RR^{p\times p}$, i.e., $\nabla\psi(\Theta)(Y)\in\RR^{p\times p}$ for all $ Y\in\RR^{p\times r}$ as a directional derivative.  See more details in Appendix \ref{sec:linalg:gradient}. Similarly, the gradient $\nabla\log p(\calD_j;\bTheta|\hat U_j)\in\RR^{p\times r}$ can be identified as an operator maps from $\RR^{p\times r}\rightarrow\RR$ such that $\cdot\mapsto \inp{\nabla\log p(\calD_j;\bTheta|\hat U_j)}{\cdot}$. Let $\nabla \psi(\Theta)^{\ast}: \RR^{p\times p}\mapsto \RR^{p\times r}$ the adjoint operator satisfying 
$$
\big<\nabla \psi(\Theta)(Y), M\big>=\big<\nabla\psi(\Theta)^{\ast}(M), Y\big>,\quad {\rm for}\ \forall Y\in\RR^{p\times r}\quad {\rm and}\quad \forall M\in\RR^{p\times p}. 
$$
Let $\circ$ denote the composition of operators. The trace of a self-adjoint operator $\calL$ that maps from $\RR^{p\times p}$ to itself is defined by 
$$
\tr(\calL):=\sum_{i,j\in[p]} \big<\calL(e_ie_j^{\top}), e_ie_j^{\top}\big>,
$$
where $e_i$ denotes the $i$-th canonical basis vector of $\RR^p$.

\begin{lemma}\label{lem:vantree}
For any estimator $\hat U\in\OO^{p\times r}$ of $\psi(\Theta)$, its average-case error rate is lower bounded by 
	\begin{align*}
		\int\EE\fro{\hat U\hat U^\top - \psi(\bTheta)}^2\cdot p(\bTheta) d\bTheta \geq \frac{\big(\int\tr(\nabla\psi(\bTheta)\circ \nabla\psi(\bTheta)^*)\cdot p(\bTheta)d\bTheta\big)^2}{\sum_{j=1}^M \EE\int \tr\big(\nabla\psi(\bTheta)\circ \calI(\bTheta|\hat U_j) \circ  \nabla\psi(\bTheta)^*\big)\cdot  p(\bTheta)d\bTheta+\calJ(p) },
	\end{align*}
where $\hat U_j$ denotes any $(\eps_j, \delta_j)$-DP estimator based on dataset $\calD_j$ at $j$-th local client and 
	\begin{align*}
		\calJ(p) &= \sum_{k_1, k_2\in[p]} \int\Delta\psi_{k_1k_2}^2(\bTheta)p(\bTheta) d\bTheta,\\
		\Delta\psi_{k_1k_2}^2(\bTheta)&=\sum_{(i,j)\in[p]\times [r]}\frac{\partial^2(\psi_{k_1k_2}^2)}{\partial\Theta_{ij}^2}(\Theta), \quad \quad \forall k_1, k_2\in[p]\\
		\calI(\bTheta|\hat U_j) &= \EE \big[\big(\nabla\log p(\calD_j;\bTheta|\hat U_j)\big)^*\circ\nabla\log p(\calD_j;\bTheta|\hat U_j)\big].
	\end{align*}
\end{lemma}


It suffices to control the three terms involved in the right hand side of Lemma \ref{lem:vantree}. 
We will show (see Appendix \ref{sec:linalg:gradient} for more details) $\nabla\psi(\bTheta): \RR^{p\times r} \rightarrow \RR^{p\times p}$ is the following linear map:
\begin{align}\label{nablapsi}
	\nabla\psi(\bTheta)(Y)= \bar\bTheta_{\perp}\bar\bTheta_{\perp}^\top YW^{-1/2}\bar\bTheta^\top + \bar\bTheta W^{-1/2}Y^\top\bar\bTheta_{\perp}\bar\bTheta_{\perp}^\top.
\end{align}
Meanwhile (see Appendix \ref{sec:linalg:adjoint}), $\nabla\psi(\bTheta)^*:\RR^{p\times p}\rightarrow \RR^{p\times r}$ is given by
\begin{align}\label{nablapsiadj}
	\nabla\psi(\bTheta)^*(M) = \bar\bTheta_{\perp}\bar\bTheta_{\perp}^\top(M+M^\top)\bar\bTheta W^{-1/2}. 
\end{align}

\paragraph{Lower bound for $\int \tr\big(\nabla\psi(\bTheta)\circ\nabla\psi(\bTheta)^*\big) p(\bTheta) d\bTheta$ }
Based on \eqref{nablapsi} and \eqref{nablapsiadj}, we have for all $ M\in\RR^{p\times p}$ that
\begin{align*}
	&\quad\nabla\psi(\bTheta)\circ\nabla\psi(\bTheta)^*(M) \\
	&= \bar\bTheta_{\perp}\bar\bTheta_{\perp}^\top(M+M^\top)\bar\bTheta W^{-1}\bar\bTheta^\top +  \bar\bTheta W^{-1}\bar\bTheta^\top(M+M^\top) \bar\bTheta_{\perp}\bar\bTheta_{\perp}^\top.
\end{align*}
So, by definition, we get
\begin{align*}
	&\quad\tr\big(\nabla\psi(\bTheta)\circ\nabla\psi(\bTheta)^*\big) \\
	&= \sum_{i,j\in[p]} e_i^\top \bigg(\bar\bTheta_{\perp}\bar\bTheta_{\perp}^\top(e_ie_j^\top+e_je_i^\top)\bar\bTheta W^{-1}\bar\bTheta^\top +  \bar\bTheta W^{-1}\bar\bTheta^\top(e_ie_j^\top+e_je_i^\top) \bar\bTheta_{\perp}\bar\bTheta_{\perp}^\top\bigg) e_j\\
	&= 2\tr(\bar\bTheta_{\perp}\bar\bTheta_{\perp}^\top)\cdot\tr(\bar\bTheta W^{-1}\bar\bTheta^\top) + 2\inp{\bar\bTheta W^{-1}\bar\bTheta^\top}{\bar\bTheta_{\perp}\bar\bTheta_{\perp}}\\
	&= 2(p-r)\tr(W^{-1}). 
\end{align*}
Therefore, 
\begin{align*}
	\int \tr\big(\nabla\psi(\bTheta)\circ\nabla\psi(\bTheta)^*\big) p(\bTheta) d\bTheta = 2(p-r)\EE \tr(W^{-1}),
\end{align*}
where $W\sim W_r(I_r,p)$ follows the Wishart distribution. 
Following the Theorem 3.1 of \cite{von1988moments}, we have $\EE W^{-1} = (p-r-1)^{-1}I_r$ if $p-r-1\geq 1$. So we have 
\begin{align*}
	\int \tr\big(\nabla\psi(\bTheta)\circ\nabla\psi(\bTheta)^*\big) p(\bTheta) d\bTheta = \frac{2(p-r)r}{p-r-1}\geq 2r.
\end{align*}

\vspace{0.2cm}

\paragraph{Upper bound for $\calJ(p)$ }
Simple calculations show, for all $ k_1,k_2\in[p]$, that 
\begin{align*}
	\nabla\psi_{k_1k_2}(\bTheta) &= \bar\bTheta_{\perp}\bar\bTheta_{\perp}^\top(e_{k_1}e_{k_2}^\top + e_{k_2}e_{k_1}^\top)\bar\bTheta W^{-1/2},\\
	\Delta\psi_{k_1k_2} &= 2[\bar\bTheta_{\perp}\bar\bTheta_{\perp}^\top]_{k_1k_2}\cdot\tr(W^{-1}) - 2(p-r) [\bar\bTheta W^{-1}\bar\bTheta^\top]_{k_1k_2}. 
\end{align*}
Since $\nabla p(\bTheta) = (2\pi)^{-pr/2}\exp(-\fro{\bTheta}^2/2)(-\bTheta)$, we have $\inp{\nabla p}{\nabla \psi_{k_1k_2}} = 0$ for all $ k_1,k_2\in[p]$. As a result, 
\begin{align*}
	\int\sum_{k_1,k_2\in[p]}(\Delta\psi_{k_1k_2})^2 p(\bTheta)d\bTheta &= \int \bigg(4(p-r)^2 \fro{W^{-1}}^2 + 4(p-r) \big(\tr(W^{-1})\big)^2\bigg) p(\bTheta)d\bTheta\\
	&= 4(p-r)^2\EE\fro{W^{-1}}^2 + 4(p-r)\EE\big(\tr(W^{-1})\big)^2\\
	&\leq 4(p-r)^2\EE\fro{W^{-1}}^2 + 4(p-r)r\EE\fro{W^{-1}}^2 \\
	&\leq 8(p-r)^2\EE\fro{W^{-1}}^2
\end{align*}
as long as $p\geq 2r$. 

Following the Corollary 3.1 of \cite{von1988moments}, we have \footnote{There appears to be a typo in Corollary 3.1 (i), where the coefficient of the second term on the right hand side should be $c_2$ instead of $c_1$. }
\begin{align*}
	\EE W^{-2} &= (c_1 +c_2+ c_2 r ) I_r,
\end{align*}
where $c_1 = (p-r-2)c_2$ and $c_2 = [(p-r)(p-r-1)(p-r-3)]^{-1}$. As a result, 
\begin{align}
	\int\sum_{k_1,k_2\in[p]}(\Delta\psi_{k_1k_2})^2 p(\bTheta)d\bTheta \lesssim r. 
\end{align}

\paragraph{Upper bound for $\EE\int \tr\big(\nabla\psi(\bTheta)^*\circ \calI(\bTheta|\hat U_j) \circ  \nabla\psi(\bTheta)\big)\cdot  p(\bTheta)d\bTheta$ }
In fact, for all $ j\in[m]$, we have 
\begin{align}\label{eq:trpsispsi}
	&\quad\EE\int \tr\big(\nabla\psi(\bTheta)^*\circ \calI(\bTheta|\hat U_j) \circ  \nabla\psi(\bTheta)\big)\notag\\
	&= \EE\big[ \EE \fro{\nabla\psi(\bTheta)(\nabla\log p(\calD_j;\bTheta|\hat U_j))}^2\big]\notag\\
	&= \EE\big[ \sum_{i=1}^{n_j}\underbrace{\EE \inp{\nabla\psi(\bTheta)(\nabla\log p(\calD_j;\bTheta|\hat U_j))}{\nabla\psi(\bTheta)(\nabla\log p(X^{(j)}_i;\bTheta|\hat U_j))}}_{=: G_i^{(j)}}\big].
\end{align}
Meanwhile, for all $ j\in[m], i\in[n_j]$, we also define
\begin{align*}
	\tilde G_{i}^{(j)} := \EE \inp{\nabla\psi(\bTheta)(\nabla\log p(\calD_j;\bTheta|\hat U_j))}{\nabla\psi(\bTheta)(\nabla\log p(\tilde X^{(j)}_i;\bTheta|\hat U_j))},
\end{align*}
where $\tilde X_i^{(j)}$ is an i.i.d. copy of $X_i^{(j)}$. Note that the expectation is taken conditional on $\hat U_j$,  implying that $\EE \tilde G_i^{(j)} = 0$. 

Denote $(G_i^{(j)})^+: = 0\vee G_i^{(j)}$ and $(G_i^{(j)})^- := - 0 \wedge G_i^{(j)}$. By slightly abuse of notation, we denote $X^{(j)}:=\big[X^{(j)}_1,\cdots, X^{(j)}_{n_j}\big]$ and $p_{X^{(j)}}$ the corresponding density function. 
Since $\hat U_j$ is $(\epsilon_j,\delta_j)$-DP, by definition, we have 
\begin{align*}
	&\quad\PP\bigg((G_i^{(j)})^+\geq t\bigg) =\int \int \PP\bigg((G_i^{(j)})^+\geq t \bigg| X^{(j)}=x^{(j)} \bigg)p_{X^{(j)}}(x^{(j)})p_{\tilde X_i^{(j)}}(\tilde x_i^{(j)}) dx^{(j)} d\tilde x_i^{(j)}\\
	&\leq \int \int \bigg(e^{\epsilon_j}\PP\big((\tilde G_i^{(j)})^+\geq t \bigg| X^{(j)}=x^{(j)},\tilde X_i^{(j)}=\tilde x_i^{(j)} \big)+ \delta_j\bigg)p_{X^{(j)}}(x^{(j)})p_{\tilde X_i^{(j)}}(\tilde x_i^{(j)}) dx^{(j)} d\tilde x_i^{(j)}\\
	&=e^{\epsilon_j}\PP\bigg((\tilde G_i^{(j)})^+\geq t\bigg) + \delta_j. 
\end{align*}
Therefore, for an arbitrary $\tau>0$ to be determined later, we have  
\begin{align*}
	\int_0^{+\infty} \PP\bigg((G_i^{(j)})^+\geq t\bigg)dt&= \int_0^{\tau} \PP\bigg((G_i^{(j)})^+\geq t\bigg)dt + \int_{\tau}^{+\infty}\PP\bigg((G_i^{(j)})^+\geq t\bigg)dt \\
	&\leq e^{\epsilon_j}\int_0^{\tau} \PP\bigg((\tilde G_i^{(j)})^+\geq t\bigg)dt + \tau\delta_j +  \int_{\tau}^{+\infty}\PP\bigg((G_i^{(j)})^+\geq t\bigg)dt\\
	&\leq (1+C_1\eps_j)\int_0^{\tau} \PP\bigg((\tilde G_i^{(j)})^+\geq t\bigg)dt + \tau\delta_j +  \int_{\tau}^{+\infty}\PP\bigg((G_i^{(j)})^+\geq t\bigg)dt,
\end{align*}
where in the last inequality we used the fact that $\max_{j\in[m]}\eps_j=O(1)$.  And similarly we can show 
\begin{align*}
	\int_0^{+\infty} \PP\bigg((G_i^{(j)})^+\geq t\bigg)dt &\geq \int_0^{+\infty} \PP\bigg((\tilde G_i^{(j)})^-\geq t\bigg)dt - C_1\epsilon_j\int_0^{+\infty} \PP\bigg((\tilde G_i^{(j)})^-\geq t\bigg)dt\\
	&\quad- \tau\delta_j -\int_{\tau}^{+\infty} \PP\bigg((\tilde G_i^{(j)})^+\geq t\bigg)dt.
\end{align*}
Combine these two inequalities and we get 
\begin{align}\label{EEGij}
	\EE G_i^{(j)} \leq \EE\tilde G_i^{(j)} + 2C_1\epsilon_j \EE|\tilde G_i^{(j)}| + 2\tau \delta_j + \int_{\tau}^{+\infty} \PP\bigg((G_i^{(j)})^+\geq t\bigg)dt  + \int_{\tau}^{+\infty} \PP\bigg(( \tilde G_i^{(j)})^-\geq t\bigg)dt.
\end{align}
The first term in above right hand side vanishes. We now bound $\EE|\tilde G_i^{(j)}|$. By Cauchy-Schwarz inequality,  we get 
\begin{align*}
	&\quad\EE|\tilde G_i^{(j)}| \leq \sqrt{\EE|\tilde G_i^{(j)}|^2} \\
	& = \sqrt{\EE\big( \EE \inp{\nabla\psi(\bTheta)(\nabla\log p(\calD_j;\bTheta|\hat U_j))}{\nabla\psi(\bTheta)(\nabla\log p(\tilde X^{(j)}_i;\bTheta|\hat U_j))}\big)^2}\\
	&\leq \sqrt{\EE \big[\EE \fro{\nabla\psi(\bTheta)(\nabla\log p(\calD_j;\bTheta|\hat U_j))}^2\big]}\cdot \sqrt{\EE \big[\EE\fro{\nabla\psi(\bTheta)\big(\nabla\log p(\tilde X^{(j)}_i;\bTheta|\hat U_j)\big)}^2\big]}.
\end{align*}
Using the data processing inequality, we have 
\begin{align*}
	\EE \big[\EE\fro{\nabla\psi(\bTheta)\big(\nabla\log p(\tilde X^{(j)}_i;\bTheta|\hat U_j)\big)}^2\big] &\leq  \EE\fro{\nabla\psi(\bTheta)\big(\nabla\log p(\tilde X^{(j)}_i;\bTheta)\big)}^2.
\end{align*}
From \eqref{eq:Sji} and \eqref{nablapsi}, we obtain 
\begin{align*}
	\EE \big[\EE\fro{\nabla\psi(\bTheta)\big(\nabla\log p(\tilde X^{(j)}_i;\bTheta|\hat U_j)\big)}^2\big] \leq \frac{2\lambda^2}{(\lambda+\sigma^2)\sigma^2}\op{ W^{-2}}. 
\end{align*}
In summary, we have
\begin{align}\label{EEabsGij}
	\EE|\tilde G_i^{(j)}| \leq \sqrt{\frac{2\lambda^2}{(\lambda+\sigma^2)\sigma^2}\op{W^{-2}}} \cdot \sqrt{\EE \big[\EE \fro{\nabla\psi(\bTheta)(\nabla\log p(\calD_j;\bTheta|\hat U_j))}^2\big]}. 
\end{align}

It remains to bound the tail probabilities $ \PP\bigg((G_i^{(j)})^+\geq t\bigg)$ and $\PP\bigg(( \tilde G_i^{(j)})^-\geq t\bigg)$. Without loss of generality, we take $i = 1$. 
We shall first consider $\EE|G_1^{(j)}|^k$ for some large and absolute integer $k>0$.
Recall $G_1^{(j)} = \EE \inp{\nabla\psi(\bTheta)(\nabla\log p(\calD_j;\bTheta|\hat U_j))}{\nabla\psi(\bTheta)(\nabla\log p(X^{(j)}_1;\bTheta|\hat U_j))}$. By definition, we get 
\begin{align*}
	\EE|G_1^{(j)}|^k &= \EE\big|\EE \inp{\nabla\psi(\bTheta)(\nabla\log p(\calD_j;\bTheta|\hat U_j))}{\nabla\psi(\bTheta)(\nabla\log p(X^{(j)}_1;\bTheta|\hat U_j))}\big|^k\\
	&\leq \EE\big|\inp{\nabla\psi(\bTheta)(S_j)}{\nabla\psi(\bTheta)(S_{j,1})}\big|^k.
\end{align*}
where the inequality is due to Jenson's inequality and recall 
\begin{align*}
	S_j := \nabla\psi(\bTheta)(\nabla\log p(\calD_j;\bTheta|\hat U_j)),\\
	S_{j,1} := \nabla\psi(\bTheta)(\nabla\log p(X_1^{(j)};\bTheta|\hat U_j))
\end{align*}
Observe that 
$
\inp{\nabla\psi(\bTheta)\big(S_j\big)}{\nabla\psi(\bTheta)(S_{j,1})} = \sum_{i=1}^{n_j}\inp{\nabla\psi(\bTheta)\big(S_{j,i}\big)}{\nabla\psi(\bTheta)(S_{j,1})}. 
$
Therefore
\begin{align*}
	&\quad\EE|\inp{\nabla\psi(\bTheta)(S_j)}{\nabla\psi(\bTheta)(S_{j,1})}|^k \\
	&= \EE\bigg| \fro{\nabla\psi(\bTheta)(S_{j,1})}^2+ \Big<\sum_{i=2}^{n_j}\nabla\psi(\bTheta)\big(S_{j,i}\big), \nabla\psi(\bTheta)(S_{j,1})\Big>\bigg|^k\\
	&\leq 2^k\EE \fro{\nabla\psi(\bTheta)(S_{j,1})}^{2k} + 2^k\EE \bigg|\Big<\sum_{i=2}^{n_j}\nabla\psi(\bTheta)\big(S_{j,i}\big), \nabla\psi(\bTheta)(S_{j,1})\Big>\bigg|^k.
\end{align*}
Denote 
$$
Y^{(j)}_{2:n_j}:=\mat{Y_2^{(j)}, \cdots, Y_{n_j}^{(j)}}\quad {\rm and}\quad Z^{(j)}_{2:n_j}:=\mat{Z_2^{(j)}, \cdots, Z_{n_j}^{(j)}}.
$$ 
Then we can write
\begin{align*}
	\sum_{i=2}^{n_j}\nabla\psi(\bTheta)\big(S_{j,i}\big) &= \big((\lambda + \sigma^2)^{-1}- \sigma^{-2}\big)\bigg(\bar\bTheta_{\perp}\sum_{i=2}^{n_j}Z_i^{(j)}Y_i^{(j)\top}W^{-1}\bar\bTheta^\top + \bar\bTheta W^{-1}\sum_{i=2}^{n_j}Y_i^{(j)}Z_i^{(j)\top}\bar\bTheta_{\perp}^\top\bigg)\\
	&=\big((\lambda + \sigma^2)^{-1}- \sigma^{-2}\big)\Big(\bar\bTheta_{\perp}Z^{(j)}_{2:n_j}Y^{(j)\top}_{2:n_j}W^{-1}\bar\bTheta^\top + \bar\bTheta W^{-1}Y_{2:n_j}^{(j)}Z_{2:n_j}^{(j)\top}\bar\bTheta_{\perp}^\top\Big). 
\end{align*}
By the definitions in eq. \eqref{YZ}, we know that all entries of  $Y^{(j)}_{2:n_j}$ are i.i.d. obeying distribution $N(0,\lambda + \sigma^2)$. Similarly,  all entries of $Z^{(j)}_{2:n_j}$ are i.i.d. obeying $N(0,\sigma^2)$. Based on these facts, we get 
\begin{align*}
	\sum_{i=2}^{n_j}\inp{\nabla\psi(\bTheta)\big(S_{j,i}\big)}{\nabla\psi(\bTheta)(S_{j,1})} = 2\big((\lambda + \sigma^2)^{-1}- \sigma^{-2}\big)^2\Big<Z^{(j)}_{2:n_j}Y^{(j)\top}_{2:n_j}W^{-1}, Z_1^{(j)}Y_1^{(j)\top}W^{-1}\big>.
\end{align*}
By denoting $\mu := \big((\lambda + \sigma^2)^{-1}- \sigma^{-2}\big)^2(\lambda+\sigma^2)\sigma^2$, we can write
\begin{align*}
	\sum_{i=2}^{n_j}\inp{\nabla\psi(\bTheta)\big(S_{j,i}\big)}{\nabla\psi(\bTheta)(S_{j,1})} = 2\mu\inp{\bar Z^{(j)}_{2:n_j}\bar Y^{(j)\top}_{2:n_j} W^{-1}}{\bar Z_1^{(j)}\bar Y_1^{(j)\top} W^{-1}},
\end{align*}
where $\bar \cdot$ are the normalized version, i.e., the entries of $\bar Z^{(j)}$ and $\bar Y^{(j)}$ are i.i.d. standard normal random variables.

Using the tower rule, we get
\begin{align*}
	&\quad\EE|\inp{\bar Z^{(j)}_{2:n_j}\bar Y^{(j)\top}_{2:n_j} W^{-1}}{\bar Z_1^{(j)}\bar Y_1^{(j)\top} W^{-1}}|^k = \EE|\inp{\bar Z^{(j)}_{2:n_j}}{\bar Z_1^{(j)}\bar Y_1^{(j)\top}W^{-2}\bar Y^{(j)}_{2:n_j}}|^k\\
	&\leq k^{k/2}\cdot\EE\ltwo{\bar Z_1^{(j)}}^k\cdot\EE\ltwo{\bar Y^{(j)\top}_{2:n_j} W^{-2}\bar Y_1^{(j)}}^k\\
	&= k^{k/2}\cdot \prod_{i=0}^{(k/2)-1}(p-r+2i)\cdot\prod_{l=0}^{(k/2)-1}(n_j-1+2l)\cdot \EE \ltwo{W^{-2}\bar Y_1^{(j)}}^k  \\
	&\leq  C^k k^{k}\cdot \prod_{i=0}^{(k/2)-1}(p-r+2i)\cdot\prod_{l=0}^{(k/2)-1}(n_j-1+2l)\cdot  \fro{W^{-2}}^k,
\end{align*}
where, in the first and last inequalities, we used Lemma \ref{lemma:weighted-chi-square} to show that $\EE \ltwo{W^{-2}\bar Y_1^{(j)}}^k \leq (Ck^{1/2}\fro{W^{-2}})^k$. Here $C>0$ is an absolute constant. 

Similarly, we get
\begin{align*}
	\EE \fro{\nabla\psi(\bTheta)(S_{j,1})}^{2k} &=  (2\mu)^k\cdot \EE\fro{\bar Z_1^{(j)}\bar Y_1^{(j)\top}W^{-1}}^{2k}\\
	&= (2\mu)^k\cdot \EE\ltwo{\bar Z_1^{(j)}}^{2k}\cdot \EE\ltwo{W^{-1}\bar Y_1^{(j)}}^{2k}\\
	&\leq C^k\cdot \prod_{i=0}^{k-1}(p-r+2i)\cdot k^k\fro{W^{-1}}^{2k}. 
\end{align*}
In summary, we have
\begin{align*}
	\EE|G_1^{(j)}|^k &\leq \EE|\inp{\nabla\psi(\bTheta)(S_j)}{\nabla\psi(\bTheta)(S_{j,1})}|^k  \\
	&\leq C^kk^k\mu^k\bigg((p-r)^k\fro{W^{-1}}^{2k} + (p-r)^{k/2}n_j^{k/2}\fro{W^{-2}}^k\bigg).
\end{align*}
We can similarly show the upper bound for $\EE|\tilde G_i^{(j)}|^k$ as 
\begin{align*}
	\EE|\tilde G_i^{(j)}|^k\leq C^kk^k\mu^k\bigg((d-r)^k\fro{W^{-1}}^{2k} + (d-r)^{k/2}n_j^{k/2}\fro{W^{-2}}^k\bigg).
\end{align*}

By Markov's inequality, we get
\begin{align*}
	\PP\bigg((G_i^{(j)})^+\geq t\bigg) \leq \PP\bigg(|G_i^{(j)}| \geq t\bigg) = \PP\bigg(|G_i^{(j)}|^k \geq t^k\bigg) \leq \frac{\EE |G_i^{(j)}|^k}{t^k}. 
\end{align*}
Therefore
\begin{align*}
	\int_{\tau}^{+\infty} \PP\bigg((G_i^{(j)})^+\geq t\bigg)dt \leq \int_{\tau}^{+\infty} \frac{\EE |G_i^{(j)}|^k}{t^k}dt = \frac{1}{k-1}\tau^{-k+1}\EE |G_i^{(j)}|^k. 
\end{align*}

Observe that, by setting $\tau = (\delta_j^{-1}\EE|G_i^{(j)}|^k)^{1/k}$, we get 
\begin{align*}
	&\quad \tau\delta_j + \int_{\tau}^{+\infty} \PP\bigg((G_i^{(j)})^+\geq t\bigg)dt  + \int_{\tau}^{+\infty} \PP\bigg((\tilde G_i^{(j)})^-\geq t\bigg)dt \\
	&\leq \underbrace{\frac{1}{k-1}\delta_j\tau + \cdots + \frac{1}{k-1}\delta_j\tau}_{k-1 \text{terms}}+ \frac{2}{k-1}\tau^{-k+1}\EE |G_i^{(j)}|^k\\
	&\leq 2(\EE|G_i^{(j)}|^k)^{1/k}\delta_j^{\frac{k-1}{k}}\\
	&\leq Ck\frac{\lambda^2}{(\lambda+\sigma^2)\sigma^2}\big((d-r)\fro{W^{-1}}^2 + (d-r)^{1/2}n_j^{1/2}\fro{W^{-2}}\big)\delta_j^{\frac{k-1}{k}}.
\end{align*}
By plugging the above bound into \eqref{EEGij}, we get 
\begin{align*}
	\EE G_i^{(j)} &\leq 2C_1\epsilon_j \sqrt{\frac{2\lambda^2}{(\lambda+\sigma^2)\sigma^2}\op{W^{-2}} }\sqrt{\EE \big[\EE \fro{\nabla\psi(\bTheta)(\nabla\log p(\calD_j;\bTheta|\hat U_j))}^2\big]} \\
	&\quad+ Ck\frac{\lambda^2}{(\lambda+\sigma^2)\sigma^2}\big((d-r)\fro{W^{-1}}^2 + (d-r)^{1/2}n_j^{1/2}\fro{W^{-2}}\big)\delta_j^{\frac{k-1}{k}}.
\end{align*}
Together with \eqref{eq:trpsispsi}, we get 
\begin{align*}
	&\quad\EE\big[ \EE \fro{\nabla\psi(\bTheta)(\nabla\log p(\calD_j;\bTheta|\hat U_j))}^2\big]\\
	&\leq 2C_1n_j\epsilon_j \sqrt{\frac{2\lambda^2}{(\lambda+\sigma^2)\sigma^2}\op{W^{-2}} }\sqrt{\EE \big[\EE \fro{\nabla\psi(\bTheta)(\nabla\log p(\calD_j;\bTheta|\hat U_j))}^2\big]} \\
	&\quad+ Ckn_j\frac{\lambda^2}{(\lambda+\sigma^2)\sigma^2}\big((p-r)\fro{W^{-1}}^2 + (p-r)^{1/2}n_j^{1/2}\fro{W^{-2}}\big)\delta_j^{\frac{k-1}{k}}.
\end{align*}
Therefore, as long as 
\begin{align*}
	Ck\cdot n_j\frac{\lambda^2}{(\lambda+\sigma^2)\sigma^2}\big((p-r)r + (p-r)^{1/2}r^{1/2}n_j^{1/2}\big)\delta_j^{\frac{k-1}{k}} \leq C_1^2n_j^2\epsilon_j^2\frac{\lambda^2}{(\lambda+\sigma^2)\sigma^2},
\end{align*}
we have
\begin{align*}
	\EE\big[ \EE \fro{\nabla\psi(\bTheta)(\nabla\log p(\calD_j;\bTheta|\hat U_j))}^2\big] \leq C_1^2n_j^2\epsilon_j^2\frac{\lambda^2}{(\lambda+\sigma^2)\sigma^2} \op{W^{-2}}. 
\end{align*}
As a result, we get
\begin{align}\label{upperbound:trhatC}
	\EE\int \tr\big(\nabla\psi(\bTheta)^*\circ \calI(\bTheta|\hat U_j) \circ  \nabla\psi(\bTheta)\big) \cdot p(\bTheta) d\bTheta\leq  C_1^2n_j^2\epsilon_j^2\frac{\lambda^2}{(\lambda+\sigma^2)\sigma^2}\cdot\EE\op{W^{-2}},
\end{align}
where recall that $W\sim W_r(I_r,p)$ follows the Wishart distribution. 

Using the data processing inequality,  we have another upper bound for $\EE \tr\big(\nabla\psi(\bTheta)^*\circ \calI(\bTheta|\hat\U_j) \circ  \nabla\psi(\bTheta)\big)$ as 
\begin{align*}
	\EE \tr\big(\nabla\psi(\bTheta)^*\circ \calI(\bTheta|\hat U_j) \circ  \nabla\psi(\bTheta)\big) \leq \tr\big(\nabla\psi(\bTheta)\circ\calC_j\circ\nabla\psi(\bTheta)^*\big),
\end{align*}
where $\calC_j$ is defined in \eqref{def:calCj}. 
From \eqref{eq:gcg} and $\calC_j = \sum_{i=1}^{n_j}\calC_{j,i}$, we see 
\begin{align*}
	\tr\big(\nabla\psi(\bTheta)\circ\calC_j\circ\nabla\psi(\bTheta)^*\big) = \frac{2n_j\lambda^2}{(\lambda+\sigma^2)\sigma^2}(p-r)\tr(W^{-2}).
\end{align*}
Therefore, 
\begin{align*}
	\int \tr \big(\nabla\psi(\bTheta)^*\circ \calI(\bTheta|\hat U_j)\circ  \nabla\psi(\bTheta)\big) \cdot p(\bTheta) d\bTheta\leq  \frac{4n_j\lambda^2}{(\lambda + \sigma^2)\sigma^2}\frac{r}{p-r}. 
\end{align*}
In summary, we have
\begin{align*}
	\int \tr\big(\nabla\psi(\bTheta)^*\circ \calI(\bTheta|\hat U_j)\circ  \nabla\psi(\bTheta)\big) \cdot p(\bTheta) d\bTheta\leq
	\min\bigg\{C_1^2n_j^2\epsilon_j^2\EE \op{W^{-2}}, \frac{4n_jr}{p-r} \bigg\}\cdot \frac{\lambda^2}{(\lambda+\sigma^2)\sigma^2}.
\end{align*}
Finally,  we plug these bounds into the right hand side of the inequality in Lemma \ref{lem:vantree},   we obtain (recall that we focus on the regime $\max_{j\in[m]}\epsilon_j =O(1)$)
\begin{align*}
	\int\fro{\hat U\hat U^\top - \psi(\bTheta)}^2\cdot p(\bTheta) d\bTheta \gtrsim \frac{r^2}{\sum_{i=1}^m\min\big\{n_j^2\epsilon_j^2\cdot \EE \op{W^{-2}}, \frac{n_jr}{p-r} \big\}\cdot \frac{\lambda^2}{(\lambda+\sigma^2)\sigma^2}+r}
\end{align*}
Finally we bound $\EE\op{W^{-2}}$. 
Denote the event $\calF_0 := \{\op{W - pI_r} \leq p/2\}$. 
From the basic concentration inequality of sample covariance matrix \citep{koltchinskii2017concentration},  we have $\PP(\calF_0)\geq 1-e^{-c_1p}$. 
Under $\calF_0$, we have $\lambda_{\min}(W)\geq p/2$. 
So we have 
\begin{align*}
	\EE\op{W^{-2}} &= \EE\op{W^{-2}}\cdot\mathds{1}(\calF_0) + \EE\op{W^{-2}}\cdot\mathds{1}(\calF_0^c)\\ 
	&\leq 4p^{-2}  + (\EE\op{W^{-2}}^2)^{1/2}\cdot e^{-c_1p/2}\\
	&\leq 4p^{-2}  + (\EE\fro{W^{-2}}^2)^{1/2}\cdot e^{-c_1p/2}\\
	&= 4p^{-2}  + (\EE\tr(W^{-4}))^{1/2}\cdot e^{-c_1p/2}. 
\end{align*}
The term $\EE\tr(W^{-4})$ can be computed using the Theorem 4.1 of \cite{von1988moments}, which implies $\EE\tr(W^{-4})\cdot e^{-c_1p/2}\leq p^{-2}$.

\hspace{1cm}

\noindent\textbf{Lower bound for covariance matrix estimation. }
We consider a subset $\Theta_1$ of $\Theta(\lambda,\sigma^2)$:
\begin{align*}
	\Theta_1 = \bigg\{\bSigma = \lambda\U\U^\top + \sigma^2\I: \U\in\OO_{p,r}\bigg\}
\end{align*}
In this set, both $\lambda$ and $\sigma^2$ are known to us, and it boils down to estimating $\U$. 
Therefore
\begin{align}\label{lowerbound1}
	\inf_{\hat\bSigma}\sup_{\bSigma\in\Theta_1}\EE\fro{\hat\bSigma - \bSigma}^2 &=  \inf_{\hat\U \in \calM(\boldsymbol{\epsilon}, \boldsymbol{\delta})}\sup_{\bSigma\in\Theta(\lambda,\sigma^2)}\lambda^2\cdot\EE\fro{\hat\U\hat\U^\top - \U\U^{\top}}^2\notag\\
	&\geq \frac{c_0pr}{\sum_{i=1}^m \Big(n_j\wedge(n_j^2\epsilon_j^2\cdot p^{-1}r^{-1})\Big)} (\lambda\sigma^2 + \sigma^4)\bigwedge (r\lambda^2). 
\end{align}

Now if $\lambda/\sigma^2\geq 1$, then in addition to \eqref{lowerbound1}, we consider another set
\begin{align*}
	\Theta_2 = \bigg\{\bSigma = \mat{(\lambda+\sigma^2)\V\V^\top + (\lambda + \sigma^2)\I_r & 0\\ 0&\sigma^2\I_{p-r}}: \V\in\OO_{r,\frac{r}{2}}\bigg\}.
\end{align*}
For any $\bSigma\in\Theta_2$, it admits the following decomposition:
\begin{align*}
	\bSigma = \mat{\mat{\V&\V_{\perp}}\\ 0}\diag(\underbrace{2\lambda+\sigma^2, \cdots, 2\lambda+\sigma^2}_{\frac{r}{2}\text{~times}},\lambda, \cdots, \lambda)\mat{\mat{\V&\V_{\perp}}\\ 0}^\top + \sigma^2\I,
\end{align*}
where $\V_{\perp}\in\OO_{r,\frac{r}{2}}$ is the orthogonal complement of $\V$. Since $\lambda/\sigma^2\geq 1$, we can conclude $\Theta_2\subset \Theta(\lambda,\sigma^2)$. 
Now the original problem reduces to a smaller one. Define 
\begin{align*}
	\tilde\Theta(\lambda,\sigma^2) =  \bigg\{\bSigma = &\V\bLa\V^\top + \sigma^2\I: \\
	&\V\in\OO_{r,\frac{r}{2}}, \bLa = \diag(\lambda_1,\cdots, \lambda_r), c_0\lambda\leq \lambda_r\leq\cdots\leq C_0\lambda\bigg\}.
\end{align*}
Then from \eqref{lowerbound1}, we have
\begin{align*}
	\inf_{\tilde\bSigma}\sup_{\bSigma\in\tilde\Theta(\lambda+\sigma^2,\lambda+\sigma^2)}\EE\fro{\tilde\bSigma - \bSigma}^2 &\geq \left(\frac{c_0r^2}{\sum_{i=1}^m \Big(n_j\wedge(n_j^2\epsilon_j^2\cdot r^{-2})\Big)} (\lambda + \sigma^2)^2\right)\bigwedge \big(r(\lambda+\sigma)^2\big)\\
	&\geq \left(\frac{c_0r^2\lambda^2}{\sum_{i=1}^m \Big(n_j\wedge(n_j^2\epsilon_j^2\cdot r^{-2})\Big)} \right)\bigwedge (r\lambda^2)
\end{align*}
Note that the estimation of $\V$ in $\tilde\Theta_2$ is a sub-problem of estimating $ \mat{\mat{\V&\V_{\perp}}\\ 0}$ in $\Theta_2$, we have
\begin{align*}
	\inf_{\hat\bSigma}\sup_{\bSigma\in\Theta_2}\EE\fro{\hat\bSigma - \bSigma}^2  &\geq \inf_{\tilde\bSigma}\sup_{\bSigma\in\tilde\Theta(\lambda+\sigma^2,\lambda+\sigma^2)}\EE\fro{\tilde\bSigma - \bSigma}^2\\
	&\geq \left(\frac{c_0r^2\lambda^2}{\sum_{i=1}^m \Big(n_j\wedge(n_j^2\epsilon_j^2\cdot r^{-2})\Big)} \right)\bigwedge (r\lambda^2).
\end{align*}
Together with the bound in \eqref{lowerbound1}, we conclude when $\lambda/\sigma^2\geq 1$, 
\begin{align*}
	&\quad\inf_{\hat\bSigma}\sup_{\bSigma\in\Theta(\lambda,\sigma^2)}\EE\fro{\hat\bSigma - \bSigma}^2 \geq 	\inf_{\hat\bSigma}\sup_{\bSigma\in\Theta_1\cup\Theta_2}\EE\fro{\hat\bSigma - \bSigma}^2\\
	&\geq \bigg(\bigg(\frac{c_0pr}{\sum_{i=1}^m \Big(n_j\wedge(n_j^2\epsilon_j^2\cdot d^{-1}r^{-1})\Big)} (\lambda\sigma^2 + \sigma^4)\bigg)\bigwedge (r\lambda^2)\bigg) \bigvee \bigg(\bigg(\frac{c_0r^2\lambda^2}{\sum_{i=1}^m \Big(n_j\wedge(n_j^2\epsilon_j^2\cdot r^{-2})\Big)}\bigg) \bigwedge (r\lambda^2)\bigg).
\end{align*}
Notice when $\lambda/\sigma^2\leq 1$, 
\begin{align*}
	\left(\frac{c_0dr}{\sum_{i=1}^m \Big(n_j\wedge(n_j^2\epsilon_j^2\cdot p^{-1}r^{-1})\Big)} (\lambda\sigma^2 + \sigma^4) \right)\bigwedge (r\lambda^2) \geq  \left(\frac{c_0r^2\lambda^2}{\sum_{i=1}^m \Big(n_j\wedge(n_j^2\epsilon_j^2\cdot r^{-2})\Big)}\right) \bigwedge (r\lambda^2).
\end{align*}
Therefore we conclude for any $\lambda, \sigma^2$ satisfy the condition in the theorem, 
\begin{align*}
	&\quad\inf_{\hat\bSigma}\sup_{\bSigma\in\Theta(\lambda,\sigma^2)}\EE\fro{\hat\bSigma - \bSigma}^2\\
	&\geq \bigg(\bigg(\frac{c_0pr}{\sum_{i=1}^m \Big(n_j\wedge(n_j^2\epsilon_j^2\cdot p^{-1}r^{-1})\Big)} (\lambda\sigma^2 + \sigma^4)\bigg)\bigwedge (r\lambda^2)\bigg) \bigvee \bigg(\bigg(\frac{c_0r^2\lambda^2}{\sum_{i=1}^m \Big(n_j\wedge(n_j^2\epsilon_j^2\cdot r^{-2})\Big)}\bigg) \bigwedge (r\lambda^2)\bigg).
\end{align*}
This finishes the proof of Theorem~\ref{thm:lower-bound}.

\subsection{Proof of Lemma \ref{lem:vantree}}

	We use $x = \{X^{(j)}_i, i = 1,\cdots, n_j\}_{j=1}^m$ to represent the collection of all data, $X^{(j)} = \{X^{(j)}_i, i = 1,\cdots, n_j\}$, and $\bTheta$ to be the parameter. 
	Condition on $\{\hat\U_j\}_{j=1}^m$, we define the random matrices
	\begin{align*}
		\A &= \hat\U\hat\U^\top - \psi(\bTheta), \\
		\B_{ij}  &= \sum_{(k,l)\in[p]\times[r]} \frac{\partial}{\partial \bTheta_{kl}}\big([\nabla \psi_{ij}(\bTheta)]_{kl} \cdot p(x,\bTheta|\{\hat\U_j\}_{j=1}^m) \cdot p(\bTheta)\big)\frac{1}{p(x,\bTheta|\{\hat\U_j\}_{j=1}^m) p(\bTheta)},
	\end{align*}
	where $p(x,\bTheta|\{\hat\U_j\})$ is the conditional density with parameter $\bTheta$ and we have $$p(x,\bTheta|\{\hat\U_j\}) = \prod_{j=1}^m p(X^{(j)},\bTheta|\hat\U_j).$$
	Now we define the conditional expectation $\EE [\inp{\A}{\B}|\{\hat\U_j\}] = \int \int \inp{\A}{\B}p(x,\bTheta|\{\hat\U_j\})p(\bTheta) d\bTheta dx$. Then using Cauchy-Schwarz inequality, we see 
	\begin{align}\label{CS:conditional}
		\EE[\fro{\A}^2|\{\hat\U_j\}] \geq \frac{(\EE [\inp{\A}{\B}|\{\hat\U_j\}])^2}{\EE[\fro{\B}^2|\{\hat\U_j\}]}. 
	\end{align}
	Simple calculation shows 
	\begin{align*}
		\EE[\inp{\A}{\B}|\{\hat\U_j\}] &= \int\int \sum_{ij} [\hat\U\hat\U^\top - \psi(\bTheta)]_{ij} \cdot \sum_{k,l} \frac{\partial}{\partial \bTheta_{kl}}\big([\nabla \psi_{ij}(\bTheta)]_{kl} p(x,\bTheta|\{\hat\U_j\}) p(\bTheta)\big)d\bTheta dx\\
		&= \int\int\sum_{ij,kl}[\nabla \psi_{ij}(\bTheta)]_{kl}^2\cdot p(\bTheta)d\bTheta dx\\
		&= \int \tr\big(\nabla\psi(\bTheta)\circ \nabla\psi(\bTheta)^*\big) p(\bTheta) d\bTheta,
	\end{align*}
	where the second equality holds from integration by parts, $\int p(x,\bTheta|\{\hat\U_j\})dx= 1$ and $\sum_{ij,kl}[\nabla \psi_{ij}(\bTheta)]_{kl}^2 = \tr\big(\nabla\psi(\bTheta)\circ \nabla\psi(\bTheta)^*\big)$. Meanwhile,
	\begin{align*}
		\EE[\fro{\A}^2|\{\hat\U_j\}]=\int\fro{\hat\U\hat\U^\top - \psi(\bTheta)}^2\cdot p(\bTheta) d\bTheta.
	\end{align*}
	Notice the right hand side is still a function of $\{\hat\U_j\}$. 
	Next we consider the expectation $\EE[\fro{\B}^2|\{\hat\U_j\}]$:
	\begin{align*}
		\EE[\fro{\B}^2|\{\hat\U_j\}] &= \EE \sum_{ij}\bigg(\Delta\psi_{ij}(\bTheta) + \inp{\nabla\psi_{ij}(\bTheta)}{\nabla \log p(x,\bTheta|\{\hat\U_j\})  +\nabla\log p(\bTheta)}\bigg)^2.
	\end{align*}

	Since $\psi(\bTheta)$ is independent of $\{\hat\U_j\}$, 
	\begin{align*}
		\EE \sum_{ij}\Delta\psi_{ij}^2(\bTheta) = \int\sum_{ij}\Delta\psi_{ij}^2(\bTheta)p(\bTheta)d\bTheta. 
	\end{align*}
	Notice 
	\begin{align}\label{eq:scorefunction}
		\int \nabla \log p(X^{(j)},\bTheta|\hat\U_j) \cdot p(X^{(j)},\bTheta|\hat\U_j) dx = \nabla \big[\int p(X^{(j)},\bTheta|\hat\U_j) dx\big]= 0,
	\end{align}
	where the last equality is due to $\int p(X^{(j)},\bTheta|\hat\U_j) dx = 1$.
	Thus, $\EE\nabla\log p(x,\bTheta|\{\hat\U_j\})=0$. 
	Also notice $\inp{\nabla\psi_{ij}(\bTheta)}{\nabla\log p(\bTheta)} = 0$. Therefore 
	\begin{align*}
		&\quad\EE \sum_{ij}\inp{\nabla\psi_{ij}(\bTheta)}{\nabla \log p(x,\bTheta|\{\hat\U_j\})  +\nabla\log p(\bTheta)}^2 \\
		&=\EE \sum_{ij}\inp{\nabla\psi_{ij}(\bTheta)}{\nabla \log p(x,\bTheta|\{\hat\U_j\})}^2\\
		&= \EE\sum_{ij}\inp{\nabla\psi_{ij}(\bTheta)}{\sum_{k=1}^m\nabla \log p(X^{(j)},\bTheta|\hat\U_k)}^2\\
		&= \sum_{k=1}^m\EE\sum_{ij}\inp{\nabla\psi_{ij}(\bTheta)}{\nabla \log p(X^{(j)},\bTheta|\hat\U_k)}^2,
	\end{align*} 
	where in the last line the cross terms vanish due to \eqref{eq:scorefunction}. Recall $$\calI(\bTheta|\hat\U_j)=  \EE \big[\big(\nabla\log p(X^{(j)};\bTheta|\hat\U_j)\big)^*\circ\nabla\log p(X^{(j)};\bTheta|\hat\U_j)\big].$$
	Therefore
	\begin{align*}
		\EE\sum_{ij}\inp{\nabla\psi_{ij}(\bTheta)}{\nabla \log p(X^{(j)},\bTheta|\hat\U_k)}^2 = \int \tr\big(\nabla\psi(\bTheta)^*\circ \calI(\bTheta|\hat\U_k) \circ  \nabla\psi(\bTheta)\big)\cdot  p(\bTheta)d\bTheta
	\end{align*}
	Using \eqref{eq:scorefunction} again, we obtain 
	\begin{align*}
		\EE \sum_{ij}\Delta\psi_{ij}(\bTheta)\cdot\inp{\nabla\psi_{ij}(\bTheta)}{\nabla \log p(x,\bTheta|\{\hat\U_j\})} = 0. 
	\end{align*}
	So we conclude 
	\begin{align*}
				\EE[\fro{\B}^2|\{\hat\U_j\}] = \int\sum_{ij}\Delta\psi_{ij}^2(\bTheta) p(\bTheta) d\bTheta + \sum_{j=1}^m\int \tr\big(\nabla\psi(\bTheta)\circ \calI(\bTheta|\hat\U_j) \circ  \nabla\psi(\bTheta)^*\big)\cdot  p(\bTheta)d\bTheta. 
	\end{align*}
	
	Now taking expectation w.r.t. $\hat\U_j$ in \eqref{CS:conditional}, and using Jenson's inequality yield the desired result
	\begin{align*}
		\int\EE\fro{\hat\U\hat\U^\top - \psi(\bTheta)}^2\cdot p(\bTheta) d\bTheta &\geq \EE \frac{(\EE [\inp{\A}{\B}|\{\hat\U_j\}])^2}{\EE[\fro{\B}^2|\{\hat\U_j\}]} \\
		&= \EE \frac{\bigg(\int \tr\big(\nabla\psi(\bTheta)\circ \nabla\psi(\bTheta)^*\big) p(\bTheta) d\bTheta\bigg)^2}{\EE[\fro{\B}^2|\{\hat\U_j\}]} \\
		&\geq \frac{\bigg(\int \tr\big(\nabla\psi(\bTheta)\circ \nabla\psi(\bTheta)^*\big) p(\bTheta) d\bTheta\bigg)^2}{\EE\big[\EE[\fro{\B}^2|\{\hat\U_j\}]\big]}.
	\end{align*}

\newpage
\section{Technical Lemma}
\begin{lemma}\label{lemma:weighted-chi-square}
	Let $g_1,\cdots,g_m \overset{\text{i.i.d.}}{\sim} N(0,1)$, and $c_1,\cdots,c_m\geq 0$. Then for any integer $l\geq 0$, we have 
	\begin{align*}
		\EE (\sum_{i=1}^m c_ig_i^2)^l \leq (Cl\cdot\sum_{i=1}^mc_i)^l
	\end{align*}
	for some absolute constant $C>0$. 
\end{lemma}
\begin{proof}
	We show this by expanding $(\sum_{i=1}^m c_ig_i^2)^l$. In fact, we have
	\begin{align*}
		\EE(\sum_{i=1}^m c_ig_i^2)^l = \EE\sum_{i_1,\cdots,i_l=1}^{m} c_{i_1}\cdots c_{i_l} g_{i_1}^2\cdots g_{i_l}^2\\
		\leq (\sum_i c_i)^l\cdot \EE g^{2l}\leq (Cl)^{l}\cdot (\sum_i c_i)^l,
	\end{align*}
	where $g\sim N(0,1)$ and we use the moment bound for Gaussian in the last inequality. 
\end{proof}

\begin{lemma}[\cite{koltchinskii2017concentration}]\label{lemma:covariance}
	Let $X_1,\cdots,X_n$ be i.i.d. samples from $N(0,\bSigma)$, and $\hat\bSigma = \frac{1}{n}\sum_{i=1}^nX_iX_i^\top$. Then 
	\begin{align*}
		\EE\op{\hat\bSigma - \bSigma} \asymp \op{\bSigma}\bigg(\frac{\tilde r}{n} \vee \sqrt{\frac{\tilde r}{n}}\bigg),
	\end{align*}
	where $\tilde r = \frac{\tr(\bSigma)}{\op{\bSigma}}$ is the effective rank of $\bSigma$. 
	Moreover, there exists an absolute constant $C_1>0$, such that for all $t\geq 1$, with probability exceeding $1-e^{-t}$, 
	\begin{align*}
		\bigg|\op{\hat\bSigma - \bSigma}  -\EE\op{\hat\bSigma - \bSigma}\bigg| \leq C_1 \op{\bSigma} \bigg(\frac{t}{n} \vee \sqrt{\frac{t}{n}}\bigg). 
	\end{align*}
\end{lemma}

\begin{lemma}\label{lemma:l2-norm-concentration}
	Let $X\in\RR^{d}$ be a sub-Gaussian random vector with $\EE X = 0$, and denote $\psitwo{X}$ its $\psi_2$ norm. Then we have for any $t>0$, 
	\begin{align*}
		\PP\bigg(\ltwo{X}\geq t\bigg) \leq 4^d \exp(-\frac{Ct^2}{\psitwo{X}^2})
	\end{align*}
	for some absolute constant $C>0$. 
\end{lemma}
\begin{proof}
	Let $\{x_i\}_{i=1}^N$ be an $1/2$ cover of the unit sphere $\SS^{d-1}$, then $N\leq 4^d$. Notice $\ltwo{X} = \inp{X}{\frac{X}{\ltwo{X}}}$. Then there exists some $x_0\in\{x_i\}_{i=1}^N$, such that $\ltwo{x_0 - X/\ltwo{X}}\leq 1/2$. Now 
	\begin{align*}
		\ltwo{X} = \inp{X}{\frac{X}{\ltwo{X}}} =  \inp{X}{\frac{X}{\ltwo{X}} - x_0}  + \inp{X}{x_0}\leq \frac{1}{2}\ltwo{X} + \inp{X}{x_0}. 
	\end{align*}
	This implies $\ltwo{X}\leq 2\inp{X}{x_0}$. We conclude 
	\begin{align*}
		\PP(\ltwo{X}\geq t) \leq 4^d\cdot \PP(\inp{X}{x_0}\geq t/2) \leq 4^d\exp(-\frac{Ct^2}{\psitwo{X}^2}). 
	\end{align*}
\end{proof}


\begin{lemma}\label{lemma:mixtail:moments}
	Let $X$ be random variable such that 
	\begin{align*}
		\PP(|X|\geq \max\{a+bt,\sqrt{a+bt}\}) \leq e^{-t}
	\end{align*}
	 for some $1>a>0,b>0$ and all $t>0$, then we have 
	 \begin{align*}
	 	\EE |X|^2 &\leq a+b+2ab+2b^2,\\
	 	\EE |X|^4 &\leq a^2+2ab+2b^2 + 16a^3b + 96b^4. 
	 \end{align*}
\end{lemma}
\begin{proof}
We have 
	\begin{align*}
		\EE|X|^2 = \int_{0}^{+\infty}\PP(|X|^2\geq s) ds = \int_{0}^{+\infty}\PP(|X|\geq s) 2sds.
	\end{align*}
	We then decompose the integral into three parts:
	\begin{align}\label{integral:decomposition}
		\int_0^{\infty} = \int_0^{\sqrt{a}} + \int_{\sqrt{a}}^1 + \int_1^{\infty}.
	\end{align}
	For the first part, we have
	\begin{align*}
		\int_{0}^{\sqrt{a}}\PP(|X|\geq s) 2sds\leq a.
	\end{align*}
	For the second part, we have
	\begin{align*}
		 \int_{\sqrt{a}}^1 \PP(|X|\geq s) 2sds &= \int_{0}^{\frac{1-a}{b}} \PP(|X|\geq \sqrt{a+bt}) 2(a+bt)^{1/2}\frac{1}{2}(a+bt)^{-1/2}bdt\\
		 &\leq b\int_{0}^{\frac{1-a}{b}} e^{-t} dt \leq b.
	\end{align*}
	For the third part, we have
	\begin{align*}
		\int_1^{\infty}\PP(|X|\geq s) 2sds &= 	\int_{\frac{1-a}{b}}^{\infty}\PP(|X|\geq a+bt) 2(a+bt)bdt\\
		&\leq 2b\int_{\frac{1-a}{b}}^{\infty} e^{-t}(a+bt) dt\\
		&\leq 2ab + 2b^2
	\end{align*}

For the fourth order moment, we have similarly 
\begin{align*}
	\EE|X|^4 = \int_{0}^{+\infty}\PP(|X|^4\geq s) ds = \int_{0}^{+\infty}\PP(|X|\geq s) 4s^3ds.
\end{align*}
Using the decomposition as in \eqref{integral:decomposition}, we have for the first part, 
\begin{align*}
	\int_{0}^{\sqrt{a}}\PP(|X|\geq s) 4s^3ds\leq a^2.
\end{align*}
For the second part, we have
\begin{align*}
	\int_{\sqrt{a}}^1 \PP(|X|\geq s) 4s^3ds &= \int_{0}^{\frac{1-a}{b}} \PP(|X|\geq \sqrt{a+bt}) 4(a+bt)^{3/2}\frac{1}{2}(a+bt)^{-1/2}bdt\\
	&\leq  2b\int_{0}^{\frac{1-a}{b}} e^{-t} (a+bt)dt\\
	&\leq 2ab+2b^2. 
\end{align*}
For the third part, we have
\begin{align*}
	\int_1^{\infty}\PP(|X|\geq s) 4s^3ds &= 	\int_{\frac{1-a}{b}}^{\infty}\PP(|X|\geq a+bt) 4(a+bt)^3bdt\\
	&\leq 4b\int_{\frac{1-a}{b}}^{\infty} e^{-t}(a+bt)^3 dt\\
	&\leq 16a^3b +96b^4. 
\end{align*}
\end{proof}

\section{Some Linear Algebras}\label{sec:linalg}

\subsection{Derivation for $\nabla\psi(\Theta)$}\label{sec:linalg:gradient}

Let the map $\psi:\RR^{p\times r}\rightarrow\RR^{p\times p}$ be defined as $\psi(\Theta) = \Theta(\Theta^\top\Theta)^{-1}\Theta^\top$. Then the gradient of $\psi$ evaluated at $\Theta$ is a linear map: $\nabla\psi(\Theta): \RR^{p\times r}\rightarrow\RR^{p\times p}$. We set 
\begin{align*}
	\psi_1:&\RR^{p\times r}\rightarrow\RR^{p\times r} \text{~as~} \psi_1(\Theta) = \Theta, \\
	\psi_2:&\RR^{p\times r}\rightarrow\RR^{r\times r} \text{~as~} \psi_2(\Theta) = (\Theta^\top\Theta)^{-1},\\
	\psi_3:&\RR^{p\times r}\rightarrow\RR^{r\times p} \text{~as~} \psi_1(\Theta) = \Theta^\top.
\end{align*}
Then using product rule, we have for any $Y\in\RR^{p\times r}$, 
\begin{align*}
	\nabla\psi(\Theta)(Y) &= \nabla\psi_1(\Theta)(Y)\cdot \psi_2(\Theta)\cdot\psi_3(\Theta) +  \psi_1(\Theta)\cdot \nabla\psi_2(\Theta)(Y)\cdot\psi_3(\Theta) \\
	&\quad +  \psi_1(\Theta)\cdot \psi_2(\Theta)\cdot\nabla\psi_3(\Theta)(Y).
\end{align*}
Notice here $\nabla\psi_1(\Theta):\RR^{p\times r}\rightarrow\RR^{p\times r}$ is defined as $\nabla\psi_1(\Theta)(Y) = Y$ and $\nabla\psi_3(\Theta):\RR^{p\times r}\rightarrow\RR^{r\times p}$ is defined as $\nabla\psi_3(\Theta)(Y) = Y^\top$. 
Now we compute $\nabla\psi_2(\Theta)$. Following definition of gradient, $\nabla\psi_2(\Theta):\RR^{p\times r}\rightarrow\RR^{r\times r}$. We set $\psi_{2,1}(\Theta) = \Theta^\top\Theta$, and $\psi_{2,2}(M) = M^{-1}$.Then using product rule, 
\begin{align*}
	\nabla\psi_{2,1}(\Theta)(Y) = Y^\top\Theta + \Theta^\top Y.
\end{align*}
We also define $\psi_{2,3}(M) = M$.  $\nabla\psi_{2,2}(M):\RR^{r\times r}\rightarrow\RR^{r\times r}$ can be calculated using product rule. Notice for any $N\in\RR^{r\times r}$, 
\begin{align*}
	0 = \nabla(\psi_{2,2}\cdot \psi_{2,3})(M)(N) =  \nabla\psi_{2,2}(M)(N)\cdot M + M^{-1}\cdot N,
\end{align*}
which implies $ \nabla\psi_{2,2}(M)(N) = -M^{-1}NM^{-1}$. Notice $\psi_{2}(\Theta) = \psi_{2,2}\circ\psi_{2,1}(\Theta)$. Using chain rule, we have 
\begin{align*}
	\nabla\psi_{2}(\Theta)(Y) &=\nabla(\psi_{2,2}\circ\psi_{2,1})(\Theta)(Y) = \nabla\psi_{2,2}(\psi_{2,1}(\Theta))\big(\nabla\psi_{2,1}(\Theta)(Y)\big)\\
	&= \nabla\psi_{2,2}(\Theta^\top\Theta)\big( Y^\top\Theta + \Theta^\top Y\big)\\
	&=-(\Theta^\top\Theta)^{-1}(Y^\top\Theta + \Theta^\top Y)(\Theta^\top\Theta)^{-1}. 
\end{align*}
We have 
\begin{align*}
	\psi_1(\Theta)\cdot \nabla\psi_2(\Theta)(Y)\cdot\psi_3(\Theta) = -\Theta(\Theta^\top\Theta)^{-1}(Y^\top\Theta + \Theta^\top Y)(\Theta^\top\Theta)^{-1}\Theta^\top. 
\end{align*}
In summary, we have
\begin{align*}
	\nabla\psi(\Theta)(Y) &= Y(\Theta^\top\Theta)^{-1}\Theta^\top-\Theta(\Theta^\top\Theta)^{-1}(Y^\top\Theta + \Theta^\top Y)(\Theta^\top\Theta)^{-1}\Theta^\top + \Theta(\Theta^\top\Theta)^{-1}Y^\top\\
	&= \bar\Theta_{\perp}\bar\Theta_{\perp}^\top Y(\Theta^\top\Theta)^{-1}\Theta^\top + \Theta(\Theta^\top\Theta)^{-1}Y^\top\bar\Theta_{\perp}\bar\Theta_{\perp}^\top,
\end{align*}
where $\bar\Theta_{\perp}\bar\Theta_{\perp}^\top = I_p - \Theta(\Theta^\top\Theta)^{-1}\Theta^\top$.

\subsection{Derivation for $\nabla\psi(\Theta)^*$}\label{sec:linalg:adjoint}

Once we obtain the closed-form for $\nabla\psi(\Theta)$, we can compute its adjoint $\nabla\psi(\Theta)^*:\RR^{p\times p}\rightarrow\RR^{p\times r}$. For any $Y\in\RR^{p\times r}$, $M\in\RR^{p\times p}$, 
we have
\begin{align*}
	\inp{\nabla\psi(\Theta)^*(M)}{Y} &= \inp{M}{\nabla\psi(\Theta)(Y)} \\
	&= \inp{M}{\bar\Theta_{\perp}\bar\Theta_{\perp}^\top Y(\Theta^\top\Theta)^{-1}\Theta^\top + \Theta(\Theta^\top\Theta)^{-1}Y^\top\bar\Theta_{\perp}\bar\Theta_{\perp}^\top}\\
	&=  \inp{\bar\Theta_{\perp}\bar\Theta_{\perp}^\top (M+M^\top)\Theta(\Theta^\top\Theta)^{-1}}{Y}.
\end{align*}
So we conclude $\nabla\psi(\Theta)^*(M) = \bar\Theta_{\perp}\bar\Theta_{\perp}^\top (M+M^\top)\Theta(\Theta^\top\Theta)^{-1}$.

\end{document}